\documentclass[reqno, 11pt]{amsart} 

\usepackage{bbm}
\usepackage[top=2.50cm,left=2.65cm,right=2.65cm,bottom=2.50cm]{geometry}
\usepackage{amsfonts}
\usepackage{amsmath,amsthm,amscd,mathrsfs,amssymb,graphicx,enumerate,
	dsfont}
\usepackage{xypic}
\usepackage{hyperref}
\usepackage[usenames,dvipsnames]{xcolor}
\hypersetup{colorlinks=true,citecolor=NavyBlue,linkcolor=Brown,urlcolor=Orange}

\newcounter{1}

\newtheorem{thm2}{Theorem}
\newtheorem{conj2}{Conjecture}
\newtheorem{thm}{Theorem}[section]
\newtheorem{cor}[thm]{Corollary}
\newtheorem{lem}[thm]{Lemma}
\newtheorem{prop}[thm]{Proposition}
\newtheorem{prop2}{Proposition}
\newtheorem{conj}[thm]{Conjecture}
\theoremstyle{definition}
\newtheorem{defn}[thm]{Definition}

\newtheorem{rmk}[thm]{Remark}

 \DeclareMathOperator{\Spec}{Spec}

\DeclareMathOperator{\End}{End} 
\DeclareMathOperator{\Hom}{Hom} 
 
 \DeclareMathOperator{\Sym}{Sym}

\newcommand{\C}{\ensuremath\mathds{C}}

\newcommand{\Z}{\ensuremath\mathds{Z}}
\newcommand{\Q}{\ensuremath\mathds{Q}}

\newcommand{\PP}{\ensuremath\mathds{P}}

\newcommand{\HH}{\ensuremath\mathrm{H}}
\newcommand{\NN}{\ensuremath\mathrm{N}}
\newcommand{\tr}{\ensuremath\mathrm{tr}}
\newcommand{\CH}{\ensuremath\mathrm{CH}}
\newcommand{\h}{\ensuremath\mathfrak{h}}
\newcommand{\ho}{\ensuremath\mathfrak{h}^\circ}

\AtBeginDocument{%
	\def\MR#1{}
}

\begin{document}

	\title{On the birational motive of hyper-K\"ahler varieties}
	\author{Charles Vial}
	
	\thanks{2010 {\em Mathematics Subject Classification.} 14C15, 14C25, 14J42, 14K99, 14E05}
	
	\thanks{{\em Key words and phrases.} Algebraic cycles, hyper-K\"ahler varieties, moduli spaces of sheaves on K3 surfaces, birational motives, abelian varieties, co-algebras}
	
	\address{Universit\"at Bielefeld, Germany}
	\email{vial@math.uni-bielefeld.de}
	
	
	\begin{abstract} 
	We introduce a new ascending filtration, that we call the co-radical filtration in analogy with the basic theory of co-algebras, on the Chow groups of pointed smooth projective varieties. In the case of zero-cycles on projective hyper-K\"ahler manifolds, we conjecture it agrees with a filtration introduced by Voisin. This is established for moduli spaces of stable objects on K3 surfaces, for generalized Kummer varieties and for the Fano variety of lines on a smooth cubic fourfold. Our overall strategy is to view the birational motive of a smooth projective variety as a co-algebra object with respect to the diagonal embedding and to show in the aforementioned cases the existence of a so-called strict grading whose associated filtration agrees with the filtration of Voisin.
    As results of independent interest, we upgrade to rational equivalence Voisin's notion of ``surface decomposition'' and show that the birational motive of some projective hyper-K\"ahler manifolds is determined, as a co-algebra object, by the birational motive of a surface. We  also relate our co-radical filtration on the Chow groups of abelian varieties to Beauville's eigenspace decomposition.\smallskip
    
\noindent \textsc{R\'esum\'e.} Nous d\'efinissons une filtration croissante, que nous appelons la filtration co-radicale en analogie avec la th\'eorie \'el\'ementaire des co-alg\`ebres, sur les groupes de Chow des vari\'et\'es projectives lisses point\'ees. Dans le cas des z\'ero-cycles sur les vari\'et\'es projectives hyper-K\"ahl\'eriennes, nous conjecturons qu'elle co\"incide avec une filtration introduite par Voisin. Ceci est \'etabli pour les espaces de modules d'objets stables sur les surfaces K3, pour les vari\'et\'es de Kummer g\'en\'eralis\'ees et pour la vari\'et\'e de Fano des droites sur une cubique lisse de dimension 4. Notre strat\'egie consiste \`a consid\'erer la structure naturelle de co-alg\`ebre sur le motif birationnel d'une vari\'et\'e projective lisse  et \`a montrer dans les cas susmentionn\'es l'existence d'un ``scindage strict'' dont la filtration associ\'ee co\"incide avec la filtration de Voisin.
Comme d\'eveloppements connexes,
nous formulons une version motivique de la notion cohomologique de "surface decomposition" de Voisin 
et montrons que le motif birationnel de certaines vari\'et\'es projectives hyper-K\"ahl\'eriennes est d\'etermin\'e, en tant que co-alg\`ebre, par le motif birationnel d'une surface. Nous relions \'egalement notre filtration co-radicale sur les groupes de Chow des vari\'et\'es ab\'eliennes \`a la d\'ecomposition de Beauville.
	\end{abstract}
	
	\maketitle
	
	\vspace{-10pt}
	\section*{Introduction}
	
	For a smooth projective variety $X$ over an algebraically closed field $K$, we denote $$\delta: X \hookrightarrow X\times X$$ the diagonal embedding and we fix a zero-cycle $o \in \CH_0(X)$  (necessarily of degree 1) such that $\delta_*o = o\times o$\,; for instance $o$ is the class of a closed point.
	For $\alpha\in \CH_i(X)$  an $i$-dimensional  cycle class, we denote $\alpha^{\times N} := p_1^*\alpha \cdot \cdots \cdot p_N^*\alpha \in \CH_{iN}(X)$ the $N$-th exterior power of $\alpha$. 
	Recall from~\cite{Voevodsky,Voisin-smash}  that if $\alpha$ is algebraically trivial (\emph{e.g.}, $\alpha$ is a 0-cycle of degree zero), then $\alpha$ is \emph{smash-nilpotent}, \emph{i.e.}, satisfies $\alpha^{\times N} = 0$ in the Chow group with rational coefficients $\CH_{iN}(X^N)$ for $N$ large enough. 
	The following cycle will play a prominent role in this work\,:
\begin{equation}\label{E:it-red-comult}
	\bar \delta^{n-1}:= \prod_{i=1}^{n} p_{0,i}^*(\Delta_X - X\times o)\ \in\CH_{\dim X}(X\times X^{n}).
\end{equation}
Here, $\Delta_X\in \CH_{\dim X}(X\times X)$ is the class of the diagonal and $p_{0,i} : X\times X^{n} \to X\times X$ are the projectors on the product of the first and $(i+1)$-st factors.
Such a cycle $\bar \delta^{n-1}$ already appears in \cite{VoisinGT} and is shown \cite[Cor.~1.6 \& Prop.~2.2]{VoisinGT} to vanish for $n$ large enough.
Its projection to $X^n$ is the so-called $n$-th modified diagonal cycle and, for any closed point $x\in X$, we have the relation (Proposition~\ref{P:RT})
\begin{equation}\label{E:co-rad-smash}
\bar{\delta}^{n-1}_*[x] = ([x]-o)^{\times n} \ \ \mbox{in } \CH_0(X^n).
\end{equation}
Inspired by the basic formalism of (graded) co-algebras, we call $\bar \delta:= \bar \delta^1$ the \emph{reduced co-multiplication cycle} and $\bar \delta^{n-1}$ the \emph{iterated reduced co-multiplication cycle}\,; it satisfies the formula
$$\bar \delta^n = (\bar \delta \otimes \Delta_X \otimes \cdots \otimes \Delta_X) \circ \bar \delta^{n-1}.$$
	By the above, we obtain an exhaustive ascending filtration (that depends on the choice of $o$)
		$$R_k\CH_0(X) := \ker \big(\bar{\delta}^k_*  : \CH_0(X) \to \CH_0(X^{k+1}) \big) \quad \mbox{for } k\geq0,$$
that we call
the \emph{co-radical filtration} (Definition~\ref{D:co-rad}). The aim of this work is to give evidence, in case $X$ is  a hyper-K\"ahler variety equipped with its (conjectural) Beauville--Voisin $0$-cycle, that this filtration is meaningful both from a motivic and geometric point of view.

	\subsection{On zero-cycles on hyper-K\"ahler varieties}
	By \emph{hyper-K\"ahler variety}, we will mean a  projective hyper-K\"ahler manifold, or equivalently, a projective irreducible holomorphic symplectic manifold.

	\subsubsection{Main result}---
	Given a hyper-K\"ahler variety $X$ of dimension $2n$,  the \emph{Voisin filtration} $S_k\CH_0(X)$ is defined as the subgroup spanned by classes of points supported on a closed subvariety $Z_k \subset X$ of dimension $\geq n-k$ all of whose points are rationally equivalent in~$X$ (see \S \ref{SS:VoisinFil} for more details). 
	We give a new characterization of Voisin's filtration~$S_\bullet$, introduced in~\cite{Voisin-coisotropic}, in certain cases\,:

\begin{thm2}[Theorem~\ref{thm:splitting} and Remark~\ref{R:moduli}]
	\label{T:Voisin-filt}
	Let $X$ be a hyper-K\"ahler variety and denote $2n$ its dimension. Assume that $X$ is one of the following\,:
\begin{enumerate}[(i)]
	\item \label{hilb} $\mathrm{Hilb}^n(S)$, the Hilbert scheme of length-$n$ closed subschemes on a K3 surface $S$~\cite{Beauville-Chern}\,;
	\item \label{moduli} $\mathrm{M}_\sigma(v)$, a moduli space of stable objects on a K3 surface\,;
	\item \label{kum} $K_n(A)$, the generalized Kummer variety associated to an abelian surface $A$~\cite{Beauville-Chern}\,;
	\item \label{fano} $F(Y)$, the Fano variety of lines on a smooth cubic fourfold $Y$~\cite{BeauvilleDonagi}\,;
	\setcounter{1}{\value{enumi}}
\end{enumerate}
 Then there exists a point $o\in X$ such that, for all $k\geq 0$ and for all $x\in X$, 
	\begin{equation}\label{E:equivalence}
[x]\in S_k\CH_0(X) \ \iff \ ([x]-[o])^{\times k+1} = 0 \ \mbox{in}\ \CH_0(X^{k+1}),
	\end{equation}
	or equivalently, such that 
	$$S_k\CH_0(X) = R_k\CH_0(X)$$
	for all $k\geq 0$.
\end{thm2}

 Here, moduli spaces of stable objects on a K3 surface $S$ are understood as moduli spaces of objects in the derived category of coherent sheaves on $S$ with given Mukai vector $v$ that are stable with respect to a $v$-generic Bridgeland stability condition~$\sigma$\,; see \S \ref{S:moduli}. Note  that case \eqref{hilb} is a special case of~\eqref{moduli}\,; it will however be convenient for our exposition to distinguish them.
 
Since the first three families of hyper-K\"ahler varieties \eqref{hilb}-\eqref{moduli}-\eqref{kum} are dense in moduli~\cite{MM-dense}, and since the latter \eqref{fano} forms a locally complete family in moduli, we are led to formulate\,:

\begin{conj2}\label{conj2:Voisin-split-cogen}
	Let $X$ be a hyper-K\"ahler variety. Then there exists a point $o\in X$ such that
$$[x]\in S_k\CH_0(X) \ \iff \ ([x]-[o])^{\times k+1} = 0 \ \mbox{in}\ \CH_0(X^{k+1}),$$
 for all $k\geq 0$.
\end{conj2}

\subsubsection{Some consequences}---
In the special case $k=0$, Conjecture~\ref{conj2:Voisin-split-cogen} matches the expectation that $S_0\CH_0(X)$ should be spanned by the class of a point $o\in X$. When this is the case we call $o$ a \emph{Beauville--Voisin point} of $X$ and its class $[o] \in \CH_0(X)$ the Beauville--Voisin class. The existence of such a Beauville--Voisin point was already established in cases  \eqref{hilb}, \eqref{moduli}, \eqref{kum}, \eqref{fano}\,; see \cite{Voisin-coisotropic}, and \cite{MZ, SYZ, Li-Zhang} for case \eqref{moduli}.

In the special case $k=1$ and $X=F(Y)$ is the Fano variety of lines on a smooth cubic fourfold, we get thanks to the description \cite[Thm.~1.5]{CMP} of $S_1\CH_0(F(Y))$ that a point $x\in F(Y)$ satisfies $([x]-[o])^{\times 2} = 0$ in $\CH_0(F(Y)^2)$ if and only if $x$ is rationally equivalent to a 0-cycle supported on a uniruled divisor.

In the special case $k=n+1$, since by definition $S_{n+1}\CH_0(X) = 0$, Conjecture~\ref{conj2:Voisin-split-cogen} predicts that $([x]-[o])^{\times n+1} = 0 \ \mbox{in}\ \CH_0(X^{n+1})$ for all points $x\in X$. 
This bound is optimal in the sense that if $X$ is a hyper-K\"ahler variety of dimension $2n$, then for all points $o\in X$ there exists a point $x\in X$ such that $([x]-[o])^{\times n} \neq 0 \ \mbox{in}\ \CH_0(X^{n})$.
Observe indeed that, if $\sigma$ is a non-zero symplectic form on $X$, we have 
$(\Delta_X - X\times [o])^*\sigma =  \sigma$, so that $(\bar \delta^{n-1})^*(\sigma\otimes \cdots \otimes \sigma) = \sigma^n \neq 0$ in $\HH^{0}(X,\Omega^{2n}_X)$.
 If now $\bar{\delta}^{n-1}_*[x] = ([x]-[o])^{\times n} =0$ for all $x$, then by Bloch--Srinivas~\cite{BS} $\bar{\delta}^{n-1}$ is supported on $D\times X^n$ for some divisor $D$ in $X$ and so $\sigma^n = 0$, which is a contradiction. 
 
Unconditionally, we have\,:

\begin{prop2}
	\label{P:smash-hyperK}
	Let $X$ be a hyper-K\"ahler variety and denote $2n$ its dimension. Assume that $X$ is one of \eqref{hilb}, \eqref{moduli}, \eqref{kum}, \eqref{fano} or one of the following\,:
	\begin{enumerate}[(i)]
		\setcounter{enumi}{\value{1}}
		\item \label{llsvs} $Z(Y)$, a LLSvS eightfold~\cite{LLSS}\,;
		\item \label{lagrangian} a Lagrangian fibration.
	\end{enumerate}
	If  $[o]\in \CH_0(X)$ denotes the Beauville--Voisin class, then we have 
	$$([x]-[o])^{\times n+1} = 0 \ \mbox{in}\ \CH_0(X^{n+1}),$$
	for all points $x\in X$. In particular, for any points $p_i \in X$ and any $a_i \in \Q$, $i=1,\ldots, l$, we have $$\Big(\sum_{i=1}^l a_i([p_i]-[o])\Big)^{ln+1} = 0 \ \mbox{in}\ \CH_0(X^{ln+1})$$
\end{prop2}

Proposition~\ref{P:smash-hyperK} in cases \eqref{hilb}, \eqref{moduli}, \eqref{kum}, \eqref{fano} is a special instance of Theorem~\ref{T:Voisin-filt}.
In \S \ref{SS:proof1} we provide a proof in those cases, as well as in case~\eqref{llsvs}, by showing more generally that if
 the birational motive of $X$ admits a \emph{unital grading} (see \S \ref{SSS:gradings} below), then $X$ satisfies the conclusion of the proposition (this does not involve the Voisin filtration).
From a more geometric perspective, Proposition~\ref{P:smash-hyperK} holds for any hyper-K\"ahler variety $X$ of dimension $2n$ with a point $o$, that is generically covered by $n$-dimensional abelian varieties all supporting a point rationally equivalent to $o$ in $X$, 
\emph{i.e.}, for hyper-K\"ahler varieties $X$ with a point $o$, with the property that there exists a non-empty open subset $V\subseteq X$ such that for all $x\in V$ there exists an abelian variety $A$ of dimension $n$, a non-empty open subset $U\subseteq A$ and a map $f: U \to X$ with $f(U)$ containing both $x$ and a  point rationally equivalent to $o$ in $X$
(this property is known to hold for hyper-K\"ahler varieties of type \eqref{hilb}, \eqref{fano}, \eqref{llsvs}, and \eqref{lagrangian} by Lin~\cite{Lin-lag}). 
 Indeed, if $x$ is a point on $X$, then $x$ is rationally equivalent to a 0-cycle supported on $V$ and this follows from the observation,  coming from \eqref{E:co-rad-smash}, that
\begin{equation}\label{E:multiple}
[x] = \sum_ia_i[x_i]  \in \CH_0(X) \ \Longrightarrow \ ([x]-[o])^{\times k} = \sum_i a_i ([x_i]-[o])^{\times k} \in \CH_0(X^k)
\end{equation}
and from the following analogous result for abelian varieties\,:

\begin{prop2}[Theorem~\ref{thm:abelian}(d)]\label{P:smash-abelian}
	Let $A$ be an abelian variety of dimension $g$ over a field $K$.
	Then  the following exterior power vanishes
	$$([x]-[0])^{\times g+1} = 0 \ \mbox{in}\ \CH_0(A^{g+1}),$$
	for all $K$-points $x\in A(K)$.
\end{prop2}

Since we could not find a reference for Proposition~\ref{P:smash-abelian}, we provide a proof in \S \ref{SS:abelian}
using the fact, obtained by dualizing K\"unneman's Theorem~\ref{T:K} below,  that the (covariant) Chow motive of $X$ admits a unital grading with unit $[0]$.

\subsubsection{The case \eqref{moduli} of moduli of stable objects on K3 surfaces}\label{SSS:moduli}---
Let $S$ be a K3 surface and let $\operatorname{M}_\sigma(v)$ be a moduli space of stable objects on $S$. Inspired by seminal work of O'Grady~\cite{OG-moduli}, 
 Shen--Yin--Zhao~\cite{SYZ} introduced the following ascending filtration on $\CH_0(\operatorname{M}_\sigma(v))$\,:
\begin{equation}\label{E:SYZ-def}
S^{\mathrm{SYZ}}_k\CH_0(\operatorname{M}_\sigma(v)) := \langle \, [\mathcal E]  ~\big\vert~ \mathcal E \in \operatorname{M}_\sigma(v) \mbox{ such that } c_2(\mathcal E) \in S_k^{\mathrm{OG}}(S) \, \rangle,
\end{equation}
where $S_k^{\mathrm{OG}}(S):= \{[x_1]+\cdots+ [x_k]+ \Z [o_S]\in \CH_0(S) ~\big\vert~ x_i\in S \}$ is O'Grady's ascending filtration~\cite{OG-moduli} on $\CH_0(S)$ and where $o_S$ is the Beauville--Voisin point on $S$~\cite{BV}. As noted in \cite[p.~182]{SYZ}, the main result of \cite{MZ} provides a degree-1 zero-cycle $o\in \CH_0(\operatorname{M}_\sigma(v))$ such that $S^{\mathrm{SYZ}}_0\CH_0(\operatorname{M}_\sigma(v))$ is spanned by $o$.

Theorem~\ref{T:Voisin-filt} in case \eqref{moduli} is obtained by combining the recent result of Li--Zhang~\cite[Thm.~1.1]{Li-Zhang} establishing $S^{\mathrm{SYZ}}_\bullet\CH_0(\operatorname{M}_\sigma(v)) = S_\bullet \CH_0(\operatorname{M}_\sigma(v))$ with the following\,:

\begin{thm2}[Theorem~\ref{thm:splitting-moduli}]
	\label{thm2:splitting-moduli}
	Given an element $\mathcal E$ of $\operatorname{M}_\sigma(v)$ we have
\begin{align}
[\mathcal E]\in S^{\mathrm{SYZ}}_k\CH_0(\operatorname{M}_\sigma(v)) & \iff  ([\mathcal E]-[o])^{\times k+1} = 0 \ \mbox{in}\ \CH_0(\operatorname{M}_\sigma(v)^{k+1}),    \label{E:equivalence-SYZ}  
\end{align}
or, equivalently,
$$S^{\mathrm{SYZ}}_k\CH_0(X) = R_k\CH_0(X)$$
for all $k\geq 0$.
\end{thm2}
By \cite{BFMS} we further have the equivalence  $([\mathcal E]-[o])^{\times k+1} = 0 \iff  (c_2(\mathcal E)- \lambda[o_S])^{\times k+1} = 0 $, where $\lambda:= \deg(c_2(\mathcal E) )$.
As is noted in \cite[\S 5]{OG-moduli}, if $\mathcal E \in \operatorname{M}_\sigma(v)$ is such that $c_2(\mathcal E) \in S_k^{\mathrm{OG}}(S)$, then $c_2(\mathcal E) - \lambda[o_S]$ is represented by a $0$-cycle of degree 0 on a curve of geometric genus $k$ and hence $(c_2(\mathcal E)- \lambda[o_S])^{\times k+1} = 0 \ \mbox{in}\ \CH_0(S^{k+1})$ by the nilpotence result of Voevodsky--Voisin~\cite{Voevodsky, Voisin-smash}. 
Alternately, this can be deduced from Proposition~\ref{P:smash-hyperK} for $S$\,; indeed, if $c_2(\mathcal E) - \lambda [o_S] = [x_1] +\cdots + [x_k] - k[o_S]$, then one deduces directly  from $([x_i] - [o_S])^{\times 2}=0$ that  $(c_2(\mathcal E)- \lambda[o_S])^{\times k+1} = 0$.
One concludes from~\eqref{E:multiple} that $[\mathcal E]\in S^{\mathrm{SYZ}}_k\CH_0(\operatorname{M}_\sigma(v)) $ implies $ (c_2(\mathcal E)- \lambda[o_S])^{\times k+1} = 0 $. 
Our main contribution in case \eqref{moduli} is thus the implication $([\mathcal E]-[o])^{\times k+1} = 0
\Longrightarrow [\mathcal E]\in S^{\mathrm{SYZ}}_k\CH_0(\operatorname{M}_\sigma(v)) $\,; this is proved in Theorem~\ref{thm:splitting-moduli}.
Theorem~\ref{thm2:splitting-moduli} leaves however open the question \cite[Ques.~3.2]{SYZ} of whether $[\mathcal E]\in S^{\mathrm{SYZ}}_k\CH_0(\operatorname{M}_\sigma(v))  \Longrightarrow c_2(\mathcal E) \in S_k^{\mathrm{OG}}(S)$.

\subsubsection{The strategy of proof}---
We start by outlining the geometric content of the proof of Theorem~\ref{T:Voisin-filt} in case \eqref{fano}. We write $F$ for the Fano variety of lines on the smooth cubic fourfold $Y$.
 Theorem~\ref{T:Voisin-filt} in case $k=0$ is equivalent to the fact that $S_0\CH_0(F)$ is spanned by the Beauville--Voisin class~$[o]$\,; this is due to Voisin. 
  Denote now $\varphi : F \dashrightarrow F$ Voisin's rational map \cite{Voisin-Intrinsic} which is defined as follows\,: for a general line $l$ in $Y$ there is a unique plane $\Pi_l$ in $\PP^5$ tangent to $Y$ along $l$ and not contained in $Y$ and one sets $\varphi(l) = l'$ with $l'$ the line such that $Y\cap \Pi_l = 2l +l'$. 
  The Voisin filtration on $\CH_0(F)$ admits a splitting $$\CH_0(F) =  \CH_0(F)_{(0)} \oplus \CH_0(F)_{(1)} \oplus \CH_0(F)_{(2)}, \
\emph{i.e.}, S_k\CH_0(F) = \oplus_{i\leq k} \CH_0(F)_{(i)},$$
 with $\CH_0(F)_{(i)}$ the eigenspace for the action of $\varphi_*$ with eigenvalue $(-2)^i$\,; see \cite[\S 4.2]{Voisin-coisotropic} and \cite[Thm.~21.9]{SV}. Since we are working over an algebraically closed field, the exterior product map $ \CH_0(F) \otimes \CH_0(F) \longrightarrow \CH_0(F\times F) $ is surjective (see~\S \ref{SS:0cycle}).
Denoting $\delta : F \to F\times F$ the diagonal embedding, the basic observation is that, since  $\delta\circ \varphi = (\varphi \times \varphi)\circ \delta$ and since $(\varphi \times \varphi)_*$ acts as multiplication by $(-2)^{i+j}$ on the image of $\CH_0(F)_{(i)} \otimes \CH_0(F)_{(j)} \to \CH_0(F\times F)$, we have 
\begin{equation}\label{E:formalization}
\delta_*\CH_0(F)_{(k)} \subseteq \operatorname{im}\Big(\bigoplus_{i+j=k} \CH_0(F)_{(i)} \otimes \CH_0(F)_{(j)} \longrightarrow \CH_0(F\times F) \Big).
\end{equation}
Using the fact that $\CH_0(F)_{(0)}$ is spanned by the Beauville--Voisin class $[o]$, we get by projecting on both factors
\begin{equation}\label{E:comult-CH}
\bar\delta_*\CH_0(F)_{(k)} \subseteq \operatorname{im}\Big(\bigoplus_{i+j=k, i,j>0} \CH_0(F)_{(i)} \otimes \CH_0(F)_{(j)} \longrightarrow \CH_0(F\times F) \Big).
\end{equation}

Let us now investigate the case $k=1$. Let $x\in F$ be a point such that $[x]\in S_1\CH_0(F)$. Then $[x]-[o]$ lies in $\CH_0(F)_{(1)}$. From \eqref{E:comult-CH} we immediately get that $\bar \delta^1_*[x] = ([x]-[o])^{\times 2} = 0$ and hence the inclusion $S_1\CH_0(F) \subseteq \ker (\bar \delta^1_* : \CH_0(F) \to \CH_0(F\times F))$.
For the converse inclusion, let $x$ be a point on $F$ such that $[x] \notin S_1\CH_0(F)$. We write $[x] = [x]_{(0)} + [x]_{(1)} + [x]_{(2)}$ for the eigenspace decomposition of $[x]$. By assumption, we have $[x]_{(2)}\neq 0$ and we want to show that $\bar \delta^1_*[x] = ([x]-[o])^{\times 2} \neq 0$. Since $\bar \delta^1_*([x]_{(i)}) = 0$ for $i=0,1$ by the above, we have to show that  $\bar \delta^1_*([x]_{(2)}) \neq 0$. For that purpose, it suffices to show that $\bar\delta^1_*$ is injective when restricted to $\CH_0(F)_{(2)}$. This is established in the proof of Theorem~\ref{thm:cogeneration} in case \eqref{fano}, by showing the existence of a correspondence $\mu \in \CH^4(F^2\times F)$ such that $\mu \circ \bar{\delta}^1$ acts by multiplication by $2$ on $\CH_0(F)_{(4)}$.

Finally in case $k=2$,  $\bar \delta^1_*(\CH_0(F)_{(2)})$  lies in the image of $\CH_0(F)_{(1)} \otimes \CH_0(F)_{(1)} \to \CH_0(F\times F)$, due to \eqref{E:comult-CH}. Since $\bar \delta^1$ acts as zero on both $\CH_0(F)_{(0)}$ and  $\CH_0(F)_{(1)}$, we see that $\bar \delta^2 := (\bar \delta^1 \otimes \operatorname{id}) \circ \bar \delta^1$ acts as zero on $\CH_0(F)$. 
\medskip

The proofs 
of Theorem~\ref{T:Voisin-filt} and~\ref{thm2:splitting-moduli} essentially follow the above argument\,: first we construct a splitting  $\CH_0(X) = \bigoplus_k \CH_0(X)_{(k)}$ of the filtration of Voisin (resp.\ Shen--Yin--Zhao), with $\CH_0(X)_{(0)}$ spanned by the Beauville--Voisin point,  that satisfies \eqref{E:comult-CH}. This already yields the inclusion $S_\bullet \subseteq R_\bullet$ (resp.\ $S^{\mathrm{SYZ}}_\bullet \subseteq R_\bullet$). The reverse inclusion is then established by showing that  $\bar \delta^{k-1}_*$ is injective when restricted to $\CH_0(F)_{(k)}$. 

These arguments have natural generalizations that are better expressed in the language of birational motives\,: 
in each of the cases \eqref{hilb}, \eqref{moduli}, \eqref{kum} and \eqref{fano}, we show that the Voisin (resp.\ Shen--Yin--Zhao) filtration admits a ``motivic'' splitting in the sense that it is induced by a grading $\ho(X) = \bigoplus_{0\leq k \leq n} \ho(X)_{(k)}$ on the birational motive  of $X$. 
If this grading is a \emph{unital grading}, \emph{i.e.}, if it is compatible with the diagonal embedding and is such that $\ho(X)_{(0)}$ is the unit motive, then the induced grading  $\CH_0(X) = \bigoplus_k \CH_0(X)_{(k)}$ satisfies \eqref{E:comult-CH}. 
If in addition the grading is a \emph{strict grading}, \emph{i.e.}, if $\bar \delta^{k-1}: \ho(X)_{(k)} \to (\ho(X)_{(1)})^{\otimes k}$ is split injective for all $k$, then $\bar \delta^{k-1}_*$ is injective when restricted to $\CH_0(X)_{(k)}$, and we can conclude that $S_\bullet = R_\bullet$.

The structure of the proofs of Theorem~\ref{T:Voisin-filt} and~\ref{thm2:splitting-moduli} are then as follows\,:
we start by constructing in Theorem~\ref{thm:BMCK} unital gradings in cases \eqref{hilb}, \eqref{moduli}, \eqref{kum} and \eqref{fano}, then show that these are strict gradings in Theorem~\ref{thm:cogeneration}, and finally check in the proof of Theorem~\ref{thm:splitting} (resp.\ Theorem~\ref{thm:splitting-moduli}) that the induced gradings on $\CH_0$ provide splittings to the filtration of Voisin (resp.\ Shen--Yin--Zhao).

\subsection{Birational motives of hyper-K\"ahler varieties and the co-radical filtration}

The notions of \emph{unital grading} and of \emph{strict grading} mentioned above are related to the natural co-algebra structure on the birational motive of $X$.
We now proceed to explain 
how the strict grading on the birational motive of~$X$ and the co-radical filtration on $\CH_0(X)$ come into play.

\subsubsection{Birational motives as co-algebra objects}---
	If $X$ denotes a smooth projective variety over a field~$K$, 
	the diagonal embedding $\delta:X\hookrightarrow X\times_K X$ and the structure map $\epsilon: X \to \operatorname{Spec}K$ formally satisfy
	\begin{itemize}
	\item the co-unital law	$(\mathrm{id}\times
	\epsilon)\circ \delta=\mathrm{id}=(\epsilon\times \mathrm{id})\circ \delta : X\to X $\,;
	\item the co-associative law  $(\delta\times \mathrm{id})\circ\delta= (\mathrm{id}\times \delta)\circ \delta : X\to X \times X \times X$\,;
	\item  the co-commutativity law $\delta=\tau\circ \delta : X \to X\times X$, where $\tau : X\times X \to X\times X$ is the morphism permuting the two factors.
	\end{itemize}
	The contravariant action of $\delta$ and $\epsilon$, together with the $\otimes$-structure on the category of Chow motives, endows then the Chow motive of $X$ with the structure of a commutative algebra object. A famous result of K\"unnemann~\cite{K-abelian}, reviewed in Appendix~\ref{SS:abelian}, is the following
	\begin{thm2}[K\"unnemann] \label{T:K}
		Let $A$ be an abelian variety over a field $K$. Then the Chow motive $\h(A)$ (with rational coefficients) of $A$ admits a canonical direct summand $\h^1(A)$ and the induced morphism
		$$\operatorname{Sym}^* \h^1(A) \stackrel{\sim}{\longrightarrow} \h(A)$$ is an isomorphism of algebra objects. 
	\end{thm2}
Such an isomorphism endows naturally the algebra object $\h(A)$ with a grading.
 Of course, such an isomorphism is a lift to rational equivalence of the well-known fact that the cohomology algebra of an abelian variety is isomorphic to the symmetric power of its degree-1 cohomology group. (Note that due to the fact that cup-product is graded-commutative, $\operatorname{Sym}^{2\dim A+1}\HH^1(A) = 0$).
 \medskip
 
 Let now $X$ be a hyper-K\"ahler variety of dimension $2n$. A result of Bogomolov~\cite{Bogomolov} shows that the natural map $\operatorname{Sym}^{\leq n}\HH^2(X,\Q) \hookrightarrow \HH^*(X,\Q)$ is injective. In addition, since $\HH^{2k,0}(X) = \C \sigma^k$ for all $0\leq k \leq n$ for some nowhere degenerate 2-form $\sigma$, the above natural map restricts to an isomorphism $\operatorname{Sym}^{\leq n}\HH^{2,0}(X) \stackrel{\sim}{\longrightarrow} \HH^{2*,0}(X)$ on the ``birational part'' of the cohomology of~$X$. Based on Beauville's splitting principle~\cite{Bauville-splitting}, which roughly draws parallels between the intersection theory on abelian varieties and that on hyper-K\"ahler varieties, we may ask whether the ``birational part'' of the Chow motive of $X$ is generated in degree 2. 
 The correct framework for such a question is Kahn and Sujatha's pseudo-abelian $\otimes$-category of birational motives~\cite{KS},  the definition and basic properties of which are recalled in \S \ref{S:birat}.  This is naturally a \emph{covariant} theory, and in that setting the diagonal embedding $\delta$ together with the structure map $\epsilon$ naturally endow the birational motive $\ho(X)$ of $X$ with the structure of a co-commutative co-algebra structure.
 We may then phrase
 \begin{conj2}	\label{conj2:MBCK-bis}
Let $X$ be a hyper-K\"ahler variety of dimension $2n$. Then  the birational motive $\ho(X)$ (with rational coefficients) admits a direct summand $\ho(X)_{(1)}$, called the \emph{primitive part}, such that the co-induced morphism 
\begin{equation}\label{E:iso-coalgebra}
\ho(X) \stackrel{\sim}{\longrightarrow} \operatorname{Sym}^{\leq n} \ho(X)_{(1)}
\end{equation}
 is an isomorphism of co-algebra objects.
 \end{conj2}
Such an isomorphism naturally endows the co-algebra object $\ho(X)$ with a grading (see \S \ref{SS:grading}). In Proposition~\ref{P:strictgradingBCK}, we show that this grading is cohomologically meaningful in the sense that the transcendental cohomology of the direct summand $\ho(X)_{(1)}$ coincides with the transcendental cohomology of $X$ of degree 2. Moreover, we would in fact expect a direct summand $\ho(X)_{(1)}$, with the property that \eqref{E:iso-coalgebra} is an isomorphism, to be unique.
 Conjecture~\ref{conj2:MBCK-bis} is substantiated by the following\,:

 \begin{thm2}[Theorem~\ref{thm:cogeneration}]\label{thm:main2}
 	Let $X$ be one of the hyper-K\"ahler varieties \eqref{hilb}, \eqref{moduli}, \eqref{kum} or \eqref{fano}.
 	Then $X$ satisfies the conclusion of Conjecture~\ref{conj2:MBCK-bis}. In other words, denoting $2n=\dim X$,  we have a co-algebra grading
 	$$\ho(X) = \ho(X)_{(0)} \oplus \cdots \oplus \ho(X)_{(n)}$$
 	such that
 	the natural graded morphism $$\ho(X) \stackrel{\sim}{\longrightarrow} \operatorname{Sym}^{\leq n} \ho(X)_{(1)}$$ is an isomorphism of graded co-algebra objects.
 \end{thm2}

 Let us mention that we establish Theorem~\ref{thm:main2} in case \eqref{moduli} by showing, as a result of independent interest in Theorem~\ref{thm:moduli}, that
the birational motive of a moduli space $\operatorname{M}_\sigma(v)$ of stable objects on a K3 surface~$S$ is isomorphic as co-algebra object to the birational motive of $\operatorname{Hilb}^n(S)$, where $2n= \dim \operatorname{M}_\sigma(v)$.

\subsubsection{Unital gradings and strict gradings on birational motives}\label{SSS:gradings}---
Let $X$ be a smooth projective variety over $K$. Let us explain how the existence of a co-algebra isomorphism as in Theorem~\ref{thm:main2} $$\ho(X) \stackrel{\sim}{\longrightarrow} \operatorname{Sym}^{\leq n} \ho(X)_{(1)}$$ has consequences for zero-cycles on $X$ as in  Theorems~\ref{T:Voisin-filt} and~\ref{thm2:splitting-moduli}. 
We say that a grading
	\begin{equation}\label{E:unital-grading}
	\ho(X) = \ho(X)_{(0)} \oplus \cdots \oplus \ho(X)_{(n)}
	\end{equation}
 is a \emph{co-algebra grading} if the restriction of the co-multiplication $\delta: \ho(X) \to \ho(X)\otimes \ho(X)$ to the summand $\ho(X)_{(k)}$ factors through $\bigoplus_{i+j=k} \ho(X)_{(i)}\otimes \ho(X)_{(j)}$. In addition the grading is said to be\,:
 \begin{itemize}
\item \emph{unital} if  the restriction $\epsilon_i : \ho(X)_{(i)}\to \mathds 1$ of the degree map is an isomorphism if $i=0$ and zero otherwise (see \S \ref{SS:grading}, Definition~\ref{D:MBCK} and Proposition~\ref{P:strictgradingBCK})\,;
\item \emph{strict}  if in addition the natural graded morphism $\ho(X) \to \operatorname{Sym}^{\leq n} \ho(X)_{(1)}$ is split injective, or in other words if $\ho(X)$ is co-generated by $\ho(X)_{(1)}$ (see \S \ref{SS:strictgrading}).
 \end{itemize}
In particular, if a hyper-K\"ahler variety $X$ fulfills the conclusion of Conjecture~\ref{conj2:MBCK-bis} then $\ho(X)$ has a strict grading.
If \eqref{E:unital-grading} defines a unital grading on the birational motive of a smooth projective variety $X$, we denote $o: \mathds 1 \to \ho(X)$ the morphism $\epsilon_0^{-1} : \mathds 1 \to \ho(X)_{(0)} \hookrightarrow \ho(X)$\,; 
this defines  a degree-1 zero-cycle  $o\in  \CH_0(X)$ such that $\delta_*o = o\times o \ \mbox{in}\ \CH_0(X\times_KX)$, \emph{i.e.}, $o: \mathds 1 \to \ho(X)$ is a unit in the co-algebra sense (see \S \ref{SS:grading}). Note that the class of any $K$-point on $X$ defines a unit in $\CH_0(X)$. We refer to Section~\ref{S:co-alg} for a review of co-algebra objects in an abstract $\otimes$-category.

\subsubsection{The co-radical filtration on zero-cycles}---
Let us now fix a smooth projective variety $X$ over $K$ equipped with a \emph{unit}  $o\in  \CH_0(X)$. 
In analogy to the elementary theory of co-algebras (as exposed for instance in \cite{sweedler}), we define (\S \ref{SS:co-rad} and Definition~\ref{D:co-rad}) the \emph{co-radical filtration} $R_\bullet\CH_0(X)$ associated to the unit $o$ as\,:
	$$R_k\CH_0(X) := \ker \big(\bar{\delta}^k_*  : \CH_0(X) \to \CH_0(X^{k+1}) \big) \quad \mbox{for } k\geq0,$$
where
$\bar \delta^k := \bar p^{\otimes k+1}\circ \delta^k$ with $\delta^k$ the diagonal embedding $X\hookrightarrow X^{k+1}$ and $\bar \delta^0 := \bar p := \mathrm{id} - X\times o$ the projector with kernel $o$.
Of course, this cycle $\bar \delta^k$ agrees with the one defined in~\eqref{E:it-red-comult}\,; 
the morphism $\bar \delta := \bar{\delta}^1$ is called the \emph{reduced co-multiplication}, while $\bar \delta^k$ is called the \emph{$k$-th iterated reduced co-multiplication}.
We note that, even when $K$ is algebraically closed, $R_k\CH_0(X)$ need not be generated by classes of points\,; see Remark~\ref{rmk:co-radical-abelian}. The equivalences~\eqref{E:equivalence} and~\eqref{E:equivalence-SYZ} however show that this is the case for hyper-K\"ahler varieties as in \eqref{hilb}, \eqref{moduli}, \eqref{kum} and \eqref{fano}.

\subsubsection{The co-radical filtration and the conjectural Bloch--Beilinson filtration}---
Assume now that the birational motive of $X$, considered as a co-algebra object, admits a unital  grading as in~\eqref{E:unital-grading} and let
$$G_k \CH_0(X) := \CH_0\big(\ho(X)_{(0)}\oplus \cdots \oplus \ho(X)_{(k)}\big) =\CH_0\big(\ho(X)_{(0)}\big) \oplus \cdots \oplus \CH_0\big(\ho(X)_{(k)}\big)$$
be the associated ascending filtration. As explained in \S\ref{R:oppositeBB}, the filtration $G_\bullet$ is opposite to the conjectural Bloch--Beilinson filtration $F^\bullet$ in the sense that 
$$G_k\CH_0(X)\cap F^{2k}\CH_0(X) = \CH_0(\ho(X)_{(k)}).$$
The following formal result, which is a direct translation of Proposition~\ref{P:crucial} to the setting of birational motives and whose proof is given in \S \ref{SS:proofTco-rad},
 provides the link between the ascending filtration associated to a unital, resp.~strict, grading and the co-radical filtration associated to the unit $o$\,:

\begin{prop2}\label{P:co-rad}
Assume $\ho(X)$ admits a unital grading as in~\eqref{E:unital-grading}. Then
$$G_k \CH_0(X) \subseteq R_k\CH_0(X).$$ 
In particular, $\CH_0(X) =  R_n\CH_0(X)$ and hence $([x]-o)^{\times n+1} = 0 \ \mbox{in}\ \CH_0(X^{n+1})$ for all $x\in X$.
Moreover,  if the unital grading~\eqref{E:unital-grading} is strict, then 
$$G_k \CH_0(X) = R_k\CH_0(X).$$ 
\end{prop2}

As a consequence, we see that if $X$ is a hyper-K\"ahler variety whose birational motive admits a strict grading (\emph{e.g.}, if Conjecture~\ref{conj2:MBCK-bis} holds for $X$),
 then the co-radical filtration is opposite to the conjectural Bloch--Beilinson filtration  $F^{2\bullet} \CH_0(X)$ and we have $$R_k\CH_0(X)\cap F^{2k}\CH_0(X) = \CH_0(\ho(X)_{(k)}).$$
Moreover, Proposition~\ref{P:co-rad} 
reduces the proof of Theorem~\ref{T:Voisin-filt} (resp.\ Theorem~\ref{thm2:splitting-moduli}) to showing that the Voisin filtration $S_\bullet$ (resp.\ the filtration $S_\bullet^{\mathrm{SYZ}}$) coincides with the filtration $G_\bullet$ induced by a strict grading.

\subsection{Further results and remarks} We describe some motivation for this work, as well as related results. These are not used in the proofs of Theorems~\ref{T:Voisin-filt}, \ref{thm2:splitting-moduli} and~\ref{thm:main2}.

\subsubsection{Splitting of filtrations on the Chow group of zero-cycles}
Beauville's splitting principle~\cite{Bauville-splitting} asserts that the conjectural Bloch--Beilinson filtration on the Chow ring of a hyper-K\"ahler variety~$X$ admits a splitting compatible with the ring structure, \emph{i.e.} compatible with pull-back along the diagonal embedding $\delta: X \hookrightarrow X\times X$.
Another motivation for this work was to make sense of what it means for a filtration on the Chow group of zero-cycles to admit a  splitting compatible with the diagonal embedding. 
As explained in \S \ref{SS:0cycle}, 
the diagonal embedding map does not provide a co-algebra structure on $\CH_0(X)$ (which is the reason we work with birational motives), but rather a map
$$\delta_* : \CH_0(X) \longrightarrow \CH_0(X\times_K X).$$
As is explained above, assuming Conjecture~\ref{conj2:MBCK-bis}, the co-radical filtration associated to a strict grading on the birational motive of a hyper-K\"ahler variety is opposite to the conjectural Bloch--Beilinson filtration and provides a splitting.
We would like to spell out explicitly, by avoiding the use of birational motives, or any mention of the conjectural Bloch--Beilinson filtration, what it means for this splitting to be compatible with the diagonal embedding.
For that purpose we introduce in 
Appendix~\ref{S:delta}
the notion of \emph{$\delta$-filtration} and we conjecture, with evidence provided by \eqref{hilb}, \eqref{moduli}, \eqref{kum} and \eqref{fano}, that, for every hyper-K\"ahler variety, there exists a unit $o\in \CH_0(X)$ such that the associated co-radical filtration is a split $\delta$-filtration.
In addition, in Conjecture~\ref{conj:Voisin-split-cogen}, based on the evidence provided by Theorem~\ref{T:Voisin-filt}, we conjecture that the Voisin filtration coincides with the co-radical filtration, and should hence also conjecturally be a split $\delta$-filtration\,; see Conjecture~\ref{conj:Voisin-split-cogen2}.

\subsubsection{The co-radical filtration for positive-dimensional cycles}
The co-radical filtration can also be naturally defined for positive-dimensional cycles on a smooth projective variety $X$ equipped with a unit $o\in \CH_0(X)$\,; see Definition~\ref{D:corad-positive}.
In \S\ref{SS:abelian} of Appendix~\ref{S:app}, we consider the case of an abelian variety~$A$, and we show in Theorem~\ref{thm:abelian} that the co-radical filtration defines a ring filtration on $\CH^*(A)$ that is opposite to the candidate Bloch--Beilinson filtration of Beauville~\eqref{E:BB-abelian},
 thereby establishing in particular Proposition~\ref{P:smash-abelian}. The key point is that the (contravariant) Chow motive of $A$ is generated in degree~1 (Theorem~\ref{T:K}). On the other hand, since the cohomology algebra of a hyper-K\"ahler variety is not generated in degree~2 in general, one does \emph{not} expect the co-radical filtration to be opposite to the Bloch--Beilinson filtration for positive-dimensional cycles on hyper-K\"ahler varieties in general\,; see Remark~\ref{R:MCK-corad}.

\subsubsection{Relation to work of Barros--Flapan--Marian--Silversmith}\label{SS:BFMS}
In independent work concerned with tautological classes in the Chow groups of moduli spaces of stable sheaves on K3 surfaces, 
Barros, Flapan, Marian and Silversmith have introduced in \cite[\S 4]{BFMS} an ascending filtration $S^{\mathrm{BFMS}}_k\CH_i(\operatorname{M}_\sigma(v))$ (see~\eqref{E:BFMS} for the definition), which in Proposition~\ref{P:BFMS-R} is shown to  coincide with our co-radical filtration (Definition~\ref{D:corad-positive})\,:
$$S^{\mathrm{BFMS}}_k\CH_i(\operatorname{M}_\sigma(v)) = R_{i+k}\CH_i(\operatorname{M}_\sigma(v)).$$
Regarding zero-cycles, the inclusion $S_k^{\mathrm{SYZ}}\CH_0(\operatorname{M}_\sigma(v)) \subseteq S^{\mathrm{BFMS}}_k\CH_0(\operatorname{M}_\sigma(v))$, together with the implication $[\mathcal E]\in S^{\mathrm{SYZ}}_k\CH_0(\operatorname{M}_\sigma(v)) \Rightarrow  ([\mathcal E]-[o])^{\times k+1} = 0 $ of \eqref{E:equivalence-SYZ}, is established in~\cite[Lem.~3]{BFMS}.
In Theorem~\ref{thm:splitting-moduli}, by exploiting the existence (Theorem~\ref{thm:main2}) of a strict grading on the birational motive $\ho(\operatorname{M}_\sigma(v))$,
we prove the reverse implication and thereby settle
$$S^{\mathrm{SYZ}}_k \CH_0(\operatorname{M}_\sigma(v))  = R_k\CH_0(\operatorname{M}_\sigma(v)), $$
 and henceforth the conjectured equality $S^{\mathrm{SYZ}}_k\CH_0(\operatorname{M}_\sigma(v))  = S^{\mathrm{small}}_k\CH_0(\operatorname{M}_\sigma(v))$ of \cite[Rem.~5]{BFMS}.

\subsubsection{Motivic surface decomposability for hyper-K\"ahler varieties}
Recently, Voisin~\cite{VoisinTriangle} introduced the cohomological notion of \emph{surface decomposability} (see Definition~\ref{def:Voisin}), conjectured that every hyper-K\"ahler variety is surface decomposable and established this in a number of cases, including \eqref{hilb}, \eqref{kum}, \eqref{fano} and \eqref{llsvs}. In Definition~\ref{D:MSD}, we introduce the notion of \emph{motivic surface decomposition}. This notion is concerned with zero-cycles and provides, in case the surfaces have vanishing irregularity, a refinement of Voisin's notion which is concerned with global $k$-forms\,; see Proposition~\ref{prop:voisin}. As such, our motivic surface decomposability  can be thought of as a lift to rational equivalence of Voisin's cohomological surface decomposability. In Conjecture~\ref{conj:MSD}, we conjecture that every hyper-K\"ahler variety is motivically surface decomposable, and we establish in Theorem~\ref{thm:MSD} the conjecture in cases \eqref{hilb}, \eqref{moduli}, \eqref{kum}, \eqref{fano} and \eqref{llsvs}.
In  the language of birational motives, Proposition~\ref{prop:coalg} implies that, for a  hyper-K\"ahler variety admitting a motivic surface decomposition,  the co-algebra structure on the birational motives of~$X$ is determined by the co-algebra structure on the birational motives of surfaces. Precisely, we have for instance\,:

\begin{thm2}[Theorem~\ref{thm:MSD} and Proposition~\ref{prop:coalg}]
	\label{thm2:msd}
Let $X$ be a smooth projective variety of dimension $2n$ that is birational to one of the hyper-K\"ahler varieties \eqref{hilb}, \eqref{moduli}, \eqref{kum}, \eqref{fano} or \eqref{llsvs}. Then there exists a smooth projective surface $B$, a split injective morphism $\gamma:\ho(X) \to \ho(B^n)$ and a split surjective morphism $\gamma' : \ho(B^n) \to \ho(X)$ such that 
\begin{enumerate}[(a)]
	\item $\gamma'\circ \gamma = \mathrm{id} : \ho(X) \to \ho(X)$.
	\item $(\gamma'\otimes \gamma')\circ \delta_{B^n} \circ \gamma = \delta_X : \ho(X) \to \ho(X\times X)$.
\end{enumerate}
\end{thm2}

\noindent Equivalently, in view of Lemma~\ref{lem:zero}, there exists a smooth projective surface $B$, correspondences $\gamma,\gamma' \in \CH^{2n}(X\times B^n)$  such that 
\begin{enumerate}[(a)]
	\item $(\gamma'\circ \gamma)_* = \mathrm{id} : \CH_0(X) \to \CH_0(X)$\,;
	\item $(\gamma'\otimes \gamma')_*(\delta_{B^n})_* \gamma_* = (\delta_X)_* :  \CH_0(X) \to \CH_0(X\times X)$.
\end{enumerate}

As a corollary to Theorem~\ref{thm:MSD} and its proof, we obtain in Corollary~\ref{C:SD} the existence of a surface decomposition in new cases, namely for moduli spaces of stable objects on K3 surfaces, and we reduce the existence of a surface decomposition for moduli spaces of stable objects in the Kuznetsov component of a cubic fourfold to a conjecture of Shen--Yin~\cite[Conj.~0.3]{SY}. We also obtain surface decompositions in case~\eqref{fano} for any surface with vanishing irregularity
 that is the base of a uniruled divisor\,; see Remark~\ref{R:anysurface}.

\subsection{Organization of the paper}
In \S \ref{S:co-alg} we review the notion of co-algebra object in a $\otimes$-category and introduce the notions of \emph{unital grading}, \emph{strict grading}, and \emph{co-radical filtration} in this general setting. 
In \S \ref{S:birat}, we introduce the $\otimes$-category of \emph{birational motives} due to Kahn--Sujatha, explain how it can be viewed as a ``$\otimes$-enhancement of $\CH_0$'', and explain how the diagonal embedding endows the birational motive of a variety with the structure of a co-algebra object. 
In \S \ref{S:moduli} we reformulate results of O'Grady, Marian, Shen, Yin and Zhao  and show that the birational motive of moduli spaces of stable objects on a K3 surface~$S$ is isomorphic as co-algebra object to the birational motive of~$\mathrm{Hilb}^n(S)$. 
 We then
show in \S \ref{S:MBCK} that the birational motives of hyper-K\"ahler varieties as in \eqref{hilb}, \eqref{moduli}, \eqref{kum}, \eqref{fano} can be endowed with a unital grading and in \S \ref{S:strict} that this unital grading is in fact a strict grading, by establishing Theorem~\ref{thm:main2}. 
In~\S \ref{S:corad}, we simply spell out the abstract definition of \emph{co-radical filtration} given in \S \ref{S:co-alg} in the case of varieties $X$ equipped with a unit $o\in \CH_0(X)$ and prove Proposition~\ref{P:co-rad}. 
Finally in \S \ref{SS:VoisinFil} we conclude in cases \eqref{hilb}, \eqref{kum}, \eqref{fano} (resp.\ in case \eqref{moduli})
that the filtration of Voisin (resp.\ Shen--Yin--Zhao) coincides with the co-radical filtration, thereby establishing our main Theorems~\ref{T:Voisin-filt} and~\ref{thm2:splitting-moduli}.

The paper concludes with three appendices, which are of independent interest and whose results are not used in the main body of this article\,: Appendix~\ref{S:app} discusses the co-radical filtration for positive-dimensional cycles, in particular for abelian varieties, as well as its relation to multiplicative Chow--K\"unneth decompositions and modified diagonals, and Appendix~\ref{S:delta} argues that the filtration of Voisin should be a \emph{split $\delta$-filtration}. In Appendix~\ref{S:MSD} we introduce the notion of \emph{motivic surface decomposition} and establish Theorem~\ref{thm2:msd}.

\subsection{Notation and Conventions} Given a field $K$, $\Omega$ denotes a
	\emph{universal domain} containing~$K$, \emph{i.e.}, an algebraically closed
	field of infinite transcendence degree over its prime subfield. For a smooth projective variety $X$ over $K$, we denote $\delta : X \hookrightarrow X\times_KX$ the diagonal embedding and $\epsilon : X \to \operatorname{Spec}K$ the structure morphism. In \S \S \ref{SS:biratmot}, \ref{SS:co-alg-biratmot} and~\ref{SS:0cycle}, Chow groups are with integral coefficients, unless explicitly stated otherwise. From \S \ref{SS:cor-coalg} onwards, Chow groups will be understood to be with rational coefficients.

	\subsection{Acknowledgements} Thanks to Lie Fu, Robert Laterveer and Mingmin Shen for useful exchanges during a pleasant stay at the University of Amsterdam in February 2020. Thanks to Giuseppe Ancona and Robert Laterveer for useful comments.

\section{On co-algebra objects in a $\otimes$-category}
\label{S:co-alg}
In this section, we fix a commutative ring $R$ and let $\mathcal{C}$ be a $R$-linear, symmetric monoidal category with tensor unit
$\mathds{1}$ (a $\otimes$-category over $R$, in the language of \cite[\S 2.2.2]{andre}).

\subsection{Co-algebra objects}\label{SS:co-alg}
A \emph{co-algebra object} in  $\mathcal{C}$ is an object $M$ together with a
co-unit morphism $\epsilon : M \to \mathds{1}$ and a co-multiplication morphism $\delta :
M\to M\otimes M$ satisfying the co-unit axiom $(\mathrm{id}\otimes
\epsilon)\circ \delta=\mathrm{id}=(\epsilon\otimes \mathrm{id})\circ \delta$ and the co-associativity axiom
$(\delta\otimes \mathrm{id})\circ\delta= (\mathrm{id}\otimes \delta)\circ \delta$. It is called
\emph{co-commutative} if moreover $\delta=c_{M, M}\circ \delta$ is satisfied, where
$c_{M, M}$ is the commutativity constraint of the category~$\mathcal{C}$. 
We define  inductively $\delta ^k := (  \delta \otimes \mathrm{id} \otimes \cdots \otimes \mathrm{id}) \circ   \delta^{k-1} : M\to  M^{\otimes (k+1)}$.
A \emph{morphism of co-algebra objects} between two co-algebra objects $M$ and $N$ is a
morphism $\phi: M\to N$ in $\mathcal{C}$ that preserves co-multiplication  and co-unit. We note that co-algebra structures on objects $M$ and $N$ of
$\mathcal{C}$ induce naturally a co-algebra structure on the tensor product $M\otimes N$, and that a morphism $\phi: M \to N$ of co-algebra objects
induces naturally a morphism of co-algebra objects $\phi^{\otimes n} : M^{\otimes
	n} \to N^{\otimes n}$ which is an isomorphism if $\phi$ is. Finally, if the co-algebra object $M$ is co-commutative, then the co-multiplication $\delta : M\to M\otimes M$ is a morphism of co-algebras.

A \emph{unit} for $M$ is a non-zero morphism $u: \mathds 1 \to M$ in $\mathcal C$ such that  $\delta \circ u = u \otimes u$ (which forces the additional identity $\epsilon \circ u = 1$). Equivalently, a unit is a co-algebra morphism $u: \mathds 1 \to M$.

\subsection{Unital graded co-algebra objects}\label{SS:grading}
Let $(M,\delta,\epsilon)$ be a co-commutative co-algebra object of $\mathcal C$. 
A \emph{grading} on the co-algebra objects $M$ is a (finite) direct sum decomposition 
$$M = M_{(0)} \oplus M_{(1)} \oplus \cdots \oplus M_{(n)}$$ in $\mathcal C$ with respect to which both $\delta$ and $\epsilon$ are graded morphisms, where the unit object $\mathds 1$ is understood to be of pure grade $0$. In other words, the grading has the property that
the restriction of the co-unit
$$\xymatrix{\overline M := \bigoplus_{i>0}M_{(i)} \ar@{^(->}[r] & M \ar[r]^\epsilon & \mathds 1
}$$ is zero, and
the restriction of the co-multiplication factors as 
$$\xymatrix{
	M_{(k)} \ar@{^(->}[r] \ar@{-->}[drr]  & M\ar@{->}[r]^\delta  &M\otimes M \\
	&& \bigoplus_{i+j=k} M_{(i)} \otimes M_{(j)} . \ar@{^(->}[u]
}$$  
Here we have followed the classical references \cite[\S 2]{MM}  and \cite[\S 11]{sweedler}.
Furthermore, a graded co-algebra object $M = M_{(0)} \oplus M_{(1)} \oplus \cdots \oplus M_{(n)}$ is said to be \emph{unital} if the restriction of the co-unit $\epsilon_0 : M_{(0)} \to \mathds 1$ is an isomorphism (in case $M$ is a co-algebra, this corresponds to the notion of \emph{connected (graded) co-algebra} in \cite[\S 2]{MM}) and of \emph{pointed irreducible graded co-algebra} in \cite{sweedler}). The terminology is justified by the fact that the graded morphism
$$\xymatrix{ u : \mathds 1 \ar[r]^{(\epsilon_0)^{-1}}_{\sim} & M_{(0)} \ar@{^(->}[r] & M}$$ defines a \emph{unit}\,; it is the unique graded unit morphism $\mathds 1 \to M$. 
We then write $(M,\delta, \epsilon, u)$ for a unital graded co-algebra. The tensor product of two unital graded co-algebra objects equipped with the obvious grading, co-multiplication, co-unit and unit is a unital graded co-algebra object.

\subsection{The reduced co-multiplication} \label{SS:co-mult}
Let $M = (M,\delta,\epsilon)$ be a co-algebra object of $\mathcal C$, equipped with a unit $u: \mathds 1 \to M$. 
The \emph{reduced co-multiplication} is defined to be 
\begin{equation}\label{E:reduced-comult}
\bar \delta := (\delta - u\otimes \mathrm{id}_{M} - \mathrm{id}_{M}\otimes u)\circ \bar p \ :  M \to  M\otimes  M,
\end{equation}
where
$\bar p:= \mathrm{id}_M-u\epsilon : M \to M$. We then define inductively $\bar \delta ^k := (\bar \delta \otimes \mathrm{id} \otimes \cdots \otimes \mathrm{id}) \circ \bar \delta^{k-1} :  M\to M\,^{\otimes (k+1)}.$
This notion is particularly relevant in case $M = (M,\delta,\epsilon,u)$ defines a unital graded co-algebra object of $\mathcal C$. In that case, $\bar p$ is the graded projector on $\overline M$, and $\bar \delta $ factors through $\overline M \otimes \overline M$ and its restriction to $\overline M$, which by abuse is still denoted $\bar \delta$, is given by
$$\bar \delta := \delta|_{\overline M} - u\otimes \mathrm{id}_{\overline M} - \mathrm{id}_{\overline M}\otimes u \ : \ \overline M \to \overline M\otimes \overline M.$$
To see that $\bar \delta $ factors through $\overline M \otimes \overline M$, one uses  the co-unit axiom, the fact that $\epsilon_0 : M_{(0)} \to \mathds 1$ is an isomorphism with inverse $u$ and the fact that the co-multiplication $\delta$ is graded.
One also easily checks that $\bar \delta$ is graded and co-associative. 
It is perhaps useful to explicitly mention that this gives that $\bar \delta |_{M_{(1)}} = 0$ and that for $k>0$ we have the factorization
$$\xymatrix{ \bar \delta|_{M_{(k)}} \ :\ M_{(k)} \ar[r]& 
\big(M_{(1)}\otimes M_{(k-1)}\big) \oplus \big(M_{(2)}\otimes M_{(k-2)} \big)\oplus \cdots \oplus \big(M_{(k-1)}\otimes M_{(1)}\big)
	 \ar@{^(->}[r] 
	 & \overline M \otimes \overline M.}$$
In particular, defining as above inductively 
$$\bar \delta ^k := (\bar \delta \otimes \mathrm{id} \otimes \cdots \otimes \mathrm{id}) \circ \bar \delta^{k-1} : \overline M\to\overline  M\,^{\otimes (k+1)},$$ 
we have 
$\bar \delta^k|_{M_{(1)} \oplus \cdots \oplus M_{(k)}} = 0$ and hence $\bar \delta^{n}=0$.
Finally $\bar \delta^k$ can be described as the composition
$$\xymatrix{\bar \delta^k \ :\ \overline M \ar@{^(->}[r] & M \ar[r]^{\delta^k\quad }& M^{\otimes (k+1)} \ar@{->>}[r]   &\overline M\,^{\otimes (k+1)}
},$$
\emph{i.e.,} $\bar p\,^{\otimes k+1}\circ\delta^k = \bar \delta^k\circ \bar p$ with the convention that $\bar \delta^0 = \bar p$.

\subsection{Co-generation of 
	co-algebra objects} 
Let $N$ be an object of $\mathcal C$. For all $n\geq 0$, we define 
$$T^{\leq n} N := \mathds 1 \oplus N \oplus N^{\otimes 2} \oplus \cdots \oplus N^{\otimes n}.$$
The canonical (iso)morphisms $N^{\otimes k} \to N^{\otimes i}\otimes N^{\otimes j}$ for $k=i+j$ naturally endow the object $T^{\leq n} N$ with the structure of a unital graded co-algebra. We call it the $n$-truncated tensor co-algebra on~$N$. 
Given a co-algebra object $M\in \mathcal C$ and a morphism $r: M\to N$ in $\mathcal C$,  the induced morphism 
$$T^{\leq n}r : M \to T^{\leq n}N,$$ where $T^{\leq n}r := \epsilon + r + r^{\otimes 2}\circ \delta +\cdots + r^{\otimes n}\circ \delta^{n-1}$, is a co-algebra morphism. 

We say that $M$ is \emph{co-generated} by $N$ if the induced co-algebra morphism 
$T^{\leq n}r : M \to T^{\leq n}N,$ is split injective for some $n>0$.
(This definition is inspired from the case of connected co-algebras, for which the co-free co-algebra generated by a sub-vector space $N\subseteq M$ coincides with the tensor co-algebra $TN$.)

Finally, we note that if $R$ contains $\Q$ and if $\mathcal C$ is pseudo-abelian, then $\delta^{k-1} :M \to M^{\otimes k}$ factors through the symmetric power
 $\operatorname{Sym}^k M$ (which can be defined as the image of the idempotent $\frac{1}{k!}\sum_{\sigma\in \mathfrak{S}_k} \sigma \, :\, M^{\otimes k} \to M^{\otimes k}$ with the symmetric group $\mathfrak{S}_k$ acting on $M^{\otimes k}$ by permuting the factors), so that  $T^{\leq n}r : M \to T^{\leq n}N$ factors through the $n$-truncated symmetric co-algebra $\operatorname{Sym}^{\leq n} N$.

\subsection{Strictly graded co-algebra objects}\label{SS:strictgrading}
 Recall that, by definition \cite[p. 232]{sweedler}, a pointed irreducible graded co-algebra $M$ is strictly graded if $M_{(1)}$ consists exactly of the primitive elements, \emph{i.e.}, if $M_{(1)} = P(M) := \{g\in M \, \vert \, \delta g = g\otimes u + u\otimes g\}$.
  In analogy, we say that a unital graded co-algebra object $M=M_{(0)}\oplus \cdots \oplus M_{(n)}$ of $\mathcal C$ is a \emph{strictly graded co-algebra object} if the restriction of the reduced co-multiplication
$$\bar \delta :  M_{(2)} \oplus M_{(3)} \oplus \cdots \oplus M_{(n)} \to \overline M \otimes \overline M$$ is split injective, or equivalently, in view of the general fact that $\bar \delta |_{M_{(1)}} = 0$, if $M_{(1)}$ is the kernel of $\bar \delta$. In that case, we say that $M_{(1)}$ is the \emph{primitive} part of $M$.

\begin{prop} \label{P:strictgrading}
A unital graded co-algebra object $M=M_{(0)}\oplus \cdots \oplus M_{(n)}$ of $\mathcal C$  is strictly graded if and only if it is co-generated by $M_{(1)}$.
\end{prop}
\begin{proof}
Note that the co-algebra morphism $M \to T^{\leq n} M_{(1)}$ is graded. 
First it is clear that if either the grading is strict or if $M$ is co-generated by $M_{(1)}$, then $\bar \delta : M_{(2)} \to M_{(1)} \otimes M_{(1)}$ is split injective. So assume that $\bar \delta : M_{(2)} \to M_{(1)} \otimes M_{(1)}$ is split injective and consider for $k\geq 2$ the composition
$$\xymatrix{\bar{\delta}^{k} : \ M_{(k+1)} \ar[r]^{\bar \delta \ \qquad }& \bigoplus_{i=1}^k M_{(k+1-i)}\otimes M_{(i)}  \ar[r] & M_{(k)}\otimes M_{(1)} \ar[rr]^{\bar \delta^{k-1}\otimes \operatorname{id}\quad } && M_{(1)}\otimes \cdots \otimes M_{(1)}.}$$
Assume that, for some $k\geq2$, $\bar \delta^{k-1} : M_{(k)} \to M_{(1)}^{\otimes k}$ is split injective\,; then $\bar \delta^{k} : M_{(k+1)} \to M_{(1)}^{\otimes k+1}$ is split injective if and only if $\bar \delta : M_{(k+1)} \to \bigoplus_{i=1}^k M_{(k+1-i)}\otimes M_{(i)}$ is split injective (note that $\bar{\delta}^{k-1}$ vanishes on $M_{(j)}$ for $j<k$\,; see \S \ref{SS:co-mult}). The proposition therefore follows by induction.
\end{proof}

\subsection{The co-radical filtration}\label{SS:co-rad}
Fix a ring homomorphism $R\to R'$ and a covariant $R$-linear functor $\operatorname{C} : \mathcal C \to R'\operatorname{-mod}$, \emph{e.g.}, $\Hom(N,-)$ for any choice of object $N\in \mathcal C$.
 Let $M = (M,\delta,\epsilon)$ be a  co-algebra object of~$\mathcal C$ equipped with a unit $u:\mathds 1 \to \mathcal C$. 
 Using the reduced co-multiplication~\eqref{E:reduced-comult}, we  introduce the \emph{co-radical filtration}
 $$R_k\operatorname{C}(M) := \ker \big(\bar \delta^k : \operatorname{C}(M) \to \operatorname{C}(M^{\otimes k+1})\big).$$
 In order to avoid confusion with the usual notion of co-radical filtration on co-algebras, we insist that in our setting the co-multiplication on $M$ does not endow the $R'$-module  $\operatorname{C}(M)$ with the structure of a co-algebra, unless $\operatorname{C}$ is a $\otimes$-functor.
 
 Suppose now that $M = (M,\delta,\epsilon,u)$ is a unital graded co-algebra object of~$\mathcal C$ with grading given by $M=M_{(0)} \oplus \cdots \oplus M_{(n)}$. 
  We then define the associated ascending filtration
$$G_k\operatorname{C}(M) :=\operatorname{C}\big(\bigoplus_{i\leq k} M_{(i)}\big) = \bigoplus_{i\leq k} \operatorname{C}(M_{(i)}).$$
 Recall that, in the context of co-algebras, a pointed irreducible graded co-algebra $M$ is strictly graded if and only if the filtration $G_\bullet$ defined by $G_kM := M_{(0)} \oplus \cdots \oplus M_{(k)}$ coincides with the co-radical filtration on $M$~\cite[Lem.~11.2.1]{sweedler}. The following proposition, which is a crucial observation for our work, justifies calling the filtration $R_\bullet$ above the co-radical filtration.

\begin{prop}\label{P:crucial} 
If $M = (M,\delta,\epsilon,u)$ is a unital graded co-algebra object of~$\mathcal C$, then we have for all $k$ the inclusion
$$G_k\operatorname{C}(M) \subseteq R_k\operatorname{C}(M)$$ with equality if the unital grading on $M$ is strict.
\end{prop}
\begin{proof}
The inclusion follows from the fact that $\bar \delta^k|_{M_{(i)}} = 0$ for $0\leq i \leq k$, which we already saw in \S \ref{SS:co-mult}. Suppose now that the unital grading on $M$ is strict. Consider an element $\alpha \in \operatorname{C}(M)$ and write 
$$\alpha = \alpha_0 + \alpha_1 + \cdots + \alpha_n\,, \qquad \alpha_i \in \operatorname{C}(M_{(i)}).$$ 
Assume that $\alpha \notin G_k\operatorname{C}(M)$, or equivalently that $\alpha_{k'} \neq 0$ for some $k'>k$. Since the unital grading on $M$ is strict, we have $\bar{\delta}^{ k'-1}\alpha_{k'} \neq 0$ and hence $\bar{\delta}^{ k}\alpha_{k'} \neq 0$.
Now, since $\bar \delta$ is graded, the elements $\bar{\delta}^{  k}\alpha_l$ belong to 
$\operatorname{C}((M^{\otimes k+1})_{(l)})$, where the grading on $M^{\otimes k+1}$ is the natural one. Since $M^{\otimes k+1} = \bigoplus_{l\geq 0} (M^{\otimes k+1})_{(l)}$,  we see that $\bar{\delta}^{  k}\alpha_{k'} \neq 0$ implies $\bar{\delta}^{  k}\alpha \neq 0$. 
We have thereby showed that if $\alpha \notin G_k\operatorname{C}(M)$ then $\alpha \notin R_k\operatorname{C}(M)$. 
\end{proof}

\section{Birational motives of varieties as co-algebra objects} \label{S:birat}
	
	\subsection{The $\otimes$-category of pure birational motives}\label{SS:biratmot}
	Let $R$ be a commutative ring. Recall that the (covariant) category $\mathcal M^{\mathrm{eff}}(K)_R$ of \emph{effective Chow motives} over $K$ with $R$-coefficients can be defined as the pseudo-abelian envelope of the category $\mathcal{CSP}(K)_R$ of smooth projective varieties over $K$ with morphisms given by $\Hom_{\mathcal{CSP}(K)_R}(X,Y) := \CH^{\dim Y}(X\times_K Y)\otimes R$ and composition law given by the composition of correspondences. 
	We write $\h(X)$ (or $\h(X)_R$ when we want to make explicit the ring of coefficients) for the Chow motive of $X$ (\emph{i.e.}, for $X$ seen as an object of  $\mathcal M^{\mathrm{eff}}(K)_R$), and given an idempotent correspondence $p \in \CH^{\dim X}(X\times_K X)\otimes R$, we write $(X,p)$ or $p\h(X)$ for $\operatorname{im}(p)$. The \emph{unit motive} is $\mathds 1 := \h (\operatorname{Spec}K)$. The category  $\mathcal M^{\mathrm{eff}}(K)_R$ is a $R$-linear $\otimes$-category, with $\otimes$-unit the unit motive and with tensor product given by $(X,p)\otimes (Y,q) = (X\times_K Y, p\otimes q)$, where $p\otimes q = p_{13}^*p\cdot p_{24}^*q$ with $p_{ij}$ being the projection morphism to the product of the $i$-th and $j$-th factor of  $X\times Y \times X \times Y$.
	
	The diagonal $\Delta_{\PP^1_K}$ of the projective line $\PP^1_K$ decomposes as a sum of two mutually orthogonal idempotents $\Delta_{\PP^1_K} = \PP^1_K\times \{0\} + \{0\}\times \PP^1_K  \ \mbox{in}\ \CH^1(\PP^1_K\times_K \PP^1_K)$ yielding a direct sum decomposition $\h(\PP^1_K) = \mathds{1} \oplus \mathds{L}$, where $\mathds{L} := (\PP^1_K, \{0\}\times \PP^1_K)$ is by definition the \emph{Lefschetz motive}.
	
	The (covariant) category $\mathcal M(K)_R$ of \emph{Chow motives} over $K$ with $R$-coefficients  is then obtained from $\mathcal M^{\mathrm{eff}}(K)_R$ by inverting the $\otimes$-endofunctor $-\otimes \mathds L$. The resulting category $\mathcal M(K)_R$  is then rigid and the functor $\mathcal M^{\mathrm{eff}}(K)_R \to \mathcal M(K)_R$ is fully faithful.
	
	If instead of inverting the Lefschetz motive, one kills the Lefschetz motive in $\mathcal M^{\mathrm{eff}}(K)_R$, one obtains the category $\mathcal M^{\circ}(K)_R$ of \emph{pure birational (Chow) motives} over $K$ with $R$-coefficients, which was introduced by Kahn--Sujatha~\cite{KS}. Precisely (\emph{cf.}~\cite[\S 2.2]{KS}), consider $\mathcal L$ the ideal of $\mathcal M^{\mathrm{eff}}(K)_R$ consisting of those morphisms which factor through some object of the form $P\otimes \mathds L$\,; it is a $\otimes$-ideal called the \emph{Lefschetz ideal}. The category $\mathcal M^{\circ}(K)_R$ is then defined to be the pseudo-abelian envelope of the quotient $\mathcal M^{\mathrm{eff}}(K)_R/\mathcal L$\,; it is a (non-rigid) $R$-linear $\otimes$-category. (Note that conjecturally, the quotient $\mathcal M^{\mathrm{eff}}(K)_R/\mathcal L$ is already pseudo-abelian, so that it should not be necessary to pass to the pseudo-abelian envelope; see \cite[Prop.~4.4.1]{KS}). 
		We write $\ho(X)$ (or $\ho(X)_R$ when we want to make explicit the ring of coefficients) for the birational motive of $X$ (\emph{i.e.}, for $X$ seen as an object of  $\mathcal M^{\mathrm{eff}}(K)_R$),  and	given an idempotent correspondence $\varpi \in \Hom(\ho(X)_R,\ho(X)_R$, we write $(X,\varpi)$ or $\varpi\ho(X)$ for $\operatorname{im}(\varpi)$.
	\medskip
	
	Morphisms of birational motives have an explicit description, namely for $X$ irreducible we have (see~\cite{KS})\,:
\begin{equation}\label{E:biratmorphism}
\operatorname{Hom}(\h^\circ(X)_R,\h^\circ(Y)_R) := \CH_0(Y_{K(X)}) \otimes R = \varinjlim \CH_{\dim X} (U\times_K Y) \otimes R,
\end{equation}
where the limit runs through all non-empty open subsets of $X$.	
Hence the birational motive of $X$ with $R$-coefficients can be roughly thought of as the collection $\{\CH_0(X_L)\otimes R \, \big\vert\, L/K \mbox{ is a field extension}\}$ with morphisms that are ``motivic'', \emph{i.e.}, induced by correspondences.

Furthermore, under the functor $\mathcal M^{\mathrm{eff}}(K)_R \to \mathcal M^{\circ}(K)_R$, a morphism of effective Chow motives $\gamma : \h(X)_R \to \h(Y)_R$, that is, a correspondence in $\CH^{\dim Y}(X\times Y)\otimes R$, induces the morphism of birational motives $\ho(X)\to \ho(Y)$ given by restricting $\gamma$ to the generic point of $X$. In addition, a rational map $f: X\dashrightarrow Y$ induces a well-defined morphism $f_* : \ho(X) \to \ho(Y)$, obtained by restricting to the generic point of $X$ the graph of $f|_U$, where $U\subseteq X$ is a dense open subset over which $f$ is defined. 
	We refer to \cite[\S 2]{Shen} for further explicit calculations involving morphisms of birational motives\,; of particular relevance is the fact that a generically finite rational map $f: X\dashrightarrow Y$ induces a well-defined morphism $f^* : \ho(Y) \to \ho(X)$.
	\medskip
	
	In practice, with rational coefficients, morphisms of birational motives are the same as action of correspondences on zero-cycles.
	The following lemma shows indeed that a morphism of birational motives is uniquely
	determined by its action on zero-cycles, after base-changing to a sufficiently large field.
	\begin{lem}\label{lem:zero}
		Let $X$ and $Y$ be two smooth projective varieties over a field $K$ and fix a universal domain $\Omega$ containing $K$.  Let  $\gamma$ and $\gamma'$ be two morphisms $\h(X) \to \h(Y)$ of Chow motives. Then the
		following two conditions are equivalent\,:
		\begin{enumerate}[(a)]
			\item $\gamma_* = \gamma'_* : \CH_0(X_\Omega)\otimes \Q \to \CH_0(Y_\Omega)\otimes \Q$\,;
			\item  $\gamma = \gamma' : \ho(X)_\Q \to \ho(Y)_\Q$.
		\end{enumerate}  
	\end{lem}
	\begin{proof} 
	 First, recall the fact that if $\gamma \in \operatorname{Hom}(\h(X),\h(Y)):= \CH_{\dim X}(X\times_KY)$, then the restriction of $\gamma$ to the generic point $\eta_X$ of $X$, \emph{i.e.}, the image of $\gamma$ under the map $\CH^{\dim Y}(X\times_KY) \to \varinjlim \CH_{\dim X}(U\times_K Y) = \CH_0(Y_{K(X)})$, coincides with $(\gamma_{K(X)})_*[\eta_X]$.
		
		Assuming $(b)$, we have by~\eqref{E:biratmorphism} and by Bloch--Srinivas \cite{BS} that a non-zero multiple of $\gamma - \gamma'$ is supported on $D\times_k Y$ for some divisor $D\subset X$. It follows that $\gamma - \gamma'$ acts as zero on $ \CH_0(X_\Omega)\otimes \Q$.
		
	Conversely, since the base change map $\CH^*(X)\otimes \Q \to \CH^*(X_L) \otimes \Q$ is injective for all field extensions $L/K$ and all schemes $X$ of finite type over $K$, it follows from $(a)$ that $\gamma - \gamma'$ acts as zero on the generic point of $X$. It follows that $\gamma - \gamma'$ induces the zero morphism  $\ho(X)_\Q \to \ho(Y)_\Q$.
	\end{proof}

As a result of general interest, we have the following analogue of \cite[Thm.~3.18]{Vial-Tohoku} concerned with Chow motives\,:

	\begin{lem}\label{lem:birat-surj}
	Let  $M^\circ$ and $N^\circ$ in $\mathcal{M}^\circ(K)_\Q$ be two birational motives over a field $K$ and fix a universal domain $\Omega$ containing $K$.  Let  $\gamma : M^\circ \to N^\circ$ be a morphism and let $(\gamma_{\Omega})_* : \CH_0(M^\circ_\Omega)_\Q \to  \CH_0(N^\circ_\Omega)_\Q$ be the induced morphism. Then\,:
	\begin{enumerate}[(a)]
		\item $\gamma$ is split surjective if and only if $(\gamma_{\Omega})_*$ is surjective\,;
		\item  $\gamma$ is an isomorphism if and only if $(\gamma_{\Omega})_*$ is bijective.
	\end{enumerate}  
\end{lem}
\begin{proof} 
Item (b) follows from (a) and Lemma~\ref{lem:zero}. The ``only if''  part of (a) is clear, so suppose  $(\gamma_{\Omega})_*$ is surjective. In that case, note that $(\gamma_{K(Y)})_*$ is also surjective (\emph{e.g.}, \cite[Lem.~3.1]{Vial-Tohoku}), and denoting $N^\circ = (Y,\varpi)$, there exists thus $\rho \in  \CH_0(M^\circ_{K(Y)}) $ such that $\varpi=(\gamma_{K(Y)})_*\rho$. 
On the other hand $(\gamma_{K(Y)})_*\rho = \gamma\circ \rho$, so that $\rho \circ \varpi$ provides a splitting to $\gamma$.
\end{proof}

	\subsection{The co-algebra structure on the birational motive of varieties}\label{SS:co-alg-biratmot}
	Let $X$ be a smooth projective variety over a field~$K$.
Recall that the pullback along the diagonal embedding $\delta : X \to X\times_K X$, together with the $\otimes$-structure on the contravariant category of Chow motives ($\h(X\times X) = \h(X)\otimes \h(X)$), provides a commutative algebra structure on the Chow motive $\h(X)$ of~$X$
	 (\emph{e.g.}, \cite[Ex.~4.1.4.1.3]{andre} or \cite[\S 2.1]{FV-K3}),
	 with unit induced by the structure morphism $\epsilon: X \to \operatorname{Spec}K$.
Working covariantly instead, we have that  the pushforward along the diagonal embedding $\delta : X \to X\times_K X$ endows the covariant motive of $X$ with the structure of a co-commutative co-algebra object. Passing to the birational motive, we obtain that  the pushforward along the diagonal embedding $\delta : X \to X\times_K X$, together with the $\otimes$-structure on the category of birational motives ($\ho(X\times_K X) = \ho(X)\otimes \ho(X)$),  provides a \emph{co-commutative co-algebra structure} on the birational motive $\ho(X)$ of $X$, with  co-unit morphism $\epsilon : \ho(X) \to \mathds 1$ (also called the \emph{degree morphism}) induced by pushing forward along the structure morphism $\epsilon: X \to \operatorname{Spec} K$.  It is indeed immediate to check that 
	$(\mathrm{id}\times
	\epsilon)\circ \delta=\mathrm{id}=(\epsilon\times \mathrm{id})\circ \delta : X\to X$,
	$(\delta\times \mathrm{id})\circ\delta= (\mathrm{id}\times \delta)\circ \delta : X\to X \times X \times X$, as well as $\delta=\tau\circ \delta : X \to X\times X$, where $\tau : X\times X \to X\times X$ is the morphism permuting the two factors.

	In the same way that the pullback along a morphism of smooth projective
	varieties provides a morphism between their motives as algebra objects, a
	rational map $f:X \dashrightarrow Y$ induces a morphism $f_* : \h^\circ (X) \to
	\ho(Y)$ as co-algebra objects. This follows immediately from the commutativity
	of the diagram 
	$$\xymatrix{X \ar[rr]^{\delta_X} \ar@{-->}[d]_f & &X\times X
		\ar@{-->}[d]^{f\times f} \\
		Y \ar[rr]^{\delta_Y} && Y\times Y.
	}$$

\subsection{Co-algebra structure on birational motives and zero-cycles}\label{SS:0cycle}
Let $X$ be a smooth projective variety over a field $K$.
The co-multiplication morphism $\delta : \ho(X) \to \ho(X \times_K X) = \ho(X)\otimes \ho(X)$ does not endow $\CH_0(X) = \Hom (\mathds 1, \ho(X))$ with the structure of a co-algebra, but only provides a map
$$\delta_* : \CH_0(X) \to \CH_0(X\times_K X).$$
Indeed, given two schemes $X$ and $Y$ of finite type over $K$,  the natural map given by exterior product~\cite[\S 1.10]{fulton}\,:
\begin{equation}\label{E:exterior}
\CH_0(X) \otimes \CH_0(Y) \to \CH_0(X\times_K Y), \quad \alpha\otimes \beta \mapsto \alpha \times \beta.
\end{equation}
 can fail to be injective or surjective, even with rational coefficients. In other words, the additive functor $M\mapsto \Hom(\mathds 1,M) = \CH_0(M)$ from birational motives to abelian groups is \emph{not} a $\otimes$-functor --- this explains why we work with birational motives rather than merely with zero-cycles.
 
 To see that \eqref{E:exterior} is neither injective nor surjective in general, consider for example an elliptic curve $E$ over $\Q$\,; by the Mordell--Weil theorem it has finite rank, while its base-change to $\bar \Q$ has infinite rank. We therefore see that there exists a finite field extension $L/\Q$ such that $\CH_0(L)\otimes \CH_0(E) \to \CH_0(E_L)$ is not surjective, even with rational coefficients. Assume now that $E$ has a $K$-point $p$ such that $[p]-[0]$ is non-torsion in $\CH_0(E)$, then $([p]-[0]) \times ([p]-[0])$ is zero
 in $\CH_0(E\times_K E)$ (this is classical, but see also Theorem~\ref{thm:abelian} below for a generalization to abelian varieties of any dimension), showing that $\CH_0(E)\otimes \CH_0(E) \to \CH_0(E\times_K E)$ is not injective, even with rational coefficients. 
 However, we note that if $K$ is algebraically closed, then the exterior product map~\eqref{E:exterior} is surjective\,; indeed, any zero-cycle $\gamma \in   \CH_0(X\times_K Y)$ is then a linear combination of cycle classes of the form $[x]\times[y]$ for $x\in X(K)$ and $y\in Y(K)$.

\subsection{Correspondences and co-algebra structures on birational motives}\label{SS:cor-coalg}
From now on, our coefficient ring $R$ will be the field of rational numbers $\Q$ and Chow groups (and motives) will be understood to be with rational coefficients.

	Let $\ho(X)$ and $\ho(Y)$ be two birational motives of smooth projective
	varieties. Assume that $\ho(X)$ can be realized as a direct summand of $\ho(Y)$.
	The following proposition gives a criterion for the co-algebra structure on
	$\ho(X)$ to be determined by the co-algebra structure on $\ho(Y)$.
	
	\begin{prop}\label{prop:coalg}
		Let $X$ and $Y$ be smooth projective varieties of same dimension $d$ over a
		field~$K$ and fix a universal domain $\Omega$ containing $K$. Assume that there exist a projective variety $\Gamma$  of same
		dimension $d$ together with generically finite morphisms
		$$\xymatrix{ \Gamma \ar[r]^\phi \ar[d]_\psi & X \\
			Y
		}$$ 
		such that 
		one of the following equivalent conditions holds\,:
		\begin{enumerate}[(i)]
			\item $\phi_*[p] = \phi_*[q]$  in $\CH_0(X_\Omega)$, for any two general
			points $p$ and $q$ in $\Gamma(\Omega)$ lying on the same fiber of $\psi$.
			\item  	$\phi_* \psi^*\psi_* \alpha = \deg(\psi)\, \phi_*\alpha$ in
			$\CH_0(X_\Omega)$, for any zero-cycle $\alpha \in \CH_0(\Gamma_\Omega)$.
			\setcounter{1}{\value{enumi}}
		\end{enumerate}
		Then 
		\begin{enumerate}[(a)]
			\item $\gamma := \frac{1}{\deg \phi}\,\psi_*\phi^* : \ho(X) \to \ho(Y)$ is
			split injective\,;
			\item $\gamma' := \frac{1}{\deg \psi}\,\phi_*\psi^* : \ho(Y) \to \ho(X)$ is
			split surjective and 
			$\gamma'\circ
			\gamma = \mathrm{id}_{\ho(X)}$\,;
			\item 	the diagram
			$$\xymatrix{\ho(X) \ar[rr]^{\delta_X\quad }  \ar[d]_\gamma && \ho(X)\otimes
				\ho(X) \\
				\ho(Y) \ar[rr]^{\delta_Y \quad } && \ho(Y)\otimes \ho(Y)
				\ar[u]_{\gamma'\otimes \gamma'}		
			}$$
			commutes.
		\end{enumerate}
		In particular, if in addition $\psi_*[p] = \psi_*[q]$  in $\CH_0(Y_\Omega)$ for
		any two general points $p$ and $q$ in $\Gamma(\Omega)$ lying on the same fiber
		of $\phi$, then $\gamma : \ho(X) \to \ho(Y)$ is an isomorphism of co-algebra
		objects (with inverse~$\gamma'$).
	\end{prop}
	\begin{proof}
		First we explain why assumptions $(i)$ and $(ii)$ are equivalent.
		Since $\psi$ is generically finite and by generic flatness, for $p$ a general point on $Y_\Omega$ we have that $\psi^*\psi_*[p]$ is a multiple of $[p_1] + \cdots + [p_n]$ where $\{p_1,\ldots,p_n\} = \psi^{-1}\psi(p)$ and where $n = \deg(\psi)$. Assuming $(i)$, we get $\phi_*\psi^*\psi_*[p] = \deg(\psi) \phi_*[p]$.
		 On the other hand, if $p$ and $q$ are two general points on the same fiber of $\psi$, we have by proper pushforward $\psi_*[p] = \psi_*[q]$.
		 Assuming~$(ii)$, we get $\deg(\psi)\, \phi_*[p] = \phi_* \psi^*\psi_* [p] =   \phi_* \psi^*\psi_* [q] = \deg(\psi)\, \phi_*[q]$.
		
		Let us now proceed to show that assumptions $(i)$ and $(ii)$ imply $(a)$, $(b)$ and $(c)$.
		Items $(a)$ and $(b)$ simply follow from $(c)$ by projecting on the first factor\,; or directly from $(ii)$ and from the projection formula\,: $$
		\phi_*\psi^*\psi_*\phi^* =
		\deg(\psi)\, \phi_*\phi^* = \deg(\psi)\deg(\phi)\, \mathrm{id}_{\ho(X)}.$$
		Item $(c)$ follows from
		\begin{align*}
		(\phi_*\psi^*
		\otimes \phi_*\psi^*)\circ \delta_Y \circ \psi_*\phi^*
		&= (\phi_*\psi^*\otimes \phi_*\psi^*)\circ (\psi_*\otimes \psi_*) \circ
		\delta_\Gamma \circ\phi^*\\
		&=  \deg(\psi)^2\, (\phi_*\otimes \phi_*) \circ \delta_\Gamma \circ\phi^*\\
		&=  \deg(\psi)^2\, \delta_X \circ\phi_*\phi^*\\
		&=  \deg(\psi)^2\deg(\phi)\, \delta_X.
		\end{align*} 
		Here, the second equality uses $(ii)$, the last equality uses the projection formula, and the first
		and third equalities use the compatibility of the co-algebra structure on
		birational motives of smooth projective varieties and pushforwards along
		rational maps.
	\end{proof}
	
	\begin{rmk}\label{R:biratGamma}
		We note that if $\pi: \widetilde{\Gamma}\to \Gamma$ is a birational morphism of projective varieties, \emph{e.g.}, a desingularization, then the equivalent assumptions $(i)$ and $(ii)$ of Proposition~\ref{prop:coalg} are satisfied for $\phi\circ \pi$ and $\psi\circ \pi$ if they are satisfied for $\phi$ and $\psi$. 
	\end{rmk}

	\subsection{The birational motive of finite quotient varieties}\label{SS:quotient}
	
	Let $X$ be a smooth projective variety over a field $K$ and let $G$ be a finite group acting on $X$. Due to the fact~\cite[Ex.~1.7.6]{fulton} that $\CH^*(X/G)\otimes \Q = (\CH^*(X)\otimes \Q)^G$ and due to the fact that the formalism of correspondences carries through with rational coefficients to finite quotients of smooth projective varieties,  all the formalism developed so far in this section carries through, so long as one works with rational coefficients, to finite quotients of smooth projective varieties. 
	This will prove important in our examples, since we will take as a birational model for~\eqref{hilb} the symmetric quotient $S^{(n)} := S^n/\mathfrak S_n$ and as a birational model for~\eqref{kum} a certain quotient $A^n/\mathfrak{S}_{n+1}$ (see the proof of Theorem~\ref{thm:MSD}).

	\section{The birational motive of moduli spaces of objects on K3 surfaces}
	\label{S:moduli}
	
Let $S$ be a K3 surface over a field $K$, let $[o]$ be the numerical class of a point $o\in S(\bar K)$ and let $v = (v_0,v_2,v_4) \in \Z[S] \oplus \operatorname{NS}(S) \oplus \Z[o]$ be a primitive 
class with non-negative Mukai self-intersection $v^2 := -v_0v_4 + v_2v_2 - v_4v_0 \ge 0$. For a generic stability condition $\sigma \in \operatorname{Stab}^\dagger (S)$ with respect to~$v$ (see \cite{Bridgeland}), we denote $\operatorname{M}_\sigma(v)$ the moduli space of $\sigma$-stable objects, in the bounded derived category $D^b(S)$ of coherent sheaves on $S$, with Mukai vector $v$\,; $\operatorname{M}_\sigma(v)$ is a smooth projective hyper-K\"ahler variety of dimension $2n = v^2+2$.

In the case of moduli spaces $M_H(v)$ of Gieseker-stable sheaves, with Mukai vector $v$, with respect to a generic polarization $H$ on the K3 surface $S$ (which are special cases of moduli spaces of $\sigma$-stable objects in $D^b(S)$), Markman~\cite[\S 3.4]{MarkmanJAG} has established that there exists an isomorphism between the cohomology algebras 
of $M_H(v)$ and $\operatorname{Hilb}^n(S)$
 that in addition preserves the Hodge structures. Recently, Frei~\cite{Frei} extended Markman's result to positive characteristic, with $\ell$-adic cohomology with its Galois structure, in place of singular cohomology with its Hodge structure.
 
 Given the above and Beauville's splitting principle, it is natural to ask whether the Chow motives of $\operatorname{M}_\sigma(v)$ and $\operatorname{Hilb}^n(S)$ are isomorphic as algebra objects. The following theorem gives evidence by establishing the above in the context of birational motives.
	
	\begin{thm}\label{thm:moduli}
		Let $S$ be a K3 surface and let $M_{\sigma}(v)$ be a moduli space of stable
		objects on~$S$.
		If $2n$ denotes the dimension of $M_{\sigma}(v)$, then the birational Chow motives
		$\ho(M_{\sigma}(v))$ and $\ho(\mathrm{Hilb}^n(S))$ are isomorphic as co-algebra objects.	
	\end{thm}
	\begin{proof}
		 As in \cite[\S	2.2]{SYZ}, we consider the incidence
		$$R:= \{(\mathcal{E},\xi) \in \operatorname{M}_\sigma(v)\times \mathrm{Hilb}^n(S) ~\big\vert~
		c_2(\mathcal{E}) = [\mathrm{Supp}(\xi)] +c\, [o_S] \ \mbox{in}\ \CH_0(S)\},$$ where
		$\mathrm{Supp}(\xi)$ is the support of $\xi$ and $c\in \Z$ is a constant
		determined by the Mukai vector $v$, and let $p_{\operatorname{M}_\sigma(v)} : R  \to \operatorname{M}_\sigma(v)$ and
		$p_{\mathrm{Hilb}^n(S)} : R \to \mathrm{Hilb}^n(S)$ be the natural projections.
		By \cite{MZ}, all points on the same fiber of $p_{\operatorname{M}_\sigma(v)}$ have the same class in
		$\CH_0(\mathrm{Hilb}^n(S))$ and all points on the same fiber of
		$p_{\mathrm{Hilb}^n(S)}$ have the same class in $\CH_0(\operatorname{M}_\sigma(v))$. Moreover
		by \cite[Thm.~0.1]{SYZ} $p_{\operatorname{M}_\sigma(v)}$ is dominant, while by the arguments in
		\cite[Prop.~1.3]{OG-moduli} $p_{\mathrm{Hilb}^n(S)}$ is dominant\,; in
		fact there exists a component $R_0 \subseteq R$ that dominates both factors
		$\operatorname{M}_\sigma(v)$ and $\mathrm{Hilb}^n(S)$. The varieties $\operatorname{M}_\sigma(v)$ and
		$\mathrm{Hilb}^n(S)$ have same dimension and, up to restricting to a linear
		section, we can further assume that $R_0$ is generically finite over both $\operatorname{M}_\sigma(v)$
		and $\mathrm{Hilb}^n(S)$. By applying Proposition~\ref{prop:coalg}, we obtain an
		isomorphism of birational motives 
		$$\ho(\operatorname{M}_\sigma(v))  \stackrel{\sim}{\longrightarrow}
		\ho(\mathrm{Hilb}^n(S)),$$
		as co-algebra objects.
	\end{proof}

	\section{Co-multiplicative birational Chow--K\"unneth decompositions}
	\label{S:MBCK}
	
In this section, we start by recalling the notion of \emph{birational Chow--K\"unneth decomposition} for the birational motive $\ho(X)$ of a smooth projective variety $X$. Such a decomposition is then said to be co-multiplicative if it defines a unital grading on $\ho(X)$ considered as a co-algebra object. (In the next section, we will in fact see that in case $X$ is a hyper-K\"ahler variety, then any unital grading on $\ho(X)$ is a co-multiplicative birational Chow--K\"unneth decomposition.)
The main result is Theorem~\ref{thm:BMCK} where we construct explicit such decompositions in case $X$ is one of \eqref{hilb}, \eqref{moduli}, \eqref{kum} or \eqref{fano}.
	
	\subsection{Birational Chow--K\"unneth decompositions} \label{SS:BCK}
	The following definitions are borrowed from Shen~\cite[\S 3]{Shen}.
	Fix a Weil cohomology theory $\HH^\bullet$ for smooth projective varieties defined over $K$\,; \emph{e.g.}, $\ell$-adic cohomology for $\ell \neq \operatorname{char}(K)$, or  Betti cohomology if $K\subseteq \C$. 
	For a smooth projective variety $X$ over $K$, we then define its transcendental cohomology to be the quotient
	$$\HH_{\tr}^k(X) := \HH^k(X)/ \NN^1\HH^k(X),$$
	where $\NN^\bullet$ denotes	the coniveau filtration\,:
	$$\NN^r \HH^k(X) := \sum_{Z\subseteq X} \ker \Big(\HH^k(X) \to \HH^k(X\setminus Z)\Big),$$ where the sum is over all codimension-$r$ closed subsets $Z$ of $X$. Note that,  \emph{e.g.}\ by \cite[\S 1.1]{ACMVbloch}, the action of correspondences preserves the coniveau filtration.
	Note also that, due to the Hard Lefschetz theorem, we have that $\HH^k_{\tr}(X) = 0$ as soon as $k>\dim X$.
	Now a morphism $\gamma \in \Hom(\ho(X), \ho(Y)) = \CH_0(Y_{K(X)})$ induces a homomorphism $$\gamma^* :  \HH^k_{\tr}(Y) \to \HH^k_{\tr}(X),$$ obtained by letting a lift of $\gamma$ to $\CH^{\dim X}(X\times Y)$ act on $\HH^k(Y)$\,; this is well-defined since the difference of any two lifts is a correspondence supported on $D\times Y$ for some divisor $D\subseteq X$ and hence sends $\HH^k(Y)$ into $\NN^1\HH^k(X)$. Therefore, given a birational motive $\varpi \ho(X)$, one may define its transcendental cohomology $\HH^*_{\tr}(\varpi\ho(X))$ as $\varpi^*\HH^*_{\tr}(X)$. 
	Note however that $\HH^*_{\tr}$ does \emph{not} define a $\otimes$-functor from the category of birational motives to the category of graded vector spaces\,; for instance, if $C$ is a smooth projective curve, then $\HH^1_{\tr}(C) \otimes \HH^1_{\tr}(C) \subsetneq \HH^2_{\tr}(C\times C)$ since $\HH^2_{\tr}(C\times C)$ does not contain the $(1,1)$-component of the diagonal $\Delta_C$.

	\begin{defn}[Birational Chow--K\"unneth decomposition] \label{D:BCK}
		Let $X$ be a smooth
	projective variety over a field~$K$. A birational Chow--K\"unneth decomposition
	of $\ho(X)$ is a decomposition $$\ho(X) = \ho_0(X) \oplus \cdots \oplus \ho_d(X), \quad \mbox{with} \ \ho_i(X) = \varpi^X_i\ho(X) = (X,\varpi^X_i)$$ such that $\HH^*_{\tr}(\ho_i(X)) := (\varpi^X_i)^*\HH_{\tr}(X) =  \HH^i_{\tr}(X).$ In other words, a birational  Chow--K\"unneth decomposition is a collection
	$\{\varpi_0^X, \ldots, \varpi_d^X \} \subset \End(\ho(X))$
	 such that
\begin{enumerate}[(a)]
	\item $\mathrm{id}_{\ho(X)}  = \varpi_0^X+\cdots + \varpi_d^X$\,;
	\item $\varpi_k^X \circ \varpi_k^X =  \varpi_k^X$ for all $k$\,;
	\item $\varpi_i^X \circ \varpi_j^X = 0$ for all $i\neq j$\,;
	\item $(\varpi_k^X)^* :  \HH^l_{\tr}(X) \to \HH^l_{\tr}(X)$ is the identity if $k=l$ and is zero otherwise.
\end{enumerate}

\end{defn}
If $X$ has a
	\emph{Chow--K\"unneth decomposition} $\{\pi_X^0,\ldots, \pi_X^{2\dim X}\}$ (in the sense of Murre~\cite{Murre}\,; see Definition~\ref{D:MCK}), then  $\varpi_i^X :=
(\pi^{2\dim X-i}_X)|_{k(X)\times X}$ defines a birational Chow--K\"unneth decomposition. In particular, in view of Murre's conjecture~\cite{Murre}, a birational Chow--K\"unneth decomposition is expected to exist for all smooth projective varieties. 
(Note that by general conjectures we should have $\varpi^X_i = 0$ for all $i>\dim X$).
Moreover, the descending filtration $F^\bullet$ on $\CH_0(X)$ defined by 
\begin{equation}\label{E:BB}
F^k \CH_0(X) := \CH_0\big (\ho_k(X) \oplus \cdots \oplus \ho_d(X)\big)
\end{equation}
is expected to be independent of the choice of a birational Chow--K\"unneth decomposition and to coincide with the conjectural Bloch--Beilinson filtration\,;  see \cite[\S 5]{Jannsen}.
Finally, having a birational Chow--K\"unneth decomposition is a stably birational invariant\,; see~\cite[Prop.~3.4]{Shen}.

	\subsection{Co-multiplicative birational Chow--K\"unneth decompositions}
	
	\begin{defn}[Co-multiplicative birational Chow--K\"unneth decomposition]
		\label{D:MBCK}
		 Let $X$ be a smooth
		projective variety over a field $K$. A birational Chow--K\"unneth decomposition
		$\{\varpi_0^X, \ldots, \varpi_d^X \}$ of $\ho(X)$ is said to be
		\emph{co-multiplicative} if the induced decomposition $$\ho(X) = \ho_0(X) \oplus \cdots \oplus \ho_d(X), \quad \ho_k(X) := \varpi_k^X\ho(X)$$ defines a \emph{unital grading} of the co-algebra object $\ho(X)$. 
	\end{defn}

\begin{lem}\label{L:concrete-coMBCK}
A birational Chow--K\"unneth decomposition
$\{\varpi_0^X, \ldots, \varpi_d^X \}$ of $\ho(X)$ is co-multiplicative if and only if
\begin{enumerate}[(a)]
	\item $\varpi^X_0 = o_{K(X)}$ for some  zero-cycle $o \in \CH_0(X)$, 
	and
	\item$(\varpi_i^X \otimes \varpi_j^X) \circ \delta_X \circ \varpi_k^X = 0$ for all $k\neq i+j$.
\end{enumerate} 
\end{lem}
\begin{proof}
	We note that, due to the fact that the idempotents $\varpi^X_k$ act as zero on $\HH^{0}(X)=\HH^0_{\tr}(X)$ for $k>0$, the degree map $\epsilon : \ho(X) \to \mathds 1$ (which is induced by the structure morphism $X \to \operatorname{Spec} K$)  restricts to the zero map on $\overline{\ho}(X) := \bigoplus_{k>0} \ho_k(X)$ for any choice of birational Chow--K\"unneth decomposition. Therefore the birational Chow--K\"unneth decomposition
$\{\varpi_0^X, \ldots, \varpi_d^X \}$ of $\ho(X)$ is 
co-multiplicative if and only if the degree morphism $\epsilon : \ho(X) \to \mathds 1$ is such that
$\epsilon_0 : \ho_0(X) \stackrel{\sim}{\longrightarrow} \mathds 1$ is an isomorphism of co-algebra objects,
and 
the induced grading is a co-algebra grading, \emph{i.e.}, for all $k$ the restriction of the co-multiplication
\begin{equation*}
\ho_k(X) \hookrightarrow \ho(X) \longrightarrow \ho(X) \otimes \ho(X) \quad \text{factors through} \
\bigoplus_{i+j=k}	\ho_{i}(X)\otimes \ho_j(X).
\end{equation*}
The latter is equivalent to $(b)$. We note that $\epsilon_0 : \ho_0(X) \stackrel{\sim}{\longrightarrow} \mathds 1$ is an isomorphism of birational motives if and only if $\varpi_0^X = o_{K(X)}$ with $o = \epsilon_0^{-1} \ \mbox{in}\ \Hom (\mathds 1, \ho(X)) = \CH_0(X)$. In order to conclude, we observe that $\varpi^X_0$ being an idempotent forces $\deg o = 1$, and then that (b) forces $o : \mathds 1 \to \ho(X)$ to be a unit, \emph{i.e.}, to be a co-algebra morphism. 
\end{proof}

In terms of zero-cycles, and working with rational coefficients, Lemma~\ref{lem:zero} and Lemma~\ref{L:concrete-coMBCK} show that a birational Chow--K\"unneth decomposition	$\{\varpi_0^X, \ldots, \varpi_d^X \}$ is co-multiplicative if and only if 
$$\CH_0(X_\Omega)_{(k)} := (\varpi_k^X)_* \CH_0(X_\Omega) = \Hom(\mathds 1_{\Omega}, (X_\Omega,\varpi_k^X))$$
defines a grading on $\CH_0(X_\Omega)$ with the property that
\begin{enumerate}[(a)]
\item $\CH_0(X_\Omega)_{(0)} = \Q o$ for some zero-cycle $o\in \CH_0(X)$ (necessarily of degree 1), and
\item $\CH_0(X_\Omega)_{(k)} \hookrightarrow\CH_0(X_\Omega) \stackrel{\delta}{\to} \CH_0(X_\Omega \times_\Omega X_\Omega) $ factors through $$\mathrm{im} \Bigg(\bigoplus_{k=i+j} \CH_0(X_\Omega)_{(i)} \otimes \CH_0(X_\Omega)_{(j)} \hookrightarrow  \CH_0(X_\Omega) \otimes \CH_0(X_\Omega) \to \CH_0(X_\Omega \times_\Omega X_\Omega) \Bigg) .$$
\end{enumerate}
In Appendix~\ref{S:delta}, we will say that this grading on $\CH_0(X_\Omega)$, induced by a co-multiplicative birational Chow--K\"unneth decomposition, is a $\delta$-grading\,; see Definition~\ref{D:delta-fil} and Proposition~\ref{P:justify}.\medskip
	
	Having a co-multiplicative birational Chow--K\"unneth decomposition is a  stable
	birational invariant among smooth projective varieties. In addition, it is stable under
	product\,; indeed, if $\{\varpi_i^X\}$ and $\{\varpi_j^Y\}$ denote co-multiplicative birational Chow--K\"unneth decompositions for smooth projective varieties $X$ and $Y$ respectively, then it is straightforward to check that $\{\varpi_k^{X\times Y} := \sum_{k=i+j} \varpi_i^X\otimes \varpi_j^Y\}$ defines a co-multiplicative birational Chow--K\"unneth decompositions for the product $X\times Y$. 
	
	Recall that a Chow--K\"unneth decomposition $\{\pi_X^0,\ldots, \pi_X^{2\dim X}\}$ for $X$ is \emph{multiplicative} if the induced decomposition $\h(X) = \h^0(X) \oplus \cdots \oplus \h^{2\dim X}(X)$, $\h^k(X) := \pi^k_X\h(X)$, defines an algebra grading\,; see Definition~\ref{D:MCK} for the definition and \cite{FLV-mck} for an overview.
	We note that if a smooth projective variety $X$ admits a \emph{multiplicative
	Chow--K\"unneth decomposition} $\{\pi_X^0,\ldots, \pi_X^{2\dim X}\}$ with $\pi^{2\dim X}_X = X\times o$ for some zero-cycle $o \in \CH_0(X)$, then the
	birational Chow--K\"unneth decomposition given by $\varpi_i^X :=
	(\pi^{2\dim X-i}_X)|_{k(X)\times X}$ is co-multiplicative.
	For instance, the canonical Chow--K\"unneth decomposition \cite{DM} of an abelian variety is multiplicative and thereby provides a co-multiplicative birational Chow--K\"unneth decomposition.
	 In \cite[Conj.~4]{SV}, it is
	conjectured that all hyper-K\"ahler varieties admit a multiplicative
	Chow--K\"unneth decomposition\,; in particular, it implies
	
	\begin{conj}[Co-multiplicative birational Chow--K\"unneth decomposition for hyper-K\"ahler
		varieties]\label{conj:BMCK}
		Every hyper-K\"ahler variety admits a co-multiplicative birational
		Chow--K\"unneth decomposition.
	\end{conj}
As for multiplicative Chow--K\"unneth decompositions,  co-multiplicative birational
Chow--K\"unneth decompositions may not be unique in general\,: \emph{e.g.}\ for abelian varieties, where any translate of a co-multiplicative birational Chow--K\"unneth decomposition provides a co-multiplicative birational Chow--K\"unneth decomposition. However, in the case of hyper-K\"ahler varieties, we would further expect a co-multiplicative birational
	Chow--K\"unneth decomposition to be \emph{unique}.
	
	\begin{thm}\label{thm:BMCK}
	A smooth projective variety of dimension $2n$ birational to one of the  hyper-K\"ahler varieties \eqref{hilb}, \eqref{moduli}, \eqref{kum} or \eqref{fano}
	 admits a co-multiplicative birational 
	Chow--K\"unneth decomposition $\{\varpi^{X}_{2i} \ \big\vert \ 0\leq i \leq n\}$.
	\end{thm}
	\begin{proof} As explained above, the existence of a co-multiplicative birational Chow--K\"unneth decomposition in cases~\eqref{hilb}
		and~\eqref{kum} follows directly from the existence of a multiplicative Chow--K\"unneth decomposition which
		was previous established in \cite{Vial} in case~\eqref{hilb} and in \cite{ftv}
		in case~\eqref{kum}. 
		Note however that these existence results are dependent on an unpublished result of Voisin~\cite[Thm.~5.12]{VoisinGT}\,; see \cite{NOY} for an independent proof in case~\eqref{hilb}.
		Note also that the vanishing of $\varpi^{2i+1}_X$ and of $\varpi^{d}_X$ for $d>2n$ requires some additional arguments.
		 We therefore provide here a direct proof of \eqref{hilb} and \eqref{kum}, which also exemplifies the fact
		that birational Chow--K\"unneth decompositions are easier to construct and to deal with
		than usual Chow--K\"unneth decompositions.\medskip
		
		Case~\eqref{hilb}. Let us first assume $n=1$, and denote $o$ any point lying on
		a rational curve on $S$. Then we claim that $\{\varpi_0^S := \eta_S\times o,
		\varpi_2^S := \Delta_S|_{\eta_S\times S} - \varpi_0^S\}$ defines a
		co-multiplicative birational Chow--K\"unneth decomposition. With respect to this
		decomposition, we have 
		$$\CH_0(S)_{(0)} = \Q [o] \quad \text{and} \quad \CH_0(S)_{(2)} = \langle\, [p] -
		[o] \ \vert \ p\in S\, \rangle.$$ In order to check that this decomposition is
		co-multiplicative, we have to prove that the cycle $(\delta_S)_* ([p] - [o]) = [(p,p)] -
		[(o,o)]$ belongs to $$\operatorname{im} \Big(\CH_0(S)_{(0)}\otimes \CH_0(S)_{(2)} \oplus 
		\CH_0(S)_{(2)}\otimes \CH_0(S)_{(0)} \longrightarrow \CH_0(S\times S)\Big)$$
		for all points $p\in S$. In fact, for all points $p\in S$, we have 
		$$[(p,p)] - [(o,o)] = \big([(o,p)] - [(o,o)]\big) + \big([(p,o)] - [(o,o)]\big)
		\quad \text{in}\ \CH_0(S\times S),$$ which establishes the claim. This can be
		seen by applying the modified diagonal relation of Beauville--Voisin~\cite{BV}
		to $[p]$ (which is itself equivalent to the fact that $\{\pi^0_S = o\times S,
		\pi^4_S = S\times o, \pi^2_S = \Delta_S - \pi^0_S -\pi_S^4\}$ defines a multiplicative Chow--K\"unneth
		decomposition by \cite[Prop.~8.4]{SV}). More simply, this follows from the
		fact that K3 surfaces are swept out by elliptic curves\,: take $E$ a possibly
		singular elliptic curve in $S$ passing through $p$, then by the
		Bogomolov--Mumford theorem $E$ intersects a rational curve in $S$, that is, $E$
		contains a point $q$ rationally equivalent to $o$ in $S$. But then pushing
		forward the relation $([p]-[q],[p]-[q]) = 0 \ \mbox{in}\ \CH_0(E\times E)$ to
		$\CH_0(S\times S)$ yields the desired relation.
		
		Now, the co-multiplicative birational Chow--K\"unneth decomposition $\{\varpi_0^S,
		\varpi_2^S\}$ provides the co-multiplicative birational Chow--K\"unneth decomposition on
		$S^n$ given by $$\varpi^{S^n}_i  = \sum_{i_1+\cdots+i_n = i} \varpi^S_{i_1}
		\otimes \cdots \otimes \varpi^S_{i_n}.$$ Its symmetrization provides then a
		co-multiplicative birational Chow--K\"unneth decomposition for $\mathrm{Hilb}^n(S)$.
		\medskip

		Case~\eqref{moduli}.  This follows directly from \eqref{hilb} and
		Theorem~\ref{thm:moduli}.\medskip

		Case~\eqref{kum}. 
		The canonical Chow--K\"unneth decomposition of Deninger--Murre~\cite{DM} for the motive of an abelian variety $B$ provide a Chow--K\"unneth decomposition with $\pi^{2g}_{B} = B\times 0_{B}$ where $g=\dim B$. Moreover, the idempotents defining this decomposition can be expressed as rational polynomials of the multiplication-by-$m$ map for integers $m\neq -1,0,1$\,; see \emph{e.g.}~\eqref{E:DM} below.
		As such, these provide a multiplicative Chow--K\"unneth decomposition.
		As in the proof of Theorem~\ref{thm:MSD}\eqref{kum}, we have that $K_n(A)$ is birational to $A_0^{n+1}/\mathfrak{S}_{n+1}$. Identifying $A_0^{n+1}$ with $A^n$,
		 the transpositions $(i,j)$ act by permuting the $i$-th and $j$-th factors of $A^n$ for $i,j\leq n$, while the transposition $(n,n+1)$ acts as $(x_1,\ldots,x_{n-1},x_n) \mapsto (x_1,\ldots,x_{n-1},-\sum_i x_i)$. 
Consider now the Deninger--Murre projectors for the abelian variety $A^n$, and observe that they are $\mathfrak{S}_{n+1}$-invariant since the action of $\mathfrak{S}_{n+1}$ commutes with the multiplication-by-$m$ maps on $A^n$. We therefore obtain a multiplicative Chow--K\"unneth decomposition for $A^n/\mathfrak S_{n+1}$ and thus a co-multiplicative birational Chow--K\"unneth decomposition $\{\varpi_i^{K_n(A)}\}$ for $K_n(A)$. That this decomposition satisfies $\varpi_i^{K_n(A)} =0$ for $i$ odd and for $i>2n$ follows from the fact that $(\varpi_i^{K_n(A)})_*\CH_0(K_n(A))$ coincides with Lin's $\CH_0(K_n(A))_{i}$ of \cite{lin} and that these vanish for $i$ odd~\cite[Thm.~1.4]{lin}  (see also~\cite[Thm.~4.3]{Vial-Pisa})  and for~$i>2n$.
		
		\medskip
		
		Case~\eqref{fano}. Let $\varphi : F \dashrightarrow F$ be Voisin's rational
		self-map~\cite{Voisin-Intrinsic}. It is known \cite{AmVoi} that $\varphi^*\sigma = -2\,\sigma$ for
		$\sigma$ a global two-form on $F$. It was shown in \cite[Thm.~21.9]{SV} that
		the action of $\varphi_*$ on $\CH_0(F)$ diagonalizes with eigenvalues $1$, $-2$
		and $4$. The projectors for this eigenspace decomposition are polynomials in
		$\varphi_*$ and define a birational Chow--K\"unneth decomposition for~$F$. In addition, we have $\pi^8_F = F\times o$, where $o$ is the canonical zero-cycle on $F$ of Voisin~\cite{Voisin-HK}. Since
		$\delta_F \circ \varphi = (\varphi \times \varphi)\circ \delta_F$ as rational
		maps $F \dashrightarrow F\times F$, this decomposition is
		co-multiplicative.
	\end{proof}

\begin{rmk}[Explicit description in case \eqref{hilb}]
	\label{R:hilb}
With respect to the co-multiplicative birational Chow--K\"unneth decomposition constructed in the proof of Theorem~\ref{thm:BMCK}\eqref{hilb}, we can describe explicitly the induced decomposition of $\CH_0(\operatorname{Hilb}^n(S))$, namely
writing $\CH_0(\operatorname{Hilb}^n(S))_{(2k)} := \CH_0(\ho_{2k}(\operatorname{Hilb}^n(S)))$, we have 
$$\CH_0(\operatorname{Hilb}^n(S))_{(2k)} = \langle\, (Z_*o)^{n-k}\cdot \prod_{i=1}^k (Z_*x_i-Z_*o) \ \big\vert \ x_1,\ldots, x_k \in S \, \rangle,$$
where $Z$ is the cycle class of the codimension-2 subset $\{(x,\xi)  \, \big\vert \, x\in  \operatorname{supp}(\xi)   \} \subset S \times \operatorname{Hilb}^n(S)$.
Note that, for any $x_1,\ldots,x_n\in S$, the zero-cycle $(Z_*x_1)\cdots (Z_*x_n)$ is the class of any point with support $\sum_i x_i$\,; sometimes we simply write it $[x_1,\ldots, x_n]$. 
For the induced ascending filtration, we have 
\begin{align*}
G_{k}\CH_0(\operatorname{Hilb}^n(S)) 
&:= \CH_0\big(\ho_0(\operatorname{Hilb}^n(S)) \oplus \cdots \oplus \ho_{2k}(\operatorname{Hilb}^n(S)\big)  \\
&= \langle\,  [x_1,\ldots,x_k,o,\ldots,o] \ \big\vert \ x_1,\ldots ,x_k \in S \, \rangle.
\end{align*}
\end{rmk}

\begin{rmk}[Explicit description in case \eqref{fano}]
	\label{R:fano}
	 By \cite{SV}, we have an explicit description of the decomposition on $\CH_0(F)$ induced by the co-multiplicative birational Chow--K\"unneth decomposition constructed in the proof of Theorem~\ref{thm:BMCK}. Namely, writing $\CH_0(F)_{(2k)} = \CH_0(\ho_{2k}(F))$, we have 
	\begin{align*}
	\CH_0(F)_{(0)} &= \Q o\,; \\
	\CH_0(F)_{(2)} &=  \langle [l]-o \ \big\vert \  l \mbox{ is a line of second type} \rangle\,;\\
	\CH_0(F)_{(4)} &= \langle [l_1]+[l_2]+[l_3]-3o \ \big\vert \ (l_1,l_2,l_3) \mbox{ is a triangle}\rangle.
	\end{align*}
	Here, we say that a line $l$ is \emph{of second type} if there exists a linear $\PP^3$ inside $\PP^5$ that is tangent to the cubic fourfold $Y$ along the line $l$\,; and we say that $(l_1,l_2,l_3)$ forms a \emph{triangle} if there exists a linear $\PP^2$ inside $\PP^5$ such that $Y\cap \PP^2 = l_1 \cup l_2 \cup l_3$. 
	Moreover, the above-defined co-multiplicative birational Chow--K\"unneth decomposition is the decomposition induced by the Chow--K\"unneth decomposition of the Chow motive $\h(F)$ constructed in \cite[\S A.2.1]{FLV-Franchetta} under the $\otimes$-functor sending effective motives to birational motives. 
\end{rmk}
	
\begin{rmk}[Double EPW sextics]
If $X$ is a double EPW sextic, then its anti-symplectic involution $\iota$ provides a birational Chow--K\"unneth decomposition. Indeed, let $o\in \CH_0(X)$ denote a $\iota$-invariant degree-1 zero-cycle on $X$ and define $\pi^8_X := X\times o$, and then define $\pi^6_X$ and $\pi^4_X$ to be the projectors on the $\iota$-invariant and $\iota$-anti-invariant parts of $(X,\Delta_X - \pi^8_X)$ respectively. The restrictions $\varpi_i^X:= \pi_X^{8-i}|_{k(X)\times X}$ of those projectors to $\h^\circ(X)$ then define a birational Chow--K\"unneth decomposition and with respect to this decomposition we have
$$\CH_0(X)_{(0)} = \Q o, \quad \CH_0(X)_{(2)} = (\CH_0(X)_{\mathrm{hom}})^- \quad \mbox{and} \  \CH_0(X)_{(4)} = (\CH_0(X)_{\mathrm{hom}})^+,$$
where the subscript $+$ indicates the $\iota$-invariant part and the subscript $-$ indicates the $\iota$-anti-invariant part of the subgroup of 0-cycles of degree zero.
 In case $X$ is birational to the Hilbert square of a K3 surface, there exists a point $o \in X$ such that this decomposition coincides by \cite[Thm.~3.6]{LV}  with the one of \eqref{hilb}, and is thus co-multiplicative. (In fact, still in case $X$ is birational to the Hilbert square of a K3 surface, Proposition~\ref{P:cogeneration} below shows that any two points $o$ and $o'$ in $X$ whose classes in $\CH_0(X)$ are $\iota$-invariant agree in $\CH_0(X)$.)
\end{rmk}

\section{Strict gradings on birational motives of hyper-K\"ahler varieties}
\label{S:strict}

The aim of this section is to establish that the unital gradings constructed in the proof of Theorem~\ref{thm:BMCK} in cases \eqref{hilb}, \eqref{moduli}, \eqref{kum} and \eqref{fano} define strict gradings as defined in \S \ref{SS:strictgrading}\,;
 see Theorem~\ref{thm:cogeneration}. In \S \ref{SS:cogeneration} below, we start with some general observations concerning co-generation properties of birational motives of smooth projective varieties.

\subsection{On the co-generation of the birational motive of a smooth projective variety}\label{SS:cogeneration}
We have the following general expectation coming from the Bloch--Beilinson philosophy\,:
\begin{conj}\label{C:cogeneration-general}
	Let $X$ be a smooth projective variety and let $i$ be a positive integer. Assume that  $\HH^{*}_{\tr}(X)$ is generated by $\HH^i_{\tr}(X)$. Then there exists a birational idempotent $\varpi_i^X \in \End (\ho(X))$ with $(\varpi_i^X)^*\HH^{*}_{\tr}(X) = \HH^{i}_{\tr}(X)$, and, for any choice of such an idempotent $\varpi_i^X$, the co-algebra object $\ho(X)$ is co-generated by $\ho_i(X):= (X,\varpi^X_i)$, meaning that the morphism
	$$\xymatrix{\ho(X) \ar[r] & \operatorname{Sym}^* \ho_i(X)}$$ co-induced by the split surjection $\ho(X) \twoheadrightarrow \ho_i(X)$ is split injective.
\end{conj}

Very recently, Voisin conjectured 
\cite[Conj.~2.11]{Voisin-lefschetz} that two points $x$ and $y$ on a smooth projective variety $X$ with $\HH^{*}_{\tr}(X)$ generated by $\HH^2_{\tr}(X)$ are rationally equivalent if and only if they have same class in $\CH_0(X)/F^3_{BB}\CH_0(X)$. Here $F^3_{BB}\CH_0(X) := \cap \ker (\Gamma_* : \CH_0(X) \to \CH_0(\Sigma) )$, where the intersection runs through all smooth projective surfaces $\Sigma$ and all correspondences $\Gamma \in \CH^2(X\times \Sigma)$, is an explicit candidate for the third step of the conjectural 
Bloch--Beilinson filtration.
Working instead with the filtration induced by a birational Chow--K\"unneth decomposition (which conjecturally should give the conjectural Bloch--Beilinson filtration), we can relate our expectation on co-generation (Conjecture~\ref{C:cogeneration-general}) to Voisin's expectation\,:

\begin{prop}\label{P:cogeneration}
	Let $X$ be a smooth projective variety. Let $\ho_{\varpi}(X) := (X,\varpi)$ be a direct summand of $\ho(X)$ and assume that $\ho(X)$ is co-generated by $\ho_{\varpi}(X)$.
	If $x$ and $y$ are two points on $X$, then
	$$[x]=[y] \ \mbox{in}\ \CH_0(X) \ \iff \ \varpi_*[x] = \varpi_*[y] \ \mbox{in}\ \CH_0(X).$$
\end{prop}
\begin{proof} First, we note that clearly $[x]=[y]$ implies $\varpi_*[x] = \varpi_*[y] $, irrespective of the co-generation assumption. Under the morphism $\ho(X) \rightarrow \operatorname{Sym}^* \ho_{\varpi}(X)$, the class of a point $x$ is mapped to $1 + (\varpi)_*[x] + (\varpi\otimes \varpi)_*\delta_*[x] + \cdots + (\varpi^{\otimes n})_* \delta^{n-1}_*[x] + \cdots$. Since $\delta^k_*[x] = [x]\times \cdots \times [x] \ \mbox{in}\ \CH_0(X^{k+1})$, we find that
	$$[x] \mapsto 1 + \varpi_*[x]+ \varpi_*[x]\times \varpi_*[x]+ \cdots + \varpi_*[x]\times\cdots \times \varpi_*[x] +\cdots.$$
	Now, under the assumption that  $\ho(X)$ is co-generated by $\ho_{\varpi}(X)$,  the morphism $\ho(X) \rightarrow \operatorname{Sym}^{\leq n} \ho_\varpi(X)$ is split injective for some $n>0$, and it is then apparent that $\varpi_*[x] = \varpi_*[y]$ implies $[x]=[y]$ in $\CH_0(X)$. 
\end{proof}

The following proposition shows that, for smooth projective varieties with transcendental cohomology generated in pure degree (\emph{e.g.}\ hyper-K\"ahler varieties), the existence of a unital grading on the birational motive of a smooth projective variety is equivalent to the existence of a co-multiplicative birational Chow--K\"unneth decomposition\,:

\begin{prop}\label{P:strictgradingBCK}  
	Let $X$ be a smooth projective variety 
	with $\HH^{*}_{\tr}(X)$ generated by $\HH^i_{\tr}(X)$\,; \emph{e.g.}, $X$ a hyper-K\"ahler variety and $i=2$, or $X$ an abelian variety and $i=1$.
	A decomposition $\ho(X) = \ho(X)_{(0)}\oplus \cdots \oplus \ho(X)_{(n)}$ defines  a unital grading on
	the  co-algebra object $\ho(X)$ if and only if, setting $\ho_{ki}(X) := \ho(X)_{(k)}$, the decomposition  
	$\ho(X) = \ho_0(X) \oplus \ho_i(X) \oplus \cdots \oplus \ho_{ni}(X)$ defines a co-multiplicative birational Chow--K\"unneth decomposition.
\end{prop}
\begin{proof}
	Clearly,
	we only need to check that if  $\ho(X) = \ho(X)_{(0)}\oplus \cdots \oplus \ho(X)_{(n)}$ defines a unital grading, then $\ho(X) = \ho_0(X) \oplus \ho_i(X) \oplus \cdots \oplus \ho_{ni}(X)$ with $\ho_{ki}(X) := \ho(X)_{(k)}$ defines a birational Chow--K\"unneth decomposition. 
	Note that a unital grading on $\ho(X)$ defines a unital grading on the co-algebra $\HH^*_{\mathrm{tr}}(X)^\vee$, \emph{i.e.}, $\HH^*_{\mathrm{tr}}(X)^\vee$ is pointed irreducible and graded. 
	By \cite[Lem.~11.2.1]{sweedler}, it is enough to check that $\HH^i_{\mathrm{tr}}(X)^\vee $ consists exactly of the primitive elements of the pointed irreducible co-algebra
	$\HH^*_{\mathrm{tr}}(X)^\vee $, \emph{i.e.}, that$$\HH^i_{\mathrm{tr}}(X)^\vee = \ker (\bar \delta_* :
	\HH^*_{\mathrm{tr}}(X)^\vee \to \HH^*_{\mathrm{tr}}(X)^\vee \otimes \HH^*_{\mathrm{tr}}(X)^\vee),$$
	where $\bar \delta$ is the reduced co-multiplication associated to any choice of a degree-1 zero-cycle on $X$ (see \S \ref{SS:co-mult} and~\eqref{E:it-red-comult}).
	The inclusion $\subseteq$ is clear since $\bar \delta_*$ is graded. Regarding the converse inclusion, still using that $\bar \delta_*$ is graded, it is enough to show that if $\tau^\vee \neq 0 \ \mbox{in}\ \HH^{ki}_{\mathrm{tr}}(X)^\vee$ for some $k>1$, then $\bar \delta_*\tau^\vee \neq 0$. By the generation assumption, we may write $\tau = \sum \sigma_1\cup \cdots \cup \sigma_k$ for some $\sigma_r$ in $\HH^{i}_{\mathrm{tr}}(X)$. But then 
	$({\delta}_*\tau^\vee)(\sum (\sigma_1\cup \cdots\cup \sigma_{k-1}) \times \sigma_k)
	= \tau^\vee({\delta}^*(\sum (\sigma_1\cup \cdots\cup \sigma_{k-1}) \times \sigma_k)) = \tau^\vee (\tau)  = 1$, where the first equality comes from Poincar\'e duality. Hence, ${\delta}_*\tau^\vee$ has a non-zero component of bi-degree $((k-1)i, i)$, and it follows that ${\bar\delta}_*\tau^\vee \neq 0$.
\end{proof}

\subsection{On the co-generation of the birational motive of a hyper-K\"ahler variety} \label{SS:cogeneration-hK}
Let now $X$ be a hyper-K\"ahler variety of dimension~$2n$. Assuming the generalized Hodge conjecture for~$X$, the Hodge structures $\HH^{2k}_{\tr}(X)$ are generated by $\sigma^k$ for a generator $\sigma$ of $\HH^0(X,\Omega^2_X)$ and we also have $\HH^{2k+1}_{\tr}(X)=0$ for all $k$. Thus, conjecturally, cup-product induces a surjection, with kernel supported in codimension 1\,:
$$\xymatrix{\operatorname{Sym}^{\leq n}\HH^2_{\tr}(X) \ar@{->>}[r] & \HH^{*}_{\tr}(X).}$$
Combining Conjecture~\ref{conj:BMCK} with Conjecture~\ref{C:cogeneration-general}, and taking into account that the kernel is supported in codimension 1, suggests the following (slight) strengthening of Conjecture~\ref{conj2:MBCK-bis}\,:

\begin{conj}\label{C:cogeneration}
	Let $X$ be a hyper-K\"ahler variety of dimension $2n$ and assume $\ho(X)$ admits a co-multiplicative birational Chow--K\"unneth decomposition
	 $\ho(X) = \ho_0(X) \oplus \ho_2(X) \oplus \cdots \oplus \ho_{2n}(X)$. Then this unital grading  is a strict grading. More strongly, the graded co-algebra morphism 
	$$\xymatrix{\ho(X) \ar[r]^{\sim\qquad} & \operatorname{Sym}^{\leq n} \ho_2(X)}$$ co-induced by the graded split surjection $\ho(X) \twoheadrightarrow \ho_2(X)$ is an isomorphism.
\end{conj}

	\begin{thm}\label{thm:cogeneration}
	The birational motives of the hyper-K\"ahler varieties \eqref{hilb}, \eqref{moduli}, \eqref{kum} and \eqref{fano}, equipped with the co-multiplicative birational Chow--K\"unneth decomposition provided by Theorem~\ref{thm:BMCK}, satisfy the conclusion of Conjecture~\ref{C:cogeneration}.
\end{thm}
\begin{proof}
	In the general situation where $M = M_{(0)} \oplus \cdots \oplus M_{(n)}$ is a unital graded co-algebra object in a $\otimes$-category $\mathcal C$ over a ring $R$, recall that the induced co-algebra morphism $M\to T^* M_{(1)}$ is graded and that the composition $M_{(k)} \hookrightarrow M \to M^{\otimes k} \to (M_{(1)})^{\otimes k}$ is nothing but the restriction of $\bar \delta^{k-1}$ to $M_{(k)}$\,; see \S \ref{S:co-alg}. 
	In addition, if $R$ contains $\Q$ and if $\mathcal  C$ is pseudo-abelian, then the resulting graded morphism of co-algebra objects $M \to T^{\leq n}M_{(1)}$ factors through $\operatorname{Sym}^{\leq n}M_{(1)}$.
	Therefore, in order to show that $M \to \operatorname{Sym}^{\leq n}M_{(1)}$ is an isomorphism of graded co-algebras, it is enough to produce inverses to the morphisms $\bar \delta^{k-1} : M_{(k)} \to \operatorname{Sym}^{k}M_{(1)}$ for all $k>1$.
	
	In the case of a hyper-K\"ahler variety $X$ with a co-multiplicative birational Chow--K\"unneth decomposition
	$\ho(X) = \ho_0(X) \oplus \ho_2(X) \oplus \cdots \oplus \ho_{2n}(X)$, it thus suffices to produce for all $k>1$ morphisms $\frac{1}{k!}\mu^k :  \operatorname{Sym}^{k}\ho_2(X) \to \ho_{2k}(X)$
	 inverse to $\bar \delta^{k-1} : \ho_{2k}(X) \to \operatorname{Sym}^{k}\ho_2(X) $, or, equivalently by Lemma~\ref{lem:zero}, such that $\mu^k_* \circ \bar{\delta}^{k-1}_*=k! \,\mathrm{id} : \CH_0(X)_{(2k)} \to \CH_0(X)_{(2k)}$ and $\bar{\delta}^{k-1}_* \circ \mu^k_*=k! \,\mathrm{id}  : \CH_0( \operatorname{Sym}^{k}\ho_2(X)) \to  \CH_0(\operatorname{Sym}^{k}\ho_2(X))$. Here the correspondence $\bar \delta^{k-1}$ is made explicit in~\eqref{E:it-red-comult}.
	 (The reason for introducing the factor $k!$ lies in \eqref{E:Pontryagin} -- this is also related to the notion of divided power Hopf algebra, although $\delta$ and $\mu$ do not endow $\operatorname{Sym}^{\leq n} \ho_2(X)$ with the structure of a bi-algebra.)
	 \medskip
	
Case \eqref{hilb}. We take on the notation of Remark~\ref{R:hilb}. 
We start by recalling that, defining  $\CH_0(\operatorname{Hilb}^n(S))_{(2k)} := \CH_0(\ho_{2k}(\operatorname{Hilb}^n(S)))$, we have
$$\CH_0(\operatorname{Hilb}^n(S))_{(2k)} = \langle\, (Z_*o)^{n-k}\cdot \prod_{i=1}^k (Z_*x_i-Z_*o) \ \big\vert \ x_1,\ldots, x_k \in S \, \rangle,$$ which in case $k=1$ takes the simple form $$\CH_0(\operatorname{Hilb}^n(S))_{(2)}  = \langle [x,o,\ldots,o]-[o,o,\ldots,o] \ \big\vert \ x\in S\rangle.$$
We also note that the idempotent $\varpi_2$ cutting $\ho_2(\operatorname{Hilb}^n(S)))$ acts explicitly on $\CH_0(\operatorname{Hilb}^n(S))$ by $$(\varpi_2)_*[x_1,\ldots, x_n] = \sum_{i=1}^n\Big( [x_i,o,\ldots,o] - [o,\ldots, o] \Big).$$
On the one hand, we define 
$$\mu^k : \ho_2(\operatorname{Hilb}^n(S))^{\otimes k} \to \ho_{2k}(\operatorname{Hilb}^n(S))$$ as the birational correspondence sending 
$$a_1\times \cdots\times a_k \ \mapsto \ (Z_*o)^{n-k} \cdot \prod_{i=1}^k Z_*a_i,$$ where $a_i:=[x_i,o,\ldots,o]-[o,o,\ldots,o] \in \CH_0(\operatorname{Hilb}^n(S))_{(2)}$. 
Clearly, $\mu^k$ is invariant under the action of the symmetric group $\mathfrak{S}_k$ and thereby provides a morphism 
$$\mu^k : \operatorname{Sym}^k\ho_2(\operatorname{Hilb}^n(S)) \to \ho_{2k}(\operatorname{Hilb}^n(S)).$$
On the other hand,
 we have $\bar \delta^{k-1} [x_1,\ldots, x_n] = ([x_1,\ldots,x_n] - [o,\ldots,o])^{\times k}$ (see \eqref{eq:redcomult} below). Now since $\bar \delta^{k-1}$ maps $\ho_{2k}(\operatorname{Hilb}^n(S))$ into the direct summand $\operatorname{Sym}^k\ho_2(\operatorname{Hilb}^n(S))$, we find that 
\begin{align*}
\bar \delta^{k-1} [x_1,\ldots, x_k,o,\ldots,o] &= \big((\varpi_2)_*([x_1,\ldots, x_k,o,\ldots,o] - [o,\ldots,o])  \big)^{\times k}\\
&= \big(a_1 +\cdots + a_k\big)^{\times k}\\
&= \sum_{\sigma \in \mathfrak{S}_k} a_{\sigma(1)} \times \cdots \times a_{\sigma(k)},
\end{align*}
where again $a_i := [x_i,o,\ldots,o] - [o,\ldots,o]$ and where the last equality is due to $(a_i)^{\times 2}=0$ which itself follows from the fact (see the proof of Theorem~\ref{thm:BMCK}\eqref{hilb}) that $([x_i]-o)\times ([x_i] - o) = 0$ in $\CH_0(S\times S)$.
It is then apparent that $\frac{1}{k!}\mu^k$ provides an inverse to $\bar \delta^{k-1}$. 
\medskip

Case \eqref{moduli} follows directly from case \eqref{hilb} and Theorem~\ref{thm:moduli}.\medskip

Case \eqref{kum}. 
Recall from the proof of Theorem~\ref{thm:MSD} that the hyper-K\"ahler variety $K_n(A)$ is birational to $A^n/\mathfrak S_{n+1}$ so that it is enough to establish the theorem for $A^n/\mathfrak S_{n+1}$. First we note that, for $k\leq n$, cup-product 
$$(\delta^{k-1})^*\ :\ \operatorname{Sym}^k \HH^2(A^n/\mathfrak S_{n+1}) \longrightarrow \HH^{2k}(A^n/\mathfrak{S}_{n+1})$$
is a morphism of Hodge structures that is an isomorphism on the degree $(2k,0)$ part of the Hodge decomposition. Second, after fixing a polarization on $A$, $\delta^{k-1}$ is generically defined for powers of~$A$ in the sense of \cite[Def.~2.1]{Vial-Pisa}. In addition, by \cite[Prop.~2.13]{Vial-Pisa}, the orthogonal projectors on the sub-Hodge structures of $\operatorname{Sym}^k \HH^2(A^n/\mathfrak S_{n+1})$ and $\HH^{2k}(A^n/\mathfrak{S}_{n+1})$ generated by forms of degree $(2k,0)$ are induced by generically defined cycles $p$ and $q$ respectively. Third, since the Hard Lefschetz isomorphisms $\HH^r(A^n) \stackrel{\sim}{\longrightarrow} \HH^{4n-r}(A^n)$ and their inverses are induced by generically defined cycles for powers of $A$, we see (as in the proof of \cite[Prop.~2.13]{Vial-Pisa}) that the inverse of the isomorphism 
$$q^*(\delta^{k-1})^*p^*\ : \ p^* \operatorname{Sym}^k \HH^2(A^n/\mathfrak S_{n+1}) \stackrel{\sim}{\longrightarrow} q^*\HH^{2k}(A^n/\mathfrak{S}_{n+1})$$
is induced by a generically defined cycle.
We may then conclude from \cite[Thm.~1]{Vial-Pisa} that ${\delta}^{k-1}$ induces an isomorphism on zero-cycles $ \CH_0(\h_{2k}(A^n/\mathfrak S_{n+1})) \stackrel{\sim}{\longrightarrow} \CH_0(\operatorname{Sym}^k\h_2 (A^n/\mathfrak S_{n+1}))$ with inverse induced by a correspondence, where $\h_i(A^n) := \h^{4n-i}(A^n)$ is the Deninger--Murre Chow--K\"unneth decomposition of $A^n$.  We then conclude from Lemma~\ref{lem:zero} that $$\bar \delta^{k-1} : \ho_{2k}(A^n/\mathfrak S_{n+1}) \stackrel{\delta^{k-1}}{\longrightarrow } \operatorname{Sym}^k\ho (A^n/\mathfrak S_{n+1}) \to \operatorname{Sym}^k\ho_2 (A^n/\mathfrak S_{n+1})$$ is an isomorphism.
\medskip

Case \eqref{fano}. Let $F:=F(Y)$ be the Fano variety of lines on a smooth cubic fourfold $Y\subseteq \PP^5$. 
We consider the self-dual Chow--K\"unneth decomposition $\h(F) = \h^0(F) \oplus \cdots \oplus \h^8(F)$ of \cite[\S A.2.1]{FLV-Franchetta} (which yields the co-multiplicative birational Chow--K\"unneth decomposition of Theorem~\ref{thm:BMCK}\eqref{fano}, see Remark~\ref{R:fano})\,; 
it is \emph{generically defined} in the sense that the idempotents $\pi^i$ defining $\h^i(F)$ are specializations of cycles on $\mathcal F \times_B \mathcal F$, where $\mathcal F \to B$ is the relative Fano variety of the universal smooth cubic fourfold.
 The multiplication morphism $\delta^* : \operatorname{Sym}^2\h^2(F) \to\h^4(F)$ is an isomorphism by \cite[Thm.~2.18(v)]{FLV-Franchetta2}. Dualizing, we obtain an isomorphism $\delta_* : \h^4(F) \to \operatorname{Sym}^2\h^6(F)$, but then $\bar \delta : \ho_4(F) \to \operatorname{Sym}^2\ho_2(F)$ is nothing but the image of $\delta_*$ under the functor $\mathcal M^{\mathrm{eff}} \to \mathcal M^\circ$ and is thus an isomorphism. 
 
 Let us however give an alternate proof that determines the inverse of $\bar \delta$.
For that purpose, we consider the cycle $L\in \CH^2(F\times F)$ of \cite[Part~3]{SV}\,; it is generically defined, its cohomology class is the Beauville--Bogomolov--Fujiki form and its action on points $l\in F(Y)$ is given by $L_*l = [S_o]-[S_l]$, where $S_l$ is the surface of lines meeting the line $l$.
We claim that the morphism 
$$\frac{1}{2}\mu\ : \ \xymatrix{\operatorname{Sym}^2\h^6(F) \ar[rr]^{\frac{1}{2}\Sym^2L} & &\operatorname{Sym}^2\h^2(F) \ar[rr]^{\quad \delta^*} &&\h^4(F)}$$ is an isomorphism of Chow motives, and that the inverse of the induced isomorphism $$\frac{1}{2}\mu : \operatorname{Sym}^2\ho_2(F) \stackrel{\sim}{\longrightarrow} \ho_4(F)$$ on birational motives is nothing but the reduced co-multiplication $\bar \delta : \ho_4(F) \to  \operatorname{Sym}^2\ho_2(F)$.
First, $\mu$ is indeed an isomorphism of Chow motives since as recalled above the morphism $\delta^* : \operatorname{Sym}^2\h^2(F) \to\h^4(F)$ is known to be an isomorphism, and since the morphism $L : \h^6(F) \to \h^2(F)$ is an isomorphism due to the fact that it is an isomorphism modulo homological equivalence by \cite[Prop.~1.3]{SV} (its inverse is given by $\frac{1}{75}L^3$) combined with the generalized Franchetta conjecture for $F\times F$~\cite[Thm.~1.10]{FLV-Franchetta}.

Second, 
 we claim that $\mu$ satisfies $\mu \circ \bar \delta = 2\,\mathrm{id}$ on $\CH_0(F)_{(4)}$, which by Lemma~\ref{lem:zero} implies that $\bar \delta$ is the inverse of the isomorphism $\frac{1}{2}\mu$. 
 Recall from \cite[\S 20]{SV} that $\CH_0(F)_{(4)}$ is killed by $L_*$ and  is spanned by cycles of the form $[l_1]+[l_2]+[l_3]-3[o]$, where $(l_1,l_2,l_3)$ form a triangle, \emph{i.e.}, there is a plane $\Pi \subset \PP^5 $ such that $\Pi\cap Y = l_1 + l_2 + l_3$. 
 With $\sum_i [S_{l_i}] = 3[S_o]$ and $\bar \delta_* [l] = [(l,l)] - [(l,o)] - [(o,l)]$ in mind, we now compute 
 \begin{align*}
 	\mu\circ \bar \delta ([l_1]+[l_2]+[l_3]-3[o]) & = \sum_i ([S_o] - [S_{l_i}])^2\\
&=   3[S_o]^2 - 2[S_o]\cdot \sum_i [S_{l_i}] + \big( \sum_i [S_{l_i}]\big)^2 - 2 \sum_{i<j} [S_{l_i}]\cdot [S_{l_j}]\\	
& = 6[S_o]^2 -  2 \sum_{i<j} [S_{l_i}]\cdot [S_{l_j}].
 \end{align*}
Now, by \cite[Prop.~20.7(i)]{SV}, we have $[S_o]^2 = 5[o]$ and for a triangle $(l_1,l_2,l_3)$ we have $[S_{l_i}]\cdot [S_{l_j}] = 6[o] + [l_k] - [l_i] - [l_j]$ for $\{i,j,k\} = \{1,2,3\}$. It follows that 
$$	\mu\circ \bar \delta ([l_1]+[l_2]+[l_3]-3[o])  = 2([l_1]+[l_2]+[l_3]-3[o]),$$
as claimed.
\end{proof}

\begin{rmk}
Theorem~\ref{thm:cogeneration} and its proof echoes \cite[Conj.~2]{SV}, where it is conjectured that for any hyper-K\"ahler variety $X$ of dimension $2n$ there exists a canonical cycle $L\in \CH^2X\times X)$ with cohomology class the Beauville--Bogomolov--Fujiki form such that $\CH_0(X)$ admits a grading $$\CH_0(X)=\CH_0(X)_{(0)} \oplus \CH_0(X)_{(2)} \oplus \cdots \oplus \CH_0(X)_{(2n)}$$ with 
$\CH_0(X)_{(2k)} = l^{n-k}\cdot (L_*\CH_0(X))^{\cdot k}$, where $l := \delta^* L \in \CH^2(X)$.
\end{rmk}

\section{The co-radical filtration on zero-cycles}\label{S:corad}
In this section, we define explicitly the co-radical filtration on the Chow group of zero-cycles on a smooth projective variety $X$ equipped with a unit $o\in\CH_0(X)$, and we prove Proposition~\ref{P:smash-hyperK} (in cases \eqref{hilb}, \eqref{moduli}, \eqref{kum}, \eqref{fano} and \eqref{llsvs}) and Proposition~\ref{P:co-rad} of the introduction. In addition, we discuss why, in the case of hyper-K\"ahler varieties, the co-radical filtration is expected to be opposite to the conjectural Bloch--Beilinson filtration.

\subsection{The co-radical filtration}
By considering the $\Q$-linear $\otimes$-category of birational motives (or of covariant effective motives), one defines as in \S \ref{SS:co-rad} the co-radical filtration on $\CH_0(X)=\Hom(\mathds 1, \ho(X))$ for all smooth projective varieties $X$ equipped with a unit $o:\mathds 1 \to \ho(X)$. In concrete terms, we have

\begin{defn}[Co-radical filtration on $\CH_0$] \label{D:co-rad}
	Let $X$ be a smooth projective variety over a field~$K$. 
	Fix  a zero-cycle $o\in \CH_0(X)$ of degree 1 such that $\delta_*o=(o,o)$, \emph{i.e.}, a unit $o:\mathds 1 \to \ho(X)$ of the co-algebra object $\ho(X)$ in the category of birational motives. We define the ascending \emph{co-radical filtration} (associated to the unit $o$)
	 $R_\bullet$ on the $\Q$-vector space $\CH_0(X)$ by 
	$$R_k\CH_0(X) :=  \ker \big(\bar \delta^{k}_* : \CH_0(X) \to \CH_0(X^{k+1}) \big).$$
	Here, $\bar \delta^k$ is the iterated reduced co-multiplication defined in \S \ref{SS:co-mult}\,: $\bar \delta^0 := \mathrm{id} - o\epsilon$, $\bar \delta = (\delta - o\times \mathrm{id} - \mathrm{id}\times o)\circ \bar \delta^0$ and $\bar \delta^k = (\bar \delta \otimes \mathrm{id} \otimes \cdots \otimes \mathrm{id})\circ \bar \delta^{k-1}$\,; in our setting $\bar \delta^k$ is explicitly given by the formula~\eqref{E:it-red-comult}.
	We also define the ascending filtration on $K$-points of $X$\,:
	$$R_k(X) := \{x\in X(K) \ \big\vert\ [x] \in R_k\CH_0(X)\}.$$
\end{defn}

Note that the co-radical filtration $R_\bullet$ depends on the choice of unit $o \in \CH_0(X)$ since, \emph{e.g.},
$R_0\CH_0(X) = \Q o$, and that a priori we only have $\langle [x] \ \big\vert \ x\in R_k(X)\rangle \subseteq R_k\CH_0(X)$.
Proposition~\ref{P:RT} below gives an explicit description of $R_k(X)$. It also shows
 that, in case $K$ is algebraically closed, due to the fact that any degree-0 zero-cycle is smash-nilpotent~\cite{Voevodsky, Voisin-smash}, the filtration $R_\bullet(X)$ is exhaustive, \emph{i.e.}, we have $X(K) =  \bigcup_{k\geq 0} R_k(X)$ (and so $\CH_0(X) = \bigcup_{k\geq 0} R_k\CH_0(X)$). 
That $R_\bullet\CH_0(X)$ is exhaustive can also be seen from the vanishing of $\bar \delta^n$ for $n$ large\,; see  \S\ref{SS:mod}.

\begin{prop}\label{P:RT}
	Let $X$ be a smooth projective variety over a field~$K$ and let  $o\in \CH_0(X)$ be a unit. Let  $R_\bullet$ be the co-radical filtration associated to $o$. Then for all $k$ we have 
	$$R_k (X) = \{x\in X(K) \ \big\vert \ ([x]-o)^{\times k+1}=0\ \mbox{in}\ \CH_0(X^{k+1})  \}.$$
\end{prop}
\begin{proof} The case $k=0$ is clear. It is then enough to show that for the reduced co-multiplication $\bar \delta$ associated to the unit $o: \mathds 1 \to \ho(X), 1\mapsto o$ (as in Definition~\ref{D:co-rad}), we have for all units  $x\in \CH_0(X)$ and all integers $k> 0$, 
	\begin{equation}\label{eq:redcomult}
	\bar{\delta}^{  k}(x-o) = (x-o)^{\times k+1}.
	\end{equation}
	 First we have  $\bar{\delta} (x-o) = (x,x) - (o,o) - (x-o,o) - (o,x-o) = (x-o,x-o).$ By induction, we get 
	\begin{align*}
	\bar{\delta}^{  k+1} (x-o) & = (\bar{\delta}\otimes \mathrm{id}^{\otimes k}) \bar{\delta}^{  k}(x-o)= (\bar{\delta}\otimes \mathrm{id}^{\otimes k})\big((x-o)^{\times k+1}\big)
	 = (\bar{\delta}(x-o), (x-o)^{\times k})\\
	 &\quad = ((x-o)^{\times 2},(x-o)^{\times k})= (x-o)^{\times k+2},
	\end{align*}
	thereby establishing \eqref{eq:redcomult}.
\end{proof}

\subsection{Proof of Proposition~\ref{P:co-rad}} \label{SS:proofTco-rad}
This is an application of Proposition~\ref{P:crucial}\,: one considers the $\Q$-linear $\otimes$-category of birational motives (see \S \ref{SS:biratmot}), the functor $\operatorname{C}=\Hom(\mathds 1,-)$ and the object $M=\ho(X)$.
\qed

\subsection{Proof of Proposition~\ref{P:smash-hyperK} in cases \eqref{hilb}, \eqref{moduli}, \eqref{kum}, \eqref{fano} and \eqref{llsvs}}\label{SS:proof1}
In those cases, the hyper-K\"ahler variety $X$ admits a co-multiplicative birational Chow--K\"unneth decomposition, \emph{i.e.},  a unital grading $\ho(X) =\ho_0(X) \oplus \cdots \oplus \ho_n(X)$. We can then conclude directly from Proposition~\ref{P:co-rad} that $R_n\CH_0(X) = \CH_0(X)$ and hence that $X$ satisfies the conclusion of Proposition~\ref{P:smash-hyperK}.

In case \eqref{llsvs}, we need the following basic fact.
	 Suppose $X$ is a smooth projective variety and assume there is a dominant rational map $\pi:Y_1\times \cdots \times Y_m \dashrightarrow X$, where $Y_1,\ldots,Y_m$ are smooth projective varieties the birational motives of which admit unital gradings $\ho(Y_i) = \ho(Y_i)_{(0)} \oplus \cdots \oplus \ho(Y_i)_{(n_i)}$. Then there exists a unit $o\in \CH_0(X)$ such that for all $x\in X(K)$ we have $([x]-0)^{\times n+1}=0\ \mbox{in}\ \CH_0(X^{n+1})$, where $n = n_1+\cdots + n_m$.
Indeed, this follows at once from the fact that the natural grading on $\ho(Y_1\times \cdots \times Y_m) = \ho(Y_1)\otimes \cdots \otimes \ho(Y_m)$ is a unital grading, Proposition~\ref{P:co-rad} and Proposition~\ref{P:RT}, and pushing forward along $\pi$.

The above in particular establishes Proposition~\ref{P:smash-hyperK} in the case~\eqref{llsvs} of LLSvS eightfolds.
Indeed, if $Z$ is the LLSvS eightfold associated to a cubic fourfold $Y$, Voisin~\cite[Prop.~4.8]{Voisin-coisotropic} has constructed a dominant rational map $\psi : F\times F \dashrightarrow Z$, where $F$ is the Fano variety of lines on~$Y$. We conclude with Theorem~\ref{thm:BMCK} where a unital grading for $\ho(F)$ was constructed.\qed

\subsection{The co-radical filtration and the conjectural Bloch--Beilinson filtration for hyper-K\"ahler varieties}	\label{R:oppositeBB}
	Since for a hyper-K\"ahler variety $X$ the Hodge numbers $h^{i,0}(X)$ vanish for $i$ odd, the Bloch--Beilinson filtration, if it exists, satisfies $F^{2i-1}\CH_0(X) = F^{2i}\CH_0(X)$ for all $i$. 
	An increasing filtration $G_\bullet$ on $\CH_0(X)$ is then said to be opposite to $F^{2\bullet}$ if the composition 
	$$G_i \CH_0(X) \hookrightarrow \CH_0(X) \twoheadrightarrow \CH_0(X)/F^{2i+2} \CH_0(X)$$ is bijective.
	Now, since the conjectural Bloch--Beilinson filtration $F^\bullet$, if it exists, is induced
	by the choice of any Chow--K\"unneth decomposition (see \eqref{E:BB} and the ensuing discussion),
	we note from Proposition~\ref{P:strictgradingBCK} that 
	the existence of a unital grading  $\ho(X) = \ho(X)_{(0)}\oplus \cdots \oplus \ho(X)_{(n)}$ provides an ascending filtration 
	$$G_k \CH_0(X) := \CH_0\big(\ho(X)_{(0)}\oplus \cdots \oplus \ho(X)_{(k)}\big) =\CH_0\big(\ho(X)_{(0)}\big) \oplus \cdots \oplus \CH_0\big(\ho(X)_{(k)}\big)$$
	opposite to the conjectural Bloch--Beilinson filtration.
	If the unital grading is in addition a strict grading (and such strict gradings exist for  hyper-K\"ahler varieties of type \eqref{hilb}, \eqref{moduli}, \eqref{kum} or \eqref{fano} by Theorem~\ref{thm:cogeneration})), we get from Proposition~\ref{P:co-rad}  that the co-radical filtration $R_\bullet$ on $\CH_0(X)$ agrees with $G_\bullet$ and hence is opposite to the conjectural Bloch--Beilinson filtration $F^{2\bullet}$.
	It is thus natural to conjecture that for all hyper-K\"ahler varieties, there exists a unit such that the associated co-radical filtration is opposite to the conjectural Bloch--Beilinson filtration.

\section{The filtrations of Voisin and Shen--Yin--Zhao and the co-radical filtration}\label{SS:VoisinFil}

Let $X$ be a smooth projective variety over an algebraically closed field $k$. Given a closed point $x\in X$,  the orbit of $x$ under rational equivalence
$$O_x := \{ x' \in X(k) \ \big\vert \ [x]=[x'] \ \mbox{in}\ \CH_0(X)\}$$
is a countable union of closed algebraic subsets of $X$ and we denote $\dim O_x$ the maximal dimension of these subvarieties. 
Although the following definition makes sense for any smooth projective variety $X$, it is of particular relevance for hyper-K\"ahler varieties.

\begin{defn}[Voisin filtration \cite{Voisin-coisotropic}] 
	Let $X$ be a hyper-K\"ahler variety of dimension $2n$. 		
	The \emph{Voisin filtration}~$S_\bullet$ is the increasing filtration on $\CH_0(X)$\,:
	$$S_0\CH_0(X) \subseteq S_1\CH_0(X)\subseteq \cdots \subseteq S_n\CH_0(X) = \CH_0(X)$$ 
	defined by 
	$$S_k\CH_0(X) := \langle\ [x] \ \big\vert \ x\in S_k(X) \rangle,$$
	where 
	$$S_k(X) :=		\{x\in X \ \big\vert \ \dim O_x \geq n-k\}.$$
\end{defn}
We note that $S_{-1}(X) = \varnothing$ (and consequently that $S_{-1}\CH_0(X) = 0$)\,: indeed if $f: Z \to X$ is a codimension-0 morphism with $Z$ smooth projective such that the image of $f_*:\CH_0(Z) \to \CH_0(X)$ is one-dimensional, then by Mumford's theorem (or Bloch--Srinivas~\cite{BS}) we find that $f^*\sigma = 0$, where $\sigma$ is a nowhere degenerate symplectic form on $X$, and hence that $\dim Z \leq n$\,; see \cite[Cor.~1.2]{Voisin-coisotropic}.

In addition, by \cite[Lem.~3.10(ii)]{Voisin-coisotropic}, if $X$ contains a constant-cycle Lagrangian subvariety~$Z$ which is connected and whose class is a linear combination of $l^n$, $l^{n-2}c_{\tr},\ldots$ for some ample divisor~$l$, where $c_{\tr}$ is the transcendental part of the Beauville--Bogomolov form (see \cite[\S 1.1]{Voisin-coisotropic} for more details), then $S_0\CH_0(X)$ is spanned by the class of any point $o$ on $Z$. In general, it is expected that $S_0\CH_0(X)$ is spanned by the class of a point $o$ and when this is the case we call $o$ a \emph{Beauville--Voisin point} of $X$.

\begin{conj}\label{conj:Voisin-split-cogen}
	Let $X$ be a hyper-K\"ahler variety. Then there exists a point $o\in X$ such that for all $k\geq 0$
	\begin{equation}\label{E:containment}
 S_k\CH_0(X) = R_k\CH_0(X) = \langle\ [x] \ \big\vert \ ([x]-[o])^{\times k+1} = 0 \ \mbox{in}\ \CH_0(X^{k+1}) \ \rangle .
	\end{equation}
\end{conj}
We insist here that the right-hand side equality of~\eqref{E:containment} is a consequence of the left-hand side equality. 
Indeed, since by definition $S_k\CH_0(X)$ is spanned by classes of points, it follows that $R_k\CH_0(X)$ is also spanned by classes of points, hence the right-hand side equality in~\eqref{E:containment} by Proposition~\ref{P:RT}. We may thus ask whether the property that $R_k\CH_0(X)$ is spanned by classes of points is specific to hyper-K\"ahler varieties\,? (See Remark~\ref{rmk:co-radical-abelian} for the case of abelian varieties.)
A related question, in the hyper-K\"ahler setting, is whether a point $x\in X$ whose class belongs to $S_k\CH_0(X)$ satisfies $x\in S_k(X)$, \emph{i.e.}, $\dim O_x \geq n-k$\,? 
(Compare with the corresponding question concerned with the Shen--Yin--Zhao filtration raised in \cite[Ques.~3.2]{SYZ}\,; see \S \ref{SSS:moduli}.)
\medskip

The following theorem gives evidence for the above Conjecture~\ref{conj:Voisin-split-cogen}.

\begin{thm}\label{thm:splitting}
	Let $X$ be a smooth projective variety birational to one of the hyper-K\"ahler varieties \eqref{hilb}, \eqref{kum} or \eqref{fano}.
Then Conjecture~\ref{conj:Voisin-split-cogen} holds for $X$\,; in particular, denoting $o$ the class of a point spanning $S_0\CH_0(X)$, 
 for all $x\in X$ we have
$$[x]\in S_k\CH_0(X) \ \iff \ ([x]-o)^{\times k+1} = 0 \ \mbox{in}\ \CH_0(X^{k+1}).$$
\end{thm}
\begin{proof} 
By Proposition~\ref{P:co-rad}, together with Theorem~\ref{thm:cogeneration}, in order to prove the theorem, it suffices to show in each case that $$S_k \operatorname{CH}_0(X) = \Big(\bigoplus_{i\leq k} \varpi^X_{2i} \Big)_* \CH_0(X) \quad \subseteq \CH_0(X),$$
	where $\{\varpi_i^X\}$ is the co-multiplicative birational Chow--K\"unneth decomposition from Theorem~\ref{thm:BMCK}.
	
	In case \eqref{hilb}, this is \cite[Thm.~2.5]{Voisin-coisotropic} together with the explicit  description of the co-multiplicative birational Chow--K\"unneth decomposition in the proof of Theorem~\ref{thm:BMCK}\eqref{hilb}\,; in that case we simply have, after identifying $\CH_0(\operatorname{Hilb}^n(S))$ with $\CH_0(S^{(n)})$, that $S_k\CH_0(\operatorname{Hilb}^n(S))$ is spanned by the classes of points $x_1+\cdots+x_k+(n-k)o$.

	The case \eqref{kum} is due to Lin~\cite{lin}. (Note that in this case our $( \varpi^X_{2i})_* \CH_0(X)$ coincides with Lin's $\CH_0(X)_{2i}$.)
	
	The case \eqref{fano} is \cite[Prop.~4.5]{Voisin-coisotropic} (together with the fact \cite[Thm.~21.9]{SV} that in this case our decomposition $\CH_0(F)$ into eigenspaces for the action of $\varphi_*$ (Theorem~\ref{thm:BMCK}) coincides with the Shen--Vial decomposition).
\end{proof}

In case~\eqref{moduli}, one can formulate a version of Theorem~\ref{thm:splitting}  by replacing Voisin's filtration $S_\bullet$ with the filtration~\eqref{E:SYZ-def} introduced by Shen--Yin--Zhao~\cite{SYZ}\,:

\begin{thm}\label{thm:splitting-moduli}
	Let $X$ be a smooth projective variety birational to a moduli space $\operatorname{M}_\sigma(v)$ of stable objects on a K3 surface.
	Then, denoting $o$ the degree-1 zero-cycle spanning $S^{\mathrm{SYZ}}_0\CH_0(\operatorname{M}_\sigma(v))$, we have 
	$$S^{\mathrm{SYZ}}_k\CH_0(\operatorname{M}_\sigma(v)) = R_k\CH_0(\operatorname{M}_\sigma(v)).$$
	In particular,
	for all $x\in X$ we have
	$$[x]\in S^{\mathrm{SYZ}}_k\CH_0(\operatorname{M}_\sigma(v)) \ \iff \ ([x]-o)^{\times k+1} = 0 \ \mbox{in}\ \CH_0(\operatorname{M}_\sigma(v)^{k+1}).$$
\end{thm}
\begin{proof}
	We note that the isomorphism $(R_0)^* : \ho(\operatorname{Hilb}^n(S)) \stackrel{\sim}{\longrightarrow} \ho(\operatorname{M}_\sigma(v)) $ of (the proof of) Theorem~\ref{thm:moduli} admits a multiple of $(R_0)_*$ as its inverse and that the co-multiplicative birational Chow--K\"unneth decomposition for $\ho(\operatorname{M}_\sigma(v))$ is transported from that of $ \ho(\operatorname{Hilb}^n(S))$ via $R_0^*$ and $(R_0)_*$.
	On the other hand, the filtrations $S_\bullet$ and $S^{\mathrm{SYZ}}_\bullet$ coincide on $\CH_0(\operatorname{Hilb}^n(S))$\,; see~\cite{SYZ}.
We may then conclude from the fact that the covariant and contravariant action of $R_0$ preserves the filtration $S^{\mathrm{SYZ}}_\bullet$.
\end{proof}

\begin{rmk}[Theorem~\ref{T:Voisin-filt} in case \eqref{moduli}]\label{R:moduli}
 Combined with the recent result of Li--Zhang~\cite[Thm.~1.1]{Li-Zhang} proving $S^{\mathrm{SYZ}}_\bullet\CH_0(\operatorname{M}_\sigma(v)) = S_\bullet \CH_0(\operatorname{M}_\sigma(v))$, Theorem~\ref{thm:splitting-moduli} establishes Theorem~\ref{T:Voisin-filt} in case \eqref{moduli}. 
\end{rmk}

\appendix

\section{The co-radical filtration on positive-dimensional cycles}\label{S:app}

The aim of this appendix, the results of which are not used in the main body of the paper, is to consider the co-radical filtration for (not necessarily zero-dimensional) cycles on a smooth projective variety equipped with a unit $o\in \CH_0(X)$. In \S\ref{SS:MCK} and \S\ref{SS:mod}, we parallel \S \ref{S:MBCK} and \S\ref{S:strict}, and explore the relations between on the one hand the co-radical filtration and on the other hand the existence of a multiplicative Chow--K\"unneth decomposition and so-called modified diagonals. In Proposition~\ref{P:BFMS-R}, we then observe that our co-radical filtration agrees with a filtration independently considered by Barros--Flapan--Marian--Silversmith~\cite[\S4]{BFMS}.
 Finally, in \S \ref{SS:abelian}, we show that the co-radical filtration on the Chow ring of an abelian variety is a ring filtration that is opposite to the candidate Bloch--Beilinson filtration of Beauville, thereby establishing Proposition~\ref{P:smash-abelian}.\medskip

We fix a smooth projective variety $X$ of pure dimension $d$ over a field $K$, equipped with a unit $o\in \CH_0(X)$.

\begin{defn} \label{D:corad-positive}
	For all $i\geq 0$, we  define on $\CH_i(X)$ 
	the increasing \emph{co-radical filtration}\,:
	$$R_k\CH_i(X) := \ker \big(\bar{\delta}^{k} : \CH_i(X) \to \CH_i(X^{k+1})  \big) =: R_{k}\Hom(\mathds 1(i),\h(X)) \quad \mbox{for  }   k\geq 0,$$
	where the right-hand side term is nothing but the filtration of \S \ref{SS:co-rad} for the covariant category of Chow motives (and the functor $\operatorname{C}:= \Hom(\mathds{1}(i),-)$). The correspondence $\bar \delta^k$ is made explicit in~\eqref{R:comult} below.
	The co-radical filtration on the Chow groups of $X$ graded by codimension is then defined as\,:
	\begin{equation*}\label{E:corad-codim}
	R_k\CH^i(X) := R_k\CH_{d-i}(X).
	\end{equation*}
\end{defn}

\subsection{Relation to multiplicative Chow--K\"unneth decompositions} \label{SS:MCK}

Recall that the pullback along the diagonal embedding $\delta : X\hookrightarrow X\times_K X$, together with the pullback along the structure morphism $\epsilon : X \to \Spec K$, defines an algebra structure on the contravariant Chow motive $\h(X)$ of $X$. The following definition mimics the classical definition of an (augmented) graded algebra.

\begin{defn}[{\cite[Def.~8.1]{SV}}] \label{D:MCK}
	A \emph{multiplicative Chow--K\"unneth decomposition} on the contravariant Chow motive $\h(X)$ is a finite direct sum decomposition 
	\begin{equation}\label{E:grading}
\h(X) := \h^0(X) \oplus \cdots \oplus \h^{2d}(X)
	\end{equation} such that 
	\begin{enumerate}[(a)]
		\item it is a Chow--K\"unneth decomposition, \emph{i.e.}, $\HH^*(\h^i(X)) = H^i(X)$ for all $i\geq 0$\,;
		\item it defines an algebra grading, \emph{i.e.},	$\delta^* : \h^i(X) \otimes \h^j(X) \to \h(X)$ factors through $\h^{i+j}(X)$ for all $i,j \geq 0$.
	\end{enumerate}
\end{defn}

If $\h(X)$ admits a multiplicative Chow--K\"unneth decomposition, then by \cite[footnote~24]{FV-JAG} we have canonical identifications $\h^{i}(X)^\vee = \h^{2d-i}(X)(d-i)$.  
Hence, if we assume in addition that $\h^0(X) \simeq \mathds 1$, then, by dualizing, we obtain a unital grading on the covariant Chow motive of~$X$, seen as a co-algebra object. By taking the image in the category of birational motives, this induces further a unital grading on the birational motive of $X$. We then denote $o$ the associated graded unit and the co-radical filtration $R_\bullet\CH^i(X)$ is implicitly  associated to $o$.

Given a Chow--K\"unneth decomposition as in~\eqref{E:grading}, 
the decreasing filtration  
$$ F^k\CH^i(X):= \bigoplus_{j\geq k}\CH^i(X)_{(j)}, \quad \mbox{where}\ \CH^i(X)_{(k)} := \CH^i(\h^{2i-k}(X)) $$
is conjecturally independent of the choice of a Chow--K\"unneth decomposition and is the Bloch--Beilinson filtration~\cite[\S 5]{Jannsen}.
We then define on $\CH^i(X)$ the filtration
$$G_k\CH^i(X) :=  \bigoplus_{j\leq k}\CH^i(X)_{(j)},$$
which  is opposite to the conjectural Bloch--Beilinson filtration. Note that the filtration $G_\bullet$ is a ring filtration if the Chow--K\"unneth decomposition is multiplicative.
We then have, by dualizing, the analogue of Proposition~\ref{P:co-rad} in the contravariant setting\,:

\begin{prop}\label{P:corad-pos}
Assume the contravariant Chow motive $\h(X)$ admits a multiplicative Chow--K\"unneth decomposition as in~\eqref{E:grading} with $\h^0(X) \simeq \mathds 1$. 
Define $R'_k\CH^i(X) := R_{2(d-i)+k}\CH^i(X)$.
Then $$G_k\CH^i(X) \subseteq R'_k\CH^i(X).$$
Moreover, if $\h(X)$ is generated by $\h^1(X)$, \emph{i.e.}, if the natural graded map $\operatorname{Sym}^*\h^1(X) \to \h(X)$ is split surjective, then equality holds and in particular $R'_\bullet\CH^*(X)$ is a ring filtration.\qed
\end{prop}

In case $\h^1(X) =0$, one has\,:

\begin{prop}\label{P:corad-pos2}
	Assume the contravariant Chow motive $\h(X)$ admits a multiplicative Chow--K\"unneth decomposition as in~\eqref{E:grading} with $\h^0(X) \simeq \mathds 1$ and $\h^1(X) = 0$. Define $R''_k\CH^i(X) := R_{(d-i)+k}\CH^i(X)$.
	Then $$G_{2k+1}\CH^i(X) \subseteq R''_k\CH^i(X).$$
	Moreover, if $\h(X)$ is generated by $\h^2(X)$, \emph{i.e.}, if the natural graded map $\operatorname{Sym}^*\h^2(X) \to \h(X)$ is split surjective, then equality holds. If in addition $\h^{2i+1}(X)=0$ for all $i\geq 0$, then $R''_\bullet\CH^*(X)$ is a ring filtration.\qed
\end{prop}

\begin{rmk}[The case of hyper-K\"ahler varieties]	 \label{R:MCK-corad}
	In case $X$ admits a multiplicative Chow--K\"unneth decomposition $\h(X) = \h^0(X) \oplus \cdots \oplus \h^{2d}(X)$ with $\h^0(X)\simeq \mathds 1$ and with $\h^1(X)=0$, which conjecturally is the case for hyper-K\"ahler varieties~\cite[Conj.~4]{SV}, then 
	 Proposition~\ref{P:corad-pos2}
	shows that
	$R''_\bullet\CH^i(X)$ contains a filtration opposite to $F^{2\bullet+1}\CH^i(X)$, where $F^\bullet$ is the conjectural Bloch--Beilinson filtration,
	and equality holds if $\h(X)$ is generated as an algebra object by $\h^2(X)$\,; \emph{e.g.}, if $X=\operatorname{Hilb}^2(S)$ for a K3 surface $S$ or if $X=F(Y)$ for a smooth cubic fourfold $Y$ (see~\cite{FLV-Franchetta2}). Note however that the latter does not hold in general for hyper-K\"ahler varieties since, for instance,  $\operatorname{Sym}^*\HH^2(X) \to \HH^*(X)$ is not surjective if $X$ is deformation-equivalent to $\operatorname{Hilb}^n(S)$ for a K3 surface $S$ and an integer $n\geq 3$, or is  deformation-equivalent to $K_n(A)$ for an integer $n\geq 2$ since $\HH^3(K_n(A)) \neq 0$.
	Arguments similar to those in the proof of Proposition~\ref{P:strictgradingBCK} show that if $\HH^*(X)$ is not generated by $\HH^2(X)$, then $\HH_*(X) = \HH^*(X)^\vee$ is not co-generated by $\HH_2(X)$, \emph{i.e.}, there exists a positive integer $k$ such that $\bar \delta^{k-1}: \HH_{2k}(X) \to \operatorname{Sym}^k \HH_2(X)$ is not injective. By the general Bloch--Beilinson philosophy, there exists conjecturally an integer~$i$ such that $\bar \delta^{k-1}$ is non-zero on $\operatorname{Gr}^{2k-2i}_F \CH^i(X)$ and so $R''_\bullet \CH^i(X)$ is conjecturally \emph{not} opposite to $F^{2\bullet+1}\CH^i(X)$ -- it only contains strictly a filtration opposite to $F^{2\bullet+1}\CH^i(X)$.
\end{rmk}

\subsection{Relation to modified diagonals}\label{SS:mod}

Let $o: \mathds 1 \to \h(X)$ be a unit for the covariant Chow motive of $X$ seen as  a co-algebra object. In other words, $o$ is a zero-cycle in $\CH_0(X)$ such that $\delta_*o = o\times o \ \mbox{in}\ \CH_0(X\times X)$. The iterated reduced co-multiplication $\bar \delta^n: \h(X) \to \h(X)^{\otimes n+1}$, as defined in \S \ref{SS:co-mult} in a general setting, satisfies 
\begin{equation}\label{R:comult}
\bar \delta^n = \prod_{i=1}^{n+1} p_{0,i}^*(\Delta_X - X\times o) \quad \ \mbox{in}\ \CH_{\dim X}(X\times X^{n+1}).
\end{equation}
Here $\Delta_X\in \CH_{\dim X}(X\times X)$ is the class of the diagonal and $p_{0,i} : X\times X^{n+1} \to X\times X$ is the projector on the product of the first and $(i+1)$-st factors. One recognizes here the cycle denoted $\Gamma^{1,n+1}(X,o)$ in \cite[(5)]{VoisinGT}, the pushforward of which under the projection $X\times X^{n+1} \to X^{n+1}$ is the so-called $(n+1)$-th modified diagonal cycle $\Gamma^{n+1}(X,o)$\,; see \cite[Lem.~2.1]{VoisinGT}.
Restricting to the generic point $\eta_X$ of $X$, we obtain explicitly for the iterated reduced co-multiplication (associated to the unit~$o$) on the birational motive of $X$\,:
$$\bar \delta^n = \prod_{i=1}^{n+1} p_{0,i}^*(\Delta_X|_{\eta_X\times X} - \eta_X \times o) \quad \ \mbox{in}\ \Hom\big(\ho(X), \ho(X)^{\otimes n+1}\big).$$
From~\cite[Cor.~1.6 \& Prop.~2.2]{VoisinGT}, we have that $\bar \delta^n: \h(X) \to \h(X)^{\otimes n+1}$ vanishes for $n$ large enough.
Recall from \S \ref{SS:co-mult} that in the general setting where $M=M_{(0)}\oplus \cdots \oplus M_{(n)}$ is a unital grading on a co-algebra object $M$, then the iterated reduced co-multiplication $\bar \delta^{n}$ vanishes. 
Therefore, from the discussion in \S \ref{SS:MCK}, we see that if $\h(X)$ admits a multiplicative Chow--K\"unneth decomposition with $\h^0(X)\simeq \mathds 1$, then $\bar \delta^{2\dim X} = 0$ (and hence by pushforward $\Gamma^{2\dim X+1}(X,o)=0$) and $\bar \delta^{\dim X+1} = 0$ (and hence by pushforward $\Gamma^{\dim X+1}(X,o)=0$) in case $\h^1(X)=0$. 
This recovers and gives a more conceptual proof of \cite[Prop.~8.12]{SV}.
In particular, \cite[Conj.~4]{SV} on the existence of a multiplicative Chow--K\"unneth decomposition for hyper-K\"ahler varieties $X$ implies the vanishing of the modified diagonal $\Gamma^{2n+1}(X,o)$, which is a conjecture of O'Grady~\cite{OG-diag}.

On the other hand, similar arguments show that if we have the weaker assumption that the birational motive $\ho(X)$ admits a co-multiplicative birational Chow--K\"unneth decomposition $\{\varpi_0^X, \ldots, \varpi_d^X \}$, then $\bar \delta^{d} = 0$ in $\Hom\big(\ho(X), \ho(X)^{\otimes d+1}\big)$ and, if $\ho_1(X)=0$, $\bar \delta^{\lfloor d/2\rfloor} = 0$ in $\Hom\big(\ho(X), \ho(X)^{\otimes \lfloor d/2\rfloor+1}\big)$, or, equivalently by Lemma~\ref{lem:zero},
 $([x]-o)^{\times d+1} = 0$ in $\CH_0(X^{d+1})$ for all $x\in X$ and, if $\ho_1(X)=0$,  $([x]-o)^{\times \lfloor d/2\rfloor+1} = 0$ in $\CH_0(X^{\lfloor d/2\rfloor+1})$ for all $x\in X$.
The arguments of Voisin~\cite[\S 3]{VoisinGT} show that if for some $k$ we have $([x]-o)^{\times {k+1}} = 0 $ in $\CH_0(X^{k+1})$ for all points $x\in X$, then the modified diagonal $\Gamma^{m}(X,o)$ vanishes for all $m\geq (\dim X+1)(k+1)$. In particular, for the $2n$-dimensional hyper-K\"ahler varieties considered in Proposition~\ref{P:smash-hyperK}, we have  $\Gamma^{m}(X,o) = 0$ for all $m\geq (2n+1)(n+1)$.
\medskip

	Focusing on moduli spaces of stable sheaves on K3 surfaces,
	 Barros, Flapan, Marian and Silversmith have independently introduced in \cite[\S 4]{BFMS} the following ascending filtration on $\CH_i(\operatorname{M}_\sigma(v))$\,:
	\begin{equation}\label{E:BFMS}
	S^{\mathrm{BFMS}}_k\CH_i(\operatorname{M}_\sigma(v))  := \{ \alpha \in \CH_i(\operatorname{M}_\sigma(v)) \ \big\vert \ p_0^*\alpha \cdot \bar \delta^{i+k} =0 \mbox{ in } \CH_i(\operatorname{M}_\sigma(v)\times \operatorname{M}_\sigma(v)^{i+k+1}) \},
	\end{equation}
	where $p_0 : \operatorname{M}_\sigma(v) \times \operatorname{M}_\sigma(v)^{i+k+1} \to \operatorname{M}_\sigma(v)$ is the projection on the first factor. (Note that this filtration is denoted $S$ in \cite{BFMS} and that $\bar \delta^j = \overline{\Delta}_{0,1}\cdots \overline{\Delta}_{0,j+1}$ by~\eqref{R:comult}). We note that the filtration \eqref{E:BFMS} can in fact be defined for any smooth projective variety $X$ equipped with a unit $o\in \CH_0(X)$,
	 and that we then have the obvious inclusion $S_k^{\mathrm{BFMS}}\CH_i(X) \subseteq R_{i+k}\CH_i(X)$ for all $i\geq 0$ and all $k\geq -i$. The above inclusion is in fact an equality\,:

	\begin{prop}\label{P:BFMS-R}
We have
$S^{\mathrm{BFMS}}_k\CH_i(X)  = R_{i+k}\CH_i(X)$.
	\end{prop}
	\begin{proof}
	I thank Alina Marian~\cite{Marian} for mentioning the equality in the statement of the proposition and for  providing the following combinatorial argument.
For ease of notation, we write $\bar \Delta := \Delta_X - X\times o$ and $\overline \Delta_{i,j}$ (resp.\ $\Delta_{i,j}$) for the pull-back of $\overline \Delta$ (resp.\ $\Delta_X$) along the projection $X^n \to X\times X$ on the product of the $i$-th and $j$-th factors. We also write $\alpha_i$ for the pull-back of a cycle $\alpha$ on $X$ along the projection $X^n \to X$ on the $i$-th factor.
For $\alpha \in \CH_*(X)$, we show $$\bar \delta^{k-1}_*\alpha=0 \mbox{ in } \CH_*(X^k) \iff \alpha_0\cdot \bar \delta^{k-1} = 0 \mbox{ in } \CH_*(X\times X^k),$$
where the left factor of $X\times X^k$ is the $0$-th factor.
The direction $\Leftarrow$ is trivial. To see the other direction, recall from~\eqref{R:comult} that $\bar \delta^{k-1} = \overline \Delta_{0,1}\cdot \overline \Delta_{0,2}  \cdots \overline \Delta_{0,k}$ and consider the cycle 
$$\gamma := \alpha_0 \cdot \Delta_{0,1}\cdot \overline \Delta_{0,2}  \cdots \overline \Delta_{0,k+1} \in \CH_*(X\times X^{k+1}).$$
We write $\gamma$ in two different ways\,:
\begin{enumerate}
	\item $\gamma = \alpha_1 \cdot \Delta_{0,1}\cdot \overline \Delta_{1,2}  \cdots \overline \Delta_{1,k+1}$ (switch the indices $0$ and $1$ everywhere)\,;
	\item $\gamma = \alpha_0 \cdot \Delta_{0,1}\cdot \overline \Delta_{0,2}  \cdots \overline \Delta_{0,k}\cdot  \overline \Delta_{1,k+1}$ (switch $0$ and $1$ on the last factor only).
\end{enumerate}
We now calculate $\pi_*\gamma$ with $\pi : X\times X^{k+1} \to X^{k+1}$ the projection on the last $k+1$ factors. From~(1) we have
$$\pi_*\gamma = \alpha_1 \cdot  \overline \Delta_{1,2}  \cdots \overline \Delta_{1,k+1},$$
while from (2) we first have $\gamma = \alpha_0 \cdot \overline\Delta_{0,1}\cdot \overline \Delta_{0,2}  \cdots \overline \Delta_{0,k}\cdot  \overline \Delta_{1,k+1}$ since $o_1\cdot \overline \Delta_{1,k+1} = 0$, and then
$$\pi_*\gamma = \pi_*(\alpha_0 \cdot \overline\Delta_{0,1}\cdot \overline \Delta_{0,2}  \cdots \overline \Delta_{0,k}) \cdot  \overline \Delta_{1,k+1},$$
where $\pi_*(\alpha_0 \cdot \overline\Delta_{0,1}\cdot \overline \Delta_{0,2}  \cdots \overline \Delta_{0,k}) $ is understood to be pulled back from $X^k$ under the projection $X^{k+1} \to X^k$ on the first $k$ factors.
Clearly then, the vanishing of $\bar \delta^{k-1}_*\alpha = \pi_*(\alpha_0 \cdot \overline\Delta_{0,1}\cdot \overline \Delta_{0,2}  \cdots \overline \Delta_{0,k})$ implies the vanishing of  $\alpha_1 \cdot  \overline \Delta_{1,2}  \cdots \overline \Delta_{1,k+1}$ which is nothing but the cycle $\alpha_0\cdot \bar \delta^{k-1} $ with a shift of indices.
	\end{proof}

\subsection{The co-radical filtration on the Chow ring of abelian varieties}
\label{SS:abelian}
The aim of this paragraph is to spell out the analogues of the results of \S \ref{S:MBCK}, \S \ref{S:strict} and \S \ref{S:corad} in the case of abelian varieties.
Let $A$ be an abelian variety over a field $K$ and denote $g$ its dimension. Let $[n] : A \to A$ denote the multiplication by $n$ morphism.
Recall from~\cite{Beauville} that the Chow ring of an abelian variety~$A$ admits a bigrading
\begin{equation}\label{E:Beauville}
\CH^*(A) = \bigoplus_{i-g\leq j \leq i} \CH^i(A)_{(j)}, 
\end{equation}
where $\CH^i(A)_{(j)} := \{ \alpha \in \CH^i(A) \ \big\vert \ [n]^*\alpha = n^{2i-j}\alpha \mbox{ for all } n\in \Z \}$
and that 
\begin{equation}\label{E:BB-abelian}
F^k \CH^*(A) := \bigoplus_{j\geq k}  \CH^*(A)_{(j)}
\end{equation}
is conjecturally the Bloch--Beilinson filtration (in particular,  it is expected that $\CH^*(A)_{(j)}=0$ for all $j<0$, or equivalently, $F^0\CH^*(A) = \CH^*(A)$).

   Fix an integer $n$ distinct from $-1,0$, or $1$. 
 The Deninger--Murre \cite{DM}  decomposition of the (contravariant) Chow motive $$\h(A) = \h^0(A) \oplus \cdots \oplus \h^{2g}(A)$$ is a Chow--K\"unneth decomposition that can be obtained by considering the projectors on the various eigenspaces for the  multiplication by $n$ morphism, \emph{i.e.}, 
 \begin{equation}\label{E:DM}
 \pi^k_A := \prod_{0\leq i \leq 2g, i\neq k}\frac{[n]^*-n^i}{n^k-n^i}.
 \end{equation}
   With that description, it is clear that the Deninger--Murre decomposition provides a multiplicative Chow--K\"unneth decomposition of $\h(A)$, and in particular a grading of $\h(A)$ considered as an algebra object. It is also clear that 
   $\CH^i(A)_{(j)} = \CH^i(\h^{2i-j}(A))$, \emph{i.e.}, that the Deninger--Murre decomposition lifts to $\h(A)$ the Beauville decomposition on $\CH^*(A)$.

   The sum morphism $\Sigma : A \times A \to A$ induces a map $\Sigma_*: \h(A) \otimes \h(A) \to \h(A)(-g)$  called the \emph{Pontryagin product} and that we more commonly denote $\ast$. K\"unnemann~\cite{K-abelian} showed that the Deninger--Murre projectors can be alternately described as 
   \begin{equation}\label{E:K-proj}
\pi^k_A = \frac{1}{(2g-k)!}\Bigg(\sum_{n=1}^{2g} \frac{(-1)^{n-1}}{n}\big(\mathrm{id}_A - A\times 0\big)^{\ast n}\Bigg)^{\ast (2g-k)}.
   \end{equation}
(In the above formula, the Pontryagin product is to be understood on $A\times A$ viewed as an abelian scheme over $A$ via the first projection.)
   With that description, K\"unnemann shows (see Theorem~\ref{T:K}) that the multiplication map
   \begin{equation}\label{E:K}
   \operatorname{Sym}^* \h^1(A) \stackrel{\sim}{\longrightarrow} \h(A)
   \end{equation}
   is an isomorphism of graded algebra objects, with inverse given by the sum of the isomorphisms
  \begin{equation*}
  \frac{1}{k!} (\Sigma^k)^* : \h^k(A) \stackrel{\sim}{\longrightarrow} \operatorname{Sym}^k \h^1(A),
  \end{equation*}
	where $\Sigma^k : A^k \to A$ is the sum homomorphism.
	\medskip
	
	Dualizing \eqref{E:K} and passing to covariant Chow motives, and setting $\h_k(A) := \pi_A^{2g-k}\h(A)$, one obtains that co-multiplication induces an isomorphism of unital graded co-algebra objects
	\begin{equation}\label{E:K-birat}
\h(A) \stackrel{\sim}{\longrightarrow}  \operatorname{Sym}^* \h_1(A)
	\end{equation}
	 with $\h(A)$ endowed with the unit $0: \mathds{1} \to \h(A)$. Its inverse is given by the sum of the isomorphisms
	\begin{equation}\label{E:Pontryagin}
	\frac{1}{k!} \Sigma^k_* : \operatorname{Sym}^k \h_1(A)  \stackrel{\sim}{\longrightarrow}  \h_k(A).
	\end{equation}
	This isomorphism \eqref{E:K-birat} in particular endows the (covariant) Chow motive $\h(A)$ with a strict grading in the sense of \S \ref{SS:strictgrading}.
	Note also as in \S \ref{SS:co-mult} that, for $k>0$ the degree-$k$ part of the isomorphism \eqref{E:K-birat} 
	is nothing but the iterated reduced co-multiplication $\bar \delta^{k-1}$.
We then have\,:

\begin{thm}\label{thm:abelian}
	Let $A$ be an abelian variety of dimension $g$ over a field $K$. The co-radical filtration $R'_k\CH^i(A):= R_{2(g-i)+k}\CH^i(A)$, where $R_\bullet$ is as in Definition~\ref{D:corad-positive}, relates to the Beauville decomposition as follows\,:
	 \begin{equation}\label{E:R-beauville}
R'_k\CH^i(A) = \bigoplus_{j=i-g}^k \CH^i(A)_{(j)}.
	 \end{equation}
	 In particular, 
	 \begin{enumerate}[(a)]
      \item we have 
      $$0=R'_{i-g-1}\CH^i(A) \subseteq R'_{i-g}\CH_i(A) \subseteq \cdots \subseteq R'_{i}\CH^i(A) = \CH_i(A),$$ and conjecturally $R'_{-1}\CH^i(A) = 0$\,;
      \item  the co-radical filtration $R'_\bullet$ on $\CH^i(A)$ is opposite to the filtration $F^\bullet$  of~\eqref{E:BB-abelian} (which, conjecturally is the Bloch--Beilinson filtration) and we have
      $$R'_k\CH^i(A) \cap F^k\CH^i(A) = \CH^i(A)_{(k)}\,;$$
      \item  the co-radical filtration $R'_\bullet$  is an algebra filtration on $\CH^i(A)$, \emph{i.e.},
      $$R'_{k}\CH^i(A) \cdot R'_{k'}\CH^j(A) \subseteq R'_{k+k'}\CH^{i+j}(A)\,;$$
      \item every point $x\in A$ satisfies $$([x]-[0])^{\times g+1} = 0 \ \mbox{in}\ \CH_0(A^{g+1}).$$
	 \end{enumerate}
\end{thm}
\begin{proof} 
The identity~\eqref{E:R-beauville} is the combination of Beauville's decomposition~\eqref{E:Beauville}, Theorem~\ref{T:K} and Proposition~\ref{P:corad-pos}. Alternately, it is the combination of Beauville's decomposition~\eqref{E:Beauville}, the strict grading of $\h(A)$ provided by \eqref{E:K-birat} and Proposition~\ref{P:crucial}. 
	Items (a), (b) and (c) then follow from the properties of the Beauville decomposition recalled above, and item (d) follows from (a) and Proposition~\ref{P:RT}.
\end{proof}

\begin{rmk}
Alina Marian~\cite{Marian} has informed us that the identity~\eqref{E:R-beauville} can also be obtained by using the coincidence (up to shift) of the co-radical filtration $R_\bullet$ with the filtration $S_\bullet^{\mathrm{BFMS}}$ defined in~\eqref{E:BFMS}, and by considering the eigenspace decomposition of $\CH^*(A\times A^{k+1})$ for the action of the map that acts as multiplication by $N$ on the first factor $A$ and as the identity on the second factor $A^{k+1}$. (The latter is based on a discussion between Alina Marian and Qizheng~Yin).
\end{rmk}

\begin{rmk}\label{rmk:co-radical-abelian}
	In contrast to the case of hyper-K\"ahler varieties where it is expected (Conjecture~\ref{conj:Voisin-split-cogen}) that $R_k\CH_0(X)$ is spanned by classes of points, it is in general not true that  $R_k\CH_0(A)$ is spanned by classes of points for a complex abelian variety $A$. Voisin~\cite[Thm.~1.8]{Voi-gon} has indeed showed that if $A$ is a very general abelian variety of dimension $g\geq 2k-1$, then the set $A_k:=\{x\in A \ \big\vert\ ([x]-[0])^{*k} = 0 \mbox{ in } \CH_0(A)\}$ is countable, and hence so is the subset $R_{k-1}A := \{x\in A \ \big\vert\ ([x]-[0])^{\times k} = 0 \mbox{ in } \CH_0(A^k)\}$. However, by Theorem~\ref{thm:abelian}, the set $R_{k-1}A$ cannot span $R_{k-1}\CH_0(A)$ for $k>1$ since the group $\CH^g(A)_{(1)} \simeq A(\C)$ is not spanned by a countable subset.
\end{rmk}

\section{$\delta$-filtrations}\label{S:delta}

Let $X$ be a hyper-K\"ahler variety. 
We saw that the Voisin filtration on $\CH_0(X)$ is expected to be induced by a strict grading on the birational motive of $X$, and as such admits a splitting that is ``compatible'' with the diagonal embedding map.
Our aim here is to provide, somewhat artificially as we feel that the language of birational motives is the right language to use in that setting, a strong notion of splitting of an ascending filtration on $\CH_0$ that avoids the use of birational motives.

\subsection{$\delta$-filtrations and $\delta$-gradings}
The following definitions are justified by Proposition~\ref{P:justify} below.

\begin{defn}
	Let $X$ be a smooth projective variety over a  field~$K$ and let $F_\bullet$ be an ascending filtration on $\CH_0(X)$ with $F_{-1}\CH_0(X) = \{0\}$. For all $n>1$, we define the filtration $F^\delta_\bullet$ on $\CH_0(X^n)$\,:
	$$F^\delta_k\CH_0(X^n) := \operatorname{im}\Big( \bigoplus_{i_1+\cdots+i_n=k} F_{i_1} \CH_0(X)\otimes\cdots \otimes F_{i_n}\CH_0(X) \to \CH_0(X^n)\Big).$$
\end{defn}

Recall from \S \ref{SS:0cycle} that if $K$ is algebraically closed, then the exterior product map $\CH_0(X)\otimes \CH_0(Y) \to \CH_0(X\times_KY)$ is surjective for all smooth projective varieties $X$ and $Y$ over $K$. Therefore, if the filtration $F_\bullet \CH_0(X)$ is exhaustive, then so is the induced filtration $F^\delta_\bullet\CH_0(X^n)$ for $n>0$.
We note however that if $\CH_0(X) = \bigoplus_{k\geq 0} \CH_0(X)_{(k)}$ is a finite grading, \emph{i.e.}, a finite direct sum decomposition, then 
\begin{equation}\label{eq:grading}
\CH_0(X^n)^\delta_{(k)} := \operatorname{im}\Big( \bigoplus_{i_1+\cdots+i_n=k}  \CH_0(X)_{(i_1)}\otimes\cdots \otimes \CH_0(X)_{(i_n)} \to \CH_0(X^n)\Big)
\end{equation}
need not define a grading on $\CH_0(X^n)$ for $n>1$, \emph{i.e.}, the pieces $\CH_0(X^n)^\delta_{(k)}$ for $k\geq 0$ need not be in a direct sum.

\begin{defn}
	\label{D:delta-fil}
	Let $X$ be a smooth projective variety and denote $\delta^n : X \hookrightarrow X^{n+1}$ the diagonal embedding.
	
	\noindent $\bullet$ A \emph{$\delta$-filtration} on $\CH_0(X)$ is an exhaustive ascending filtration $F_\bullet$ with $F_{-1}\CH_0(X) = \{0\}$ such that 
	\begin{enumerate}[(a)]
		\item $\epsilon_* : F_0\CH_0(X) \to \Q \quad \mbox{is an isomorphism}$\,;
		\item $ \delta^n_* \big( F_k \CH_0(X)\big)  \subseteq F^\delta_k \CH_0(X^{n+1})$ for all $k$ and $n$.
	\end{enumerate}
	
	\noindent $\bullet$	A \emph{$\delta$-grading} on $\CH_0(X)$ is a finite direct sum decomposition $\CH_0(X) = \bigoplus_{k\geq 0}\CH_0(X)_{(k)}$ such that 
	\begin{enumerate}[(a)]
		\item $\epsilon_* : \CH_0(X)_{(0)} \to \Q $ is an isomorphism and $\ker (\epsilon_*: \CH_0(X) \to \Q) =  \bigoplus_{k > 0}\CH_0(X)_{(k)} $\,;
		\item  $\delta^n_*\big( \CH_0(X)_{(k)}\big) \subseteq  \CH_0(X^{n+1})^\delta_{(k)}$ for all $k$ and $n$\,;
		\item  $\CH_0(X^n) = \bigoplus_{k \geq 0} \CH_0(X^n)^\delta_{(k)}$ for all $n>0$.
	\end{enumerate}
	In items (b) and (c) above,  $\CH_0(X^n)^\delta_{(k)}$ is defined as in \eqref{eq:grading}.
	
	\noindent $\bullet$ A $\delta$-grading 
	is said to be \emph{strict} 
	if $\bar \delta^k_* : \CH_0(X)_{(k+1)} \to \CH_0(X^{k+1})$ is injective for all $k\geq 0$, where 
	$\bar \delta^k = \bar p^{\otimes k+1}\circ \delta^k$ with $\delta^k : X\hookrightarrow X^{k+1}$ the diagonal embedding and $\bar p = \Delta_X - X\times o$ with $o$ the degree-1 generator of $\CH_0(X)_{(0)}$.
	
	\noindent $\bullet$	The $\delta$-filtration associated to a $\delta$-grading is the filtration defined by $$F^\delta_k\CH_0(X^n) := \bigoplus_{r\leq k} \CH_0(X^n)^\delta_{(r)}.$$
	
	\noindent $\bullet$	A $\delta$-filtration is said to be \emph{split}  if it is the filtration associated to a $\delta$-grading. 
\end{defn}

\begin{prop}\label{P:justify}
	Let $X$ be a smooth projective variety. If $\ho(X)$ admits a unital grading (resp.\ a strict grading) $\ho(X) = \ho(X)_{(0)} \oplus \cdots \oplus \ho(X)_{(n)}$, then the associated  grading on $\CH_0(X)$ defined by $\CH_0(X)_{(k)}:= \CH_0(\ho(X)_{(k)})$ is a $\delta$-grading (resp.\ a strict $\delta$-grading).
\end{prop}
\begin{proof}
	This is clear.
\end{proof}

\subsection{$\delta$-filtrations and the co-radical filtration}
Let $X$ be a smooth projective variety. Fix a unit $o\in \CH_0(X)$ and denote $R_\bullet$ the associated co-radical filtration on $\CH_0(X)$. Proposition~\ref{P:delta-coradical} below is a variant of Proposition~\ref{P:crucial}\,; it shows that $R_\bullet$ contains the maximal  $\delta$-filtration $F_\bullet\CH_0(X)$ such that $F_0\CH_0(X)=\Q o$, and that $R_\bullet$ is itself a $\delta$-filtration if $\CH_0(X)$ admits a strict $\delta$-grading with $\CH_0(X)_{(0)}=\Q o$.
(As already mentioned in \S \ref{SS:co-rad}, note the analogy with the fact~\cite[Lem.~11.2.1]{sweedler} that the filtration associated to a strict grading on a pointed irreducible co-algebra is necessarily the co-radical filtration).

\begin{prop}\label{P:delta-coradical}
	Let $X$ be a smooth projective variety and fix a unit $o\in \CH_0(X)$. If $F_\bullet\CH_0(X)$ is a $\delta$-filtration with $F_0\CH_0(X) =\Q o$,
	then 
	$$F_k\CH_0(X) \subseteq R_{k}\CH_0(X).$$ 
	If in addition $F_\bullet$ is the $\delta$-filtration associated to a strict $\delta$-grading $\CH_0(X) = \bigoplus_{k \geq 0}\CH_0(X)_{(k)}$,
	then  for all $k\geq 0$
	$$F_k\CH_0(X) = R_{k}\CH_0(X).$$
\end{prop}
\begin{proof}
Using the definition of a $\delta$-grading, namely the property that the subspaces $\CH_0(X^{k+1})^\delta_{(l)}$ (as defined in \eqref{eq:grading}) of $\CH_0(X^{k+1})$ are in a direct sum for varying $l$, the proof is an easy adaptation of the proof of Proposition~\ref{P:crucial} and is left to the interested reader. 
\end{proof}

Combining Proposition~\ref{P:justify} and Proposition~\ref{P:delta-coradical}, we see that Conjecture~\ref{conj2:MBCK-bis} implies\,:

\begin{conj}
Let $X$ be a hyper-K\"ahler variety. Then there exists a unit $o\in \CH_0(X)$ such that the associated co-radical filtration on $\CH_0(X)$ is a split $\delta$-filtration.
\end{conj}
In particular, by Theorem~\ref{thm:main2}, the hyper-K\"ahler varieties \eqref{hilb}, \eqref{moduli}, \eqref{kum} and \eqref{fano} satisfy the above conjecture.

\subsection{Splitting of the Voisin filtration}
Voisin~\cite{Voisin-coisotropic} conjectured that the filtration $S_\bullet$ on the Chow groups of zero-cycles on hyper-K\"ahler varieties is \emph{split} in the sense that it is opposite to the \emph{conjectural} Bloch--Beilinson filtration $F^{2\bullet}$. 
The following conjecture in particular provides an alternate notion for the \emph{splitting} of $S_\bullet$ that does not depend on the existence of the conjectural Bloch--Beilinson filtration\,: 

\begin{conj}\label{conj:Voisin-split-cogen2}
	The Voisin filtration $S_\bullet$ on the Chow group of zero-cycles on a hyper-K\"ahler variety $X$ is a
	 split $\delta$-filtration. 
\end{conj}

In fact, due to Conjecture~\ref{conj2:MBCK-bis} and Proposition~\ref{P:justify}, we would expect $S_\bullet$ to be the $\delta$-filtration associated to a strict $\delta$-grading. If this is the case for a hyper-K\"ahler variety $X$, then we would obtain by Proposition~\ref{P:delta-coradical} the coincidence of $S_\bullet$ and of the co-radical filtration $R_\bullet$ associated to the class of the unit spanning $S_0\CH_0(X)$, in other words we would obtain the validity of Conjecture~\ref{conj:Voisin-split-cogen} for $X$.
We note that even showing in general that the Voisin filtration~$S_\bullet$ is a $\delta$-filtration is a non-trivial matter.

\begin{thm}\label{thm:splitting2}
	Let $X$ be one of the hyper-K\"ahler varieties \eqref{hilb}, \eqref{kum} or \eqref{fano}.
	Then Conjecture~\ref{conj:Voisin-split-cogen2} holds for $X$, \emph{i.e.}, the Voisin filtration $S_\bullet$ on $\CH_0(X)$ defines a split $\delta$-filtration.
\end{thm}
\begin{proof}
	More strongly, the Voisin filtration $S_\bullet$ is a $\delta$-filtration associated to a \emph{strict} $\delta$-grading. Indeed, it suffices by Proposition~\ref{P:justify} to show that $S_\bullet$ coincides with the ascending filtration induced by a strict grading on the birational motive $\ho(X)$. This was established in the proof of Theorem~\ref{thm:splitting}.
\end{proof}

\section{Motivic surface decomposition}\label{S:MSD}
Our aim is to study a motivic version of Voisin's surface
decomposition conjecture for hyper-K\"ahler varieties. 

\subsection{Voisin's surface decomposition}
We recall the following notion due to Voisin\,:

\begin{defn}[Surface decomposition~\cite{VoisinTriangle}]\label{def:Voisin}
	A  projective manifold $X$  of even dimension $2n$  is
	said to be
	\emph{surface decomposable} if there exist a projective smooth
	variety $\Gamma$, smooth projective surfaces $S_1,\ldots, S_n$ and generically finite morphisms
	$$\xymatrix{ \Gamma \ar[r]^\phi \ar[d]_\psi & X \\
		S_1\times \cdots \times S_n 
	}$$
	such that for any global $2$-form $\sigma \in \HH^0(X,\Omega^2_X)$ there exist global $2$-forms $\tau_i \in \HH^0(S_i,\Omega^2_{S_i})$ such that
	$$\phi^*\sigma = \psi^*\big(\sum_i p_i^*\tau_i \big).$$ 
	Here $p_i : S_1\times\cdots\times S_n \to S_i$ denote the natural projections.
\end{defn}

Based on the evidence provided by \cite[Thm.~3.3]{VoisinTriangle} (which includes cases \eqref{hilb}, \eqref{fano} and \eqref{llsvs}),
Voisin formulated\,:

\begin{conj}[Voisin \cite{VoisinTriangle}, Surface decomposability for
	hyper-K\"ahler varieties]\label{conj:voisin}
	Every hyper-K\"ahler variety is  surface
	decomposable.
\end{conj}

\subsection{Motivic surface decomposition}
As will be spelled out in Proposition~\ref{prop:voisin} below, the following notion lifts the notion of Voisin's surface decomposition to rational equivalence.

\begin{defn}[Motivic surface decomposition]\label{D:MSD}
	A smooth projective variety $X$ of even dimension $2n$ over a field $K$ is
	said to be
	\emph{motivically surface decomposable} if there exist a projective
	variety $\Gamma$, smooth projective surfaces $S_1,\ldots, S_n$ and surjective morphisms
	$$\xymatrix{ \Gamma \ar[r]^\phi \ar[d]_\psi & X \\
		S_1\times \cdots \times S_n 
	}$$
	such that\,: 
	\begin{enumerate}[(i)]
		\item $\phi_*[p] = \phi_*[q]$  in $\CH_0(X_\Omega)$, for any two general
		points $p$ and $q$ in $\Gamma(\Omega)$ lying on the same fiber of $\psi$.
		\setcounter{1}{\value{enumi}}
	\end{enumerate} 
\end{defn}

\begin{rmk}\label{R:eq}
	Equivalently, 
	up to taking a linear section of $\Gamma$, 
	one can assume in Definition~\ref{D:MSD} that $\phi$
	and $\psi$ are generically finite. In that case,  $(i)$ is equivalent to (see the proof of Proposition~\ref{prop:coalg})\,:
	\begin{enumerate}[(i)]
		\setcounter{enumi}{\value{1}}
		\item  	$\phi_* \psi^*\psi_* \alpha = \deg(\psi)\, \phi_*\alpha$ in
		$\CH_0(X_\Omega)$, for any zero-cycle $\alpha \in \CH_0(\Gamma_\Omega)$.
		\setcounter{1}{\value{enumi}}
	\end{enumerate}
	In addition, as in Remark~\ref{R:biratGamma}, if resolution of singularities holds over $K$, up to desingularizing~$\Gamma$, we may assume that $\Gamma$ is smooth over $K$.
\end{rmk}

It is clear that the notion of motivic surface decomposability is stable under product and is a birational
invariant among smooth projective varieties.
Moreover, 	Proposition~\ref{prop:coalg}(c) shows that if $X$ has a motivic surface
decomposition as in Definition~\ref{D:MSD}, then the co-algebra structure on
$\ho(X)$ is determined by the co-algebra structure on $\ho(S_1\times
\cdots\times S_n)$ and hence by the co-algebra structures on the birational
motives $\ho(S_i)$, $1\leq i \leq n$.
The following proposition shows that the notion of ``motivic surface
decomposability''can be thought of as an analogue for rational equivalence of Voisin's notion of
``surface decomposability'' which is purely cohomological.

\begin{prop}\label{prop:voisin}
	Let $X$ be a smooth projective complex variety of even dimension.
	If $X$ admits a motivic surface decomposition (Definition~\ref{D:MSD}) where the surfaces $S_i$ have vanishing irregularity, then $X$ is surface decomposable (Definition~\ref{def:Voisin}).
\end{prop}
\begin{proof} As mentioned in Remark~\ref{R:eq}, up to replacing $\Gamma$ with a desingularization of a linear section, we may assume that $X$ has a motivic surface decomposition as in Definition~\ref{D:MSD} with $\Gamma$ smooth over $K$ of dimension $2n=\dim X$.
	By the Bloch--Srinivas argument \cite{BS}, if $\phi_* \psi^*\psi_*  -
	\deg(\psi)\, \phi_*$ acts trivially on zero-cycles, then its transpose $\psi^*
	\psi_*\phi^*- \deg(\psi)\, \phi^*$ acts trivially on global $k$-forms for all
	$k\geq 0$. In particular, it acts trivially on $2$-forms. The latter is
	equivalent to saying that for any global $2$-form $\sigma \in
	\HH^0(X,\Omega^2_X)$ there exists a global $2$-form $\tau \in \HH^0(S_1\times
	\cdots \times S_n, \Omega^2_{S_1\times \cdots \times S_n})$ such that
	$\phi^*\sigma = \psi^*\tau$. Indeed, one simply takes $\tau = \frac{1}{\deg
		\psi}\,\psi_*\phi^*\sigma$. In case $q(S_i)=0$ for $1\leq i \leq n$,
	this is further equivalent to the
	existence of global $2$-forms $\tau_i \in \HH^0(S_i,\Omega^2_{S_i})$ such that
	$\phi^*\sigma = \psi^*(\sum_i p_i^*\tau_i)$, where $p_i : S_1 \times \cdots
	\times S_n \to S_i$ are the natural projections, which is the original
	formulation of Voisin~\cite{VoisinTriangle} as laid out in Definition~\ref{def:Voisin}.
\end{proof}

For the record, we have the following easy result.
\begin{prop}\label{prop:stable}
	Let $X$ and $Y$ be smooth projective varieties of same dimension $d$ over a field~$K$. Assume either one of the following\,:
	\begin{enumerate}[(i)]
		\item there is a dominant rational map $f:Y\dashrightarrow X$, or
		\item there exist a projective variety $\Gamma$  and surjective morphisms
		$$\xymatrix{ \Gamma \ar[r]^\phi \ar[d]_\psi & X \\
			Y
		}$$ 
		such that 	$\phi_*[p] = \phi_*[q]$  in $\CH_0(X_\Omega)$, for any two general
		points $p$ and $q$ in $\Gamma(\Omega)$ lying on the same fiber of $\psi$. 
	\end{enumerate}
	If $Y$ is motivically surface decomposable, then $X$ is motivically surface
	decomposable.
\end{prop}
\begin{proof}
	More generally, suppose $Y$ has the following property\,: there exist  a smooth projective variety $Z$ of dimension $d$ and  a projective variety $\Gamma'$  with surjective morphisms
	$\phi' : \Gamma' \to Y$ and $\psi' : \Gamma' \to Z$
	such that 
	$\phi'_*[p] = \phi'_*[q]$  in $\CH_0(Y_\Omega)$, for any two general
	points $p$ and $q$ in $\Gamma'(\Omega)$ lying on the same fiber of $\psi'$. 
	
	In case $(i)$, if
	$\pi : \widetilde{\Gamma}\to \Gamma'$ denotes a resolution of $f\circ \phi' : \Gamma' \dashrightarrow X$, then as in Remark~\ref{R:biratGamma} we note that $\Phi := f\circ \phi\circ \pi : \widetilde \Gamma \to X$ and $\Psi := \psi \circ \pi : \widetilde \Gamma \to Z$ are such that  $\Phi_*[p] = \Phi_*[q]$  in $\CH_0(X_\Omega)$, for any two general
	points $p$ and $q$ in $\Gamma(\Omega)$ lying on the same fiber of $\Psi$.
	
	In case $(ii)$, we form the cartesian square 
	$$\xymatrix{\Gamma\times_Y \Gamma' \ar[r]^{\quad \phi_{\Gamma}} \ar[d]_{\psi_{\Gamma'}} & \Gamma \ar[d]^\psi \ar[r]^{\phi} & X \\
		\Gamma' \ar[r]^{\phi'} \ar[d]_{\psi'} & Y  \\  Z.
	}$$ Since $\phi_{\Gamma}$ maps fibers of $\psi_{\Gamma'}$ to fibers of $\psi$, we have that $\Phi :=  \phi\circ \phi_{\Gamma}$ and $\Psi := \psi' \circ \psi_{\Gamma'}$ are such that  $\Phi_*[p] = \Phi_*[q]$  in $\CH_0(X_\Omega)$, for any two general
	points $p$ and $q$ in $(\Gamma\times_Y\Gamma')(\Omega)$ lying on the same fiber of $\Psi$.
\end{proof}

In view of Proposition~\ref{prop:voisin},
we ask whether
Voisin's Conjecture~\ref{conj:voisin} admits an analogue modulo rational equivalence\,:

\begin{conj}[Motivic surface decomposability for hyper-K\"ahler
	varieties]\label{conj:MSD}
	Let $X$ be a hyper-K\"ahler variety of dimension $2n$. Then $X$ is motivically
	surface decomposable, in the sense of Definition~\ref{D:MSD}.
\end{conj}

The main result of this section provides evidence for Conjecture~\ref{conj:MSD}\,:

\begin{thm}\label{thm:MSD}
	The hyper-K\"ahler varieties \eqref{hilb}, \eqref{moduli}, \eqref{kum}, \eqref{fano} and \eqref{llsvs} are motivically surface decomposable. Moreover, for a fixed hyper-K\"ahler variety as in  \eqref{hilb}, \eqref{moduli}, \eqref{kum}, \eqref{fano} or \eqref{llsvs}, one may choose the surfaces $S_1,\ldots,S_n$ as in Definition~\ref{D:MSD} to be the same.
\end{thm}
\begin{proof}
	Case \eqref{hilb}. Obviously, $S^n$ has a motivic surface decomposition, and we apply Proposition~\ref{prop:stable} to the dominant rational map $f : S^n \dashrightarrow \mathrm{Hilb}^n(S)$ which is the composition of the quotient morphism $S^n \to
	S^{(n)}:=S^n/\mathfrak{S}_n$ with the inverse of the (birational) Hilbert--Chow
	morphism $ \mathrm{Hilb}^n(S) \to S^{(n)}$. \medskip
	
	Case \eqref{moduli}. This reduces to the case~\eqref{hilb} via Proposition~\ref{prop:stable}.
	Indeed, as in the proof of Theorem~\ref{thm:moduli}, 
	we have generically finite and surjective morphisms 
	$$\xymatrix{ R_0
		\ar[d]_{p_{\mathrm{Hilb}^n(S)}} \ar[r]^{p_{\operatorname{M}_\sigma(v)}} &
		\operatorname{M}_\sigma(v) \\
		\mathrm{Hilb}^n(S),
	}$$
	such that all points on the same fiber of $p_{\mathrm{Hilb}^n(S)}$ have same class in $\CH_0(\operatorname{M}_\sigma(v))$.
	\medskip

	Case \eqref{kum}. 	Recall that the $n$-th generalized Kummer variety $K_n(A)$ associated
	to an abelian surface $A$ is a fiber of the isotrivial fibration $\mathrm{Hilb}^{n+1}(A) \to A$
	that is the composite of the Hilbert--Chow morphism $\mathrm{Hilb}^{n+1}(A) \to
	A^{n+1}/\mathfrak{S}_{n+1}$ with the sum morphism $\Sigma :
	A^{n+1}/\mathfrak{S}_{n+1} \to A$. 
	The restriction of the Hilbert--Chow morphism provides
	a
	birational morphism from  $K_n(A)$ to the variety $A_0^{n+1} /
	\mathfrak{S}_{n+1}$, where $A_0^{n+1}$ is the fiber over $0$ of the sum
	morphism
	$\Sigma : A^{n+1} \to A$ and the action of the symmetric group
	$\mathfrak{S}_{n+1}$ is the one induced from the action on $A^{n+1}$ permuting
	the factors.  We thereby obtain a dominant rational map $A^n \dashrightarrow K_n(A)$ and we may conclude with Proposition~\ref{prop:stable}.
	\medskip
	
	Case \eqref{fano}.
	Let $Y$ be a smooth cubic fourfold and let $\mathcal A_Y$ be its Kuznetsov component, \emph{i.e.}, $$\mathcal A_Y := \{\mathcal E\in D^b(Y) \ \big\vert \ \operatorname{Ext}^*_{D^b(Y)} (\mathcal O_Y(i),\mathcal E) \ \mbox{for} \ i=0,1,2	\}.$$
	In other words, we have a semi-orthogonal decomposition $D^b(Y) = \langle \mathcal A_Y, \mathcal{O}_Y, \mathcal{O}_Y(1), \mathcal{O}_Y(2)\rangle$.

	Let $D\subset F(Y)$ be a uniruled divisor over a surface $B$, 
	$$\xymatrix{  D \ar@{^(->}[r]^{j\quad}  \ar@{-->}[d]_q & F(Y). \\ B
	}$$
	Such a divisor is provided for instance by \cite{CMP}\,; explicit examples are also given in \cite{Voisin-Intrinsic} and \cite[Lem.~1.8]{SY}.
	The rational map $q$ induces an isomorphism $q_* : \CH_0(D) \stackrel{\sim}{\longrightarrow} \CH_0(B)$, and for $k>0$ the embedding $j$ induces a morphism $j^{(k)}_* : \CH_0(B^{(k)}) \to \CH_0(F(Y))$, where $B^{(k)}$ denotes the $k$-th symmetric power of $B$.		  
	We consider now a moduli space $M$ of stable objects on $\mathcal{A}_Y$ and denote $2n$ its dimension.
	Following~\cite[\S 3.2]{SY}, we consider the incidence
	$$R:= \{(\mathcal{E},\xi) \in M \times B^{(n)} \ \big\vert \
	c_3(\mathcal{E}) = [P]_*j_*^{(n)}[\xi] +c\, [l_0] \ \mbox{in}\ \CH_1(Y)\}$$ together with the two natural projections $p_{M} : R \to M$ and $p_{B^{(n)}} : R \to B^{(n)}$. Here, $P := \{ (l,y) \in F(Y)\times Y : y\in l\}$ is the cylinder correspondence and $l_0$ is any line on $Y$ with class $\frac{1}{3}[c_1(\mathcal O_Y(1))]^3$.
	On the one hand, by \cite[Prop.~3.4]{SY}, two objects $\mathcal E_1$ and $\mathcal E_2$ satisfy $[\mathcal E_1] = [\mathcal E_2] \ \mbox{in}\ \CH_0(M)$ if and only if $c_3(\mathcal E_1) = c_3(\mathcal E_2) \ \mbox{in}\ \CH_1(Y)$\,; in particular, all points on the same fiber of $p_{B^{(n)}}$ have the same class in 
	$\CH_0(M)$. 	
	On the other hand, by \cite[Prop.~3.5]{SY} (which holds for the base of any uniruled divisor in $F(Y)$), assuming $p_{M} : R \to M$ is dominant, there is a component $R_0\subseteq R$ such that both projections $p_{M} : R \to M$ and $p_{B^{(n)}} : R \to B^{(n)}$ restricted to $R_0$ are dominant and generically finite. Combining both facts above establishes that $M$ is motivically surface decomposable, provided $p_M : R\to M$ is dominant. 
	\medskip
	
	In order to establish \eqref{fano}, it thus suffices to obtain a modular interpretation $M$ of $F(Y)$ and to show that the corresponding map $p_M : R \to M$ is dominant.
	
	First, we recall how the Fano variety $F(Y)$ of lines on $Y$ can be viewed as a moduli of stable objects in $\mathcal A_Y$.
	Let $l$ be a line on $Y$. Denote $\mathcal I_l$ the ideal sheaf of $l$ in~$Y$ and consider the stability condition on coherent sheaves on $Y$ induced by the projective embedding $Y\subset \PP^5$. 
	Following  \cite[\S2.3]{MS}, it was observed in \cite{KuMa} that the stable coherent sheaf $\mathcal F_l:= \ker(\mathcal O_Y^{\oplus 4} \longrightarrow \mathcal I_l(1))$ belongs to $ \mathcal A_Y$.
	Moreover, in \cite[Prop.~5.5]{KuMa}, $F(Y)$ is identified with the connected component of the moduli space of stable sheaves containing the objects $\mathcal F_l$ for any line $l \subset Y$.
	Now define 
	$$P_l := \operatorname{cone}\big(\mathrm{ev}^\vee : \mathcal F_l(-1) \longrightarrow \mathrm{RHom}(\mathcal F_l(-1) , \mathcal{O}_Y(-1))^\vee \otimes \mathcal O_Y(-1)
	\big)[-1].$$
	The object $P_l$ still belongs to $\mathcal A_Y$, and
	in $\operatorname{D}^b(Y)$ we have a distinguished triangle
	\begin{equation}\label{E:disttriangle}
	\xymatrix{\mathcal O_Y(-1)[1] \ar[r] & P_l \ar[r] & \mathcal{I}_l.}
	\end{equation}	
	Moreover, as explained in \cite[\S2.3]{MS}, 
	the Fano variety $F(Y)$ of lines on $Y$ identifies with the moduli space of the objects $P_l \in \mathcal A_Y$.
	By \cite[Thm.~1.1]{PLZ}, the objects $P_l$ are stable (with respect to a Bridgeland stability condition) with Mukai vector $\lambda_1+\lambda_2$ (with $\lambda_1$ and $\lambda_2$ as defined \emph{e.g.}~in \cite[\S 2.2]{PLZ}) and $F(Y)$ identifies with the moduli space $M:= \operatorname{M}_\sigma(\lambda_1+\lambda_2)$. Consequently, the points
	of $M$ are given by the stable objects $P_l \in \mathcal A_Y$ for varying $l \in F(Y)$.

	Second, we conclude by showing that $p_{M} : R \to M$ is dominant. 
	For that purpose, it is sufficient to show that for all lines $l\subset Y$, we have $c_3(P_l) \in S^{SY}_2(Y)$, where $S^{SY}_\bullet(Y)$ is the ascending filtration of Shen--Yin~\cite{SY} on $\CH_1(Y)$ defined by 
	$$S^{SY}_k(Y) := \{ [P]_*j_*^{(k)}[\xi] +\Z \, [l_0] : \xi \in B^{(k)}\} \subseteq \CH_1(Y).$$ 
	(Note that $S^{SY}_\bullet$ does not depend on the choice of uniruled divisor by \cite[Lem.~1.1]{SY}.)
	Due to~\eqref{E:disttriangle}, we have the following identity involving total Chern classes 
	$$c(P_l) = c(\mathcal I_l)\cdot c(\mathcal O_Y(-1))^{-1} \ \mbox{in}\ \CH^*(Y).$$ 
	Since the ideal sheaf $\mathcal I_l$ is supported on $l$, we have $c_i(\mathcal I_l) = 0$ for $i<3$. It follows that $$c_3(P_l) = c_3(\mathcal I_l) + 3[l_0] \ \mbox{in}\ \CH^3(Y).$$ However, by \cite[Thm.~0.4]{SY}, we have $c_3(\mathcal I_l) \in S^{SY}_2(Y)$,  and we conclude that  $c_3(P_l) \in S^{SY}_2(Y)$, as desired.\medskip

	Case \eqref{llsvs}.  Let $Y$ be a smooth cubic fourfold not containing a plane, let $F$ be its Fano variety of lines and let $Z$ be the associated LLSvS eightfold. By considering $\mu : F\times F \dashrightarrow Z$  the dominant
	rational map of degree~6 constructed by
	Voisin~\cite[Prop.~4.8]{Voisin-coisotropic}, we obtain thanks to case \eqref{fano} and Proposition~\ref{prop:stable} a motivic surface decomposition for $Z$.
\end{proof}

\begin{rmk}\label{R:anysurface}
	We note that, \emph{for any} surface $B$ obtained as the (desingularization of the) base of a uniruled divisor on the Fano variety $F(Y)$ of lines on a smooth cubic fourfold $Y$, we obtain a motivic surface decomposition for $F(Y)$ in terms of $B\times B$. In particular, by Proposition~\ref{prop:voisin}, we obtain a surface decomposition in cases~\eqref{fano} and~\eqref{llsvs} for any surface $B$ with vanishing irregularity obtained as the (desingularization of the) base of a uniruled divisor on $F(Y)$.
	This should be compared to \cite[Thm.~3.3(1)]{VoisinTriangle}, where the surface involved in the surface decompositions of~\eqref{fano} and~\eqref{llsvs} is the surface $B = \Sigma_2$ of lines of second type on~$Y$ (whose irregularity vanishes by \cite[Thm.~D]{GK}).
\end{rmk}

\begin{rmk}[Moduli spaces of stable objects on $\mathcal A_Y$] \label{R:modulicubic}
	Let $M$ be a moduli space of stable objects on $\mathcal A_Y$ and let $2n$ be its dimension. The proof of Theorem~\ref{thm:MSD}\eqref{fano} shows that $M$ is motivically surface decomposable if $p_M : R \to M$ is dominant, or equivalently if for any object $\mathcal E \in M$ we have $c_3(\mathcal E) \in S^{SY}_n(Y)$. The latter is precisely \cite[Conj.~0.3]{SY}.
\end{rmk}

As an application of the above, we can complete the list of
\cite[Thm.~3.3]{VoisinTriangle}\,:

\begin{cor}\label{C:SD}
	\begin{enumerate}[(a)]
		\item 		Moduli of stable objects on K3 surfaces are surface decomposable.
		\item Let $M$ be a moduli of stable objects on the Kuznetsov component $\mathcal A_Y$ of a smooth cubic fourfold~$Y$. If $M$ satisfies \cite[Conj.~0.3]{SY}, then $M$ admits a surface decomposition and the surfaces involved can be chosen to be pairwise equal and to be equal to any surface~$B$ with vanishing irregularity obtained as the base of a uniruled divisor on the Fano variety $F(Y)$.
	\end{enumerate}
\end{cor}
\begin{proof}
	Case (a) is the combination of Theorem~\ref{thm:MSD}\eqref{moduli} (where a motivic surface decomposition is obtained in terms of a K3 surface) and Proposition~\ref{prop:voisin}. Case (b) was outlined in Remark~\ref{R:modulicubic} (see also Remark~\ref{R:anysurface}).
\end{proof}

\bibliographystyle{amsalpha}
\bibliography{bib}

\end{document}